\documentclass[12pt]{amsart}
\usepackage[matrix,arrow,curve]{xy}
\usepackage{amssymb,amsmath,euscript}
\theoremstyle{definition}\newtheorem{theorem}{Theorem}[section]
\newtheorem*{theorem*}{theorem}
\newtheorem{proposition}{Proposition}[section]
\newtheorem{lemma}[proposition]{Lemma}
\newtheorem{corollary}[proposition]{Corollary}
\newtheorem{conjecture}[proposition]{Conjecture}
\theoremstyle{definition}\newtheorem{definition}[proposition]{Definition}
\theoremstyle{remark}\newtheorem{example}[proposition]{Example}
\newtheorem{remark}[proposition]{Remark}
\renewcommand{\to}[1][]{\xrightarrow{#1}}
\newcounter{AP}
\begin{document}
\begin{titlepage}
{\Huge ~\\~\\Alexey Petukhov\\~\\~\\~\\~\\Ph.D. Thesis\\~\\~\\~\\A geometric approach to the study of $(\frak g, \frak k)$-modules of finite type}\\ \\~\\~\\~\\~\\~\\~\\~\\~\\~\\~\\~\\~\\~\\\begin{center}{\bf J}acobs {\bf U}niversity {\bf B}remen\end{center}\begin{center}{\bf24}-th May {\bf2011}\end{center}~\\~ \\Committee members\begin{center}Professor Dr. {\bf I}van {\bf P}enkov (Ph. D. Advisor and Chair of Committee)\end{center}~\\Professor Dr. {\bf A}lan {\bf H}uckleberry\hspace{10pt}~\\~\\Professor Dr. {\bf D}mitri {\bf P}anyushev~\\~\\~\\~\\~\\\end{titlepage}
\tableofcontents
\section{Notation and definitions}\label{Snot}
Throughout this thesis $\frak g$ will be a reductive Lie algebra
over the algebraically closed field $\mathbb C$ of characteristic 0
and $\frak k\subset\frak g$ will be a reductive in $\frak g$
subalgebra. Let U$(\frak k)$ be the universal enveloping algebra of
$\frak k$, $G$ be the adjoint group of $[\frak g, \frak g]$, $K$ be
a connected subgroup of $G$ with Lie algebra $\frak k\cap [\frak g,
\frak g]$ and $B$ be a Borel subgroup of $K$. We work in the
category of algebraic varieties over $\mathbb C$. All Lie algebras
considered are finite-dimensional.
\begin{definition}\label{Dgk}A $(\frak g, \frak k)$-{\it module} is a $\frak g$-module for which action of $\frak k$ is {\it locally finite}, i.e. for which dim~(U$(\frak k)m)<\infty$ for any $m\in M$, where U$(\frak k)m:=\{m'\in M\mid m'=um$ for some $u\in$U$(\frak k)$ and $m\in M\}$.\end{definition}
\begin{definition}Let $M$ be a {\it locally finite} $\frak k$-module, i.e a $(\frak k, \frak k)$-module. We say that $M$ {\it has finite type over $\frak k$} if all isotypic components of $\frak k$ are finite-dimensional. We say that a $(\frak g, \frak k)$-module is of {\it finite type} if $M$ has finite type over $\frak k$.\end{definition}
Let Z$(\frak g)$ be the center of the universal enveloping algebra U$(\frak g$).
\begin{definition} We say that a $\frak g$-module $M$ {\it affords a central character} if for some homomorphism of algebras $\chi:$Z$(\frak g)\to\mathbb C$ we have $zm=\chi(z)m$  for all $z\in$Z$(\frak g)$ and $m\in M$.\end{definition}
\begin{definition} We say that a $\frak g$-module $M$ {\it affords a generalized central character} if for some homomorphism \begin{center}$\chi:$Z$(\frak g)\to\mathbb C$\end{center} and some $n\in\mathbb Z_{>0}$ we have \begin{center}$(z-\chi(z))^nm=0$\end{center} for all $m\in M$ and $z\in$Z$(\frak g)$.\end{definition}
Let $M$ be a $(\frak g, \frak k)$-module of finite type and $V$ be a simple $\frak k$-module. We denote by $[M: V]_\frak k$ the supremum over all finite-dimensional $\frak k$-submodules $M'\subset M$ of the Jordan-H\"older multiplicities $[M': V]_\frak k$. By $[M: \cdot]_\frak k$  we denote the corresponding function from the set of simple $\frak k$-modules to $\mathbb Z_{\ge0}$. This is nothing but the $\frak k$-character of $M$.
\begin{definition}\label{Dbm} A {\it bounded $(\frak g, \frak k)$-module} $M$ is a $(\frak g, \frak k)$-module which is {\it bounded} as a $\frak k$-module, i.e. for which the function $[M:\cdot]_\frak k$ is uniformly bounded by some constant $C_M$.\end{definition}
\begin{definition} A {\it multiplicity-free} $(\frak g, \frak k)$-module $M$ is a $(\frak g, \frak k)$-module which is multiplicity-free as a $\frak k$-module, i.e. for which the function $[M:\cdot]_\frak k$ is uniformly bounded by 1.\end{definition}
For a finitely generated $\frak g$-module we denote by V$(M)$ the associated variety of $M$, by Ann$M$ the annihilator of $M$ in U$(\frak g)$, by GV($M$) the zero set of gr(Ann$M$) (see Subsection~\ref{SSavg}). We call a $G$-orbit $Gu$ in $\frak g^*$ {\it nilpotent} if $0\in\overline{Gu}$. If $\frak g$ is semi-simple we identify $\frak g$ and $\frak g^*$.

For a variety $X$ we denote by $\EuScript O(X)$ the structure sheaf of $X$, by $\mathbb C[X]$ the algebra of regular functions on $X$, by $\EuScript D(X)$ the sheaf of differential operators on $X$, by D$(X)$ the algebra of global sections of $\EuScript D(X)$. We denote by T$X$ (respectively, T$^*X$) the total space of the tangent (respectively, cotangent) bundle to $X$. If $\EuScript M$ is a coherent $\EuScript D(X)$-module, $\EuScript V(\EuScript M)\subset$T$^*X$ stands for {\it the singular support} of $\EuScript M$~\cite{Bo} (see also Subsection~\ref{SSd_g}).  All $\EuScript D(X)$-modules considered are quasicoherent.

For a finite-dimensional vector space $W$ we set $n_W:=$dim$W$.
\begin{definition} We call an $n_W$-tuple $\bar\lambda:=(\lambda_1,..., \lambda_{n_W}), \lambda_i\in\mathbb C,$ {\it decreasing} if $\lambda_i-\lambda_j\in\mathbb Z_{\ge0}$ for $i\ge j$. We call $\bar\lambda$ {\it semi-decreasing} if it is not decreasing but becomes decreasing when we remove one coordinate (cf. with O.~Mathieu's~\cite{M} definitions of ordered/semi-ordered tuples).\end{definition}
Assume $n_W\ge 3$. Let $\bar\lambda=(\lambda_1,..., \lambda_{n_W-1}), \lambda_j\in\mathbb C$, be a decreasing tuple and $t\in\mathbb C$. By adding an additional coordinate $t$ to $\bar\lambda$ (at any position) we obtain a semi-decreasing $n_W$-tuple $\bar\lambda^+$.
\begin{definition}{\it The monodromy} m$(\bar\lambda^+)$ of $\bar\lambda^+$ is the number e$^{2\pi i(t-\lambda_1)}$.\end{definition}
\begin{definition} We call an $n_W$-tuple $\bar\lambda:=(\lambda_1, ..., \lambda_{n_W})$, $\lambda_i\in\mathbb C$, {\it integral} if $\lambda_i-\lambda_j\in\mathbb Z$ for any $i, j$. We call $\bar\lambda$ {\it semi-integral} if it is not integral but becomes integral when we remove one term. We call a tuple $\bar\lambda$ {\it regular} if $\lambda_i\ne\lambda_j$ for all pairs $i\ne j$.\end{definition}
Any semi-decreasing $n_W$-tuple $\bar\lambda$ is regular integral, singular integral, or semi-integral. If $\bar\lambda$ is integral then m$(\bar\lambda)=1$; if $\bar\lambda$ is not integral then m$(\bar\lambda)\ne1$.

\begin{definition} We call an $n_W$-tuple a {\it Shale-Weil tuple} if $\mu_i>\mu_j$ for $i>j$, $\mu_{n_W-1}>|\mu_{n_W}|$ and $\mu_i\in\frac{1}{2}+\mathbb Z$ for all $i\in\{1,...,n_W\}$.\end{definition}
For a $K$-variety $X$ and a point $x\in X$ we denote by $K_x$ the stabilizer of $x$ in $K$ and by $\frak k_x$ the Lie algebra of $K_x$. If there exists a subgroup $H\subset K$ such that $H$ is conjugate to $K_x$ for all $x$ from some open subset $\tilde X\subset X$, we call $H$ {\it a generic stabilizer} of $K$ on $X$; we call the Lie algebra of $H$ {\it a generic isotropy subalgebra}. By definition, $X$ is a $K$-{\it spherical} variety if $X$ is irreducible and has an open $B$-orbit.

By SL$_n$, SO$_n$, SP$_n$ we denote respectively the special linear, orthogonal and symplectic groups of $n$-dimensional vector space (SP$_n$ is defined only when $n=2k$). By $\frak{sl}_n, \frak{so}_n, \frak{sp}_n$ we denote the corresponding Lie algebras.
\newpage\section{Introduction and brief statements of results}\label{Stro}
There are two well-known categories of $(\frak g, \frak k)$-modules: the category of Harish-Chandra modules and the category $\EuScript O$. In the first case $\frak k$ is a symmetric subalgebra of $\frak g$ (i.e. $\frak k$ coincides with the fixed points of an involution of $\frak g$), and in the second case $\frak k$ is a Cartan subalgebra $\frak h_\frak g$ of $\frak g$. For both types of pairs $(\frak g, \frak k)$ the $(\frak g, \frak k)$-modules in question are of finite type. I.~Penkov, V.~Serganova, and G.~Zuckerman have proposed to study, and attempt to classify, simple $(\frak g, \frak k)$-modules of finite type for arbitrary reductive in $\frak g$ subalgebras $\frak k$,~\cite{PSZ},~\cite{PZ} (see also~\cite{Mi}).

Let $X$ be the variety of all Borel subalgebras of $\frak g$. Let $\lambda\in$H$^1(X,\Omega^{1,\mathrm{cl}})$ be a cohomology class ($\Omega^{1,\mathrm{cl}}$ is the sheaf of closed holomorphic 1-forms on $X$) and $\EuScript D^\lambda(X)$ be the corresponding sheaf of twisted differential operators on $X$~\cite{Be}. The cohomology class $\lambda$ defines functors of 'localization' (Loc: $\frak g$-modules to $\EuScript D^\lambda(X)$-modules) and 'global sections' (GSec: $\EuScript D^\lambda(X)$-modules to $\frak g$-modules). For any central character $\chi$ there exists a cohomology class $\lambda$ such that GSec(Loc) is the identity functor after restriction to the category of $\frak g$-modules which affords the central character $\chi$~\cite{BeBe}. Any simple $\frak g$-module affords a central character~\cite{Dix}.

We fix $\lambda\in$H$^1(X,\Omega^{1,\mathrm{cl}})$. In the category of $\EuScript D^\lambda(X)$-modules there is a distinguished full subcategory of holonomic sheaves of modules, or simply holonomic modules. Informally, holonomic modules are $\EuScript D^\lambda(X)$-modules of minimal growth (see Definition~\ref{Dhol}). The simple holonomic modules are in one-to-one correspondence with the set of pairs $(L ,S)$, where $L$ is an irreducible closed subvariety of $X$ and $S$ is sheaf of $\EuScript D^\lambda(L')$-modules which is $\EuScript O(L')$-coherent after restriction to a suitable open subset $L'\subset L$. Moreover, a coherent holonomic module $S$ is locally free on $L'$, and one could think about it as a vector bundle $S_B$ over $L'$ with a flat connection. Note that flat local sections of $S_B$ are not necessarily algebraic.

Our first result is the following theorem, which we prove in Section~\ref{SHol} (Corollary~\ref{Cholsup}).
\begin{theorem}\label{THol} Let $M$ be a simple $(\frak g, \frak k)$-module of finite type. Then Loc$(M)$ is a holonomic $\EuScript D^\lambda(X)$-module.\end{theorem}
The following theorem is a 'geometric twin' of the previous one and is proved also in Section~\ref{SHol}. The definitions needed for the statement of the theorem see in Subsection~\ref{Sssg} and Subsection~\ref{SSHM}, or in~\cite{VP}.
\begin{theorem}\label{TGeo} Let $\EuScript Z\subset\frak g^*$ be a nilpotent $G$-orbit, $\frak k^\bot$ be the annihilator of $\frak k$ in $\frak g^*$, N$_K(\frak g^*)$ be the $K$-null-cone in $\frak g^*$. Then the irreducible components of $\EuScript Z\cap\frak k^\bot\cap$N$_K(\frak g^*$) are isotropic subvarieties of $\EuScript Z$.\end{theorem}
Let V$_{\frak g, \frak k}^\cdot$ be the set of all irreducible components of all possible intersections of N$_K(\frak k^\bot)$ with $G$-orbits in N$_G(\frak g^*)$. This finite set of subvarieties of $\frak g^*$ determines a finite set $\EuScript V_{\frak g, \frak k}^\cdot$ of subvarieties of T$^*X$ and a finite set L$_{\frak g, \frak k}^\cdot$ of subvarieties of $X$, see Section~\ref{SHol}.
\begin{theorem}\label{TSup} Let $M$ be a finitely generated $(\frak g, \frak k)$-module of finite type which affords a central character, and let $(L, S)$ the pair corresponding to Loc$M$ as introduced above. Then $L$ is an element of L$_{\frak g, \frak k}^\cdot$.\end{theorem}
We prove this theorem in Section~\ref{SHol} (Corollary~\ref{Cholsup}). Any simple Harish-Chandra module is holonomic~\cite{BeBe} and of finite-type~\cite{HC}\footnote{In the literature on Harish-Chandra modules the condition of being of finite type is often called admissibility.}. For the literature on Harish-Chandra modules, see~\cite{KV} and references therein.

In the classification of simple $(\frak g, \frak h_\frak g)$-modules of finite type the bounded simple modules play a crucial role. Based on this, and on the experience with Harish-Chandra modules, I.~Penkov and V.~Serganova have proposed to study bounded $(\frak g, \frak k)$-modules for general reductive subalgebras $\frak k$. A question arising in this context is, given $\frak g$, to describe all reductive in $\frak g$ bounded subalgebras, i.e. reductive in $\frak g$ subalgebras $\frak k$ for which at least one infinite-dimensional simple bounded $(\frak g, \frak k)$-module exists. In~\cite{PS} I.~Penkov and V.~Serganova give a partial answer to this problem, and in particular proved an important inequality which restricts
severely the class of possible $\frak k$. They also give the complete list of bounded reductive subalgebras of $\frak g =\frak{sl}_n$ which are maximal subalgebras.

In the work~\cite{Ya1} we prove the following theorem.
\begin{theorem}\label{Tbgk}There exists an infinite-dimensional simple bounded $(\frak{sl}(V), \frak k)$-module if and only if $V$ is a $K \times\mathbb C^*$-spherical variety.\end{theorem}
We reproduce the proof of this theorem in Section~\ref{Sbgk}. There are well-known pairs of algebras which admit an infinite-dimensional simple bounded module:\\
(1) $\frak k$ is a Cartan subalgebra of a simple Lie algebra $\frak g$ of type A~or~C~\cite{F};\\
(2) $\frak k$ is a symmetric subalgebra of a reductive Lie algebra $\frak g$.\\
Bounded $(\frak g, \frak h_\frak g)$-modules are nothing but weight modules with bounded weight multiplicities, and all such simple modules have been classified by O.~Mathieu~\cite{M}.

As far as we know, the simple bounded Harish-Chandra modules have not been singled explicitly out within the category of all Harish-Chandra modules.
For a given symmetric pair $(\frak g, \frak k)$, Harish-Chandra modules admit only finitely many support varieties, and a Harish-Chandra module $M$ is spherical if and only if any irreducible component of the support variety V$(M)$ is spherical. There is some progress in singling out the spherical varieties among the support varieties of Harish-Chandra modules~\cite{Pan2},~\cite{Ki}.

Let $I$ be a two-sided ideal in U($\frak g)$. By definition, $I$ is {\it primitive} if $I$ is the annihilator of some simple $\frak g$-module $M$. Let $W$ be a finite-dimensional $\mathbb C$-vector space. To any $n_W$-tuple $\bar\lambda:=(\lambda_1, ..., \lambda_{n_W})$ one assigns a weight $\lambda$ (see Subsection~\ref{SM}) and a primitive ideal I($\lambda)\subset$U$(\frak{sl}(W))$ (this is the annihilator of the simple $\frak{sl}(W)$-module with highest weight $\lambda$), and any primitive ideal of U($\frak{sl}(W))$ arises in this way from some tuple $\bar\lambda$.

Let $M_1, M_2$ be simple $(\frak g, \frak k)$-modules. Let Ann$M_1$, Ann$M_2$ be the annihilators in U$(\frak g)$ of $M_1$ and $M_2$ respectively. I.~Penkov and V.~Serganova \cite{PS} proved that, if Ann$M_1$=Ann$M_2$ and $M_1$ is a bounded $(\frak g, \frak k)$-module, then $M_2$ is $\frak k$-bounded too (see also Theorem~\ref{Trel}). Moreover, a $(\frak g, \frak k)$-module is bounded if and only if the algebra U$(\frak g)/$Ann$M$ satisfies certain relations. In Section~\ref{Sbgk} we prove the following weaker geometric version of this result.
\begin{theorem}\label{Tcob} A simple $(\frak{sl}(W), \frak k)$-module $M$ is bounded if and only if the associated variety GV($M$) is $K$-coisotropic.\end{theorem}
This shows that 'boundedness' is not a property of a module $M$ but of the ideal Ann$M$, and moreover of the nilpotent orbit GV($M)\subset\frak{sl}(W)^*$. Theorem~\ref{Tcob} also motivates our interest in the classification of $K$-coisotropic nilpotent $G$-orbits in $\frak g^*$. In the case $\frak g=\frak{sl}(W)$, the set of $K$-coisotropic nilpotent orbits in $\frak{sl}(W)^*$ is naturally identified with the set of partition equivalence classes of $K$-spherical partial $W$-flag varieties (see Subsection~\ref{SGr}). We work out the classification of $K$-spherical flag varieties in Section~\ref{Ssphgr}.

Furthermore, we make the following conjecture.
\begin{conjecture} A simple $(\frak g, \frak k)$-module $M$ is bounded if and only if the associated variety GV($M$) is $K$-coisotropic.\end{conjecture}
This statement is closely related with~\cite[Question, p.~191]{Pan} and we believe that this question will be answered soon~\cite{ZT}.

In the rest of the thesis we consider in greater detail four special pairs $(\frak g, \frak k)$:
$$\begin{tabular}{ll}a) $W=$S$^2V$&b) $W=\Lambda^2V$\\\hline1a) $(\frak{sl}(W), \frak{sl}(V))$&1b) $(\frak{sl}(W), \frak{sl}(V))$\\ 2a) $(\frak{sp}(W\oplus W^*), \frak{gl}(V))$&2b) $(\frak{sp}(W\oplus W^*), \frak{gl}(V))$\\\hline\end{tabular}\eqno(1).$$
We hope that our results about these cases shed some light on how the general theory of bounded modules looks like.

In the rest of Section~\ref{Stro} $V$ is a finite-dimensional vector space and $W$=S$^2V$ (in this case $n_V\ge 3$) or $W=\Lambda^2V$ (in this case $n_V=2k$ and $n_V\ge 5$).

For a simple $\frak{sl}(W)$-module $M$ we have dim~V$(M)\ge n_W-1$ or $M$ is finite-dimensional. A simple $\frak{sl}(W)$-module $M$ is {\it of small growth} if dim~V$(M)\le n_W-1$ (the definition of a module of small growth which is not necessarily simple is given in Section~\ref{Ssmg}). The following theorem is proved in Subsection~\ref{SSdsl}.
\begin{theorem}\label{Tsmsl}Any simple bounded $(\frak{sl}(W), \frak{sl}(V))$-module is of small growth.\end{theorem}
We note that all bounded weight modules are also of small growth. The following result is inspired by the corresponding result of O.~Mathieu for weight modules. The proof of this result is presented in Subsection~\ref{SSdsl}.
\begin{theorem}\label{Tannsl} Let $M$ be a simple bounded $(\frak{sl}(W), \frak{sl}(V))$-module and $\bar\lambda$ be an $n_W$-tuple such that Ann$M$=I$(\lambda)$. Then $\bar\lambda$ is semi-decreasing.\end{theorem}
\begin{theorem}\label{Tasl}Let $M$ be a finitely generated $(\frak{sl}(W), \frak{sl}(V))$-module and suppose I$(\lambda)\subset$Ann$M$ for some semi-decreasing tuple $\bar\lambda$. Then $M$ is $\frak{sl}(V)$-bounded.\end{theorem}
We prove Theorem ~\ref{Tasl} in Section~\ref{SSdsl}. The primitive ideals which correspond to the semi-decreasing sequences are called Joseph ideals (see Section~\ref{SJId}).

Fix $t\in\mathbb C$. Let P$_{\mathrm e^{2\pi it}}(W)$ be the cardinality of the set of isomorphism classes of simple perverse sheaves on $W$ (with respect to the stratification by GL$(V)$-orbits) which have fixed monodromy $e^{2\pi it}$ and are neither supported at 0 nor are smooth on $W$.
The above simple perverse sheaves on $W$ are described, following~\cite{BR}, in quiver terms in the Appendix. In particular, there are only finitely many isomorphism classes of simple perverse sheaves with a given monodromy. The following theorem 'counts' simple bounded $(\frak{sl}(W), \frak{sl}(V))$-modules. We prove this theorem in Subsection~\ref{SSdsl}.
\begin{theorem}\label{Tgkps}Let $\bar\lambda$ be a semi-decreasing tuple. Then there exist precisely P$_{\mathrm m(\lambda)}(W)$ non-isomorphic infinite-dimensional simple $(\frak{sl}(W), \frak{sl}(V))$-modules annihilated by I$(\lambda)$.\end{theorem}
Let $\bar\lambda$ be a decreasing tuple. Let $\EuScript J_\lambda$ be the set of infinite-dimensional simple bounded $(\frak{sl}(W), \frak{sl}(V))$-modules annihilated by Ker$\chi_\lambda$. Let $\langle\EuScript J_\lambda\rangle$ be the free vector space generated by $\EuScript J_\lambda$. The set of $n_W$-tuples carries a natural action of S$_{n_W}$. This action induces an action of S$_{n_W}$ on $\langle\EuScript J_\lambda\rangle$. The S$_{n_W}$-module $\langle\EuScript J_\lambda\rangle$ is isomorphic to a direct sum of P$_1(W)$-copies of an $(n_W-1)$-dimensional module of S$_{n_W}$.

We now turn our attention to row 2 of table (1). The space $W\oplus W^*$ has a natural symplectic form \begin{center}$\omega(\cdot, \cdot):((x_1, l_1), (x_2, l_2))\mapsto x_1(l_2)-x_2(l_1)\in\mathbb C$.\end{center} There is an inclusion $\frak{gl}(W)\subset\frak{sp}(W\oplus W^*)$. The following theorem relates bounded $(\frak{sl}(W), \frak{sl}(V))$-modules and bounded $(\frak{sp}(W\oplus W^*), \frak{gl}(V))$-modules. We prove it in Subsection~\ref{SSdsl}.
\begin{theorem}\label{Tsptosl}a) Let $\tilde M$ be a bounded $(\frak{sp}(W\oplus W^*), \frak{gl}(V))$-module. Then any simple $\frak{sl}(W)$-subquotient of $\tilde M$ is an $(\frak{sl}(W), \frak{sl}(V))$-bounded module.\\b)Let $M$ be a simple bounded $(\frak{sl}(W), \frak{sl}(V))$-module. Then there exists a simple bounded $(\frak{sp}(W\oplus W^*), \frak{gl}(V))$-module $\tilde M$ such that $M$ is an $\frak{sl}(W)$-subquotient of $\tilde M$ .\end{theorem}
Any tuple $\bar\mu=(\mu_1,..., \mu_{n_W})$ determines a two-sided ideal I$_\frak{sp}(\bar\mu)$ of U$(\frak{sp}(W\oplus W^*))$ and any primitive ideal of U$(\frak{sp}(W\oplus W^*))$ is determined by some tuple $\bar\mu$.

For a simple $\frak{sp}(W\oplus W^*)$-module $M$, we have either dim~V$(M)\ge n_W$ or dim$M<\infty$. A simple $\frak{sp}(W\oplus W^*)$-module $M$ {\it is of small growth} if dim~V$(M)\le n_W$ (the definition of a module of small growth which is not necessarily simple is given in Section~\ref{Ssmg}). The following theorems are '$\frak{sp}$-twins' of Theorem~\ref{Tsmsl}, Theorem~\ref{Tannsl}, Theorem~\ref{Tasl}.
\begin{theorem}\label{Tsmsp}Any simple bounded $(\frak{sp}(W\oplus W^*), \frak{gl}(V))$-module is of small growth.\end{theorem}
\begin{theorem}\label{Tannsp} Let $M$ be an infinite-dimensional simple bounded $(\frak{sp}(W\oplus W^*), \frak{gl}(V))$-module and $\bar\mu$ be
an $n_W$-tuple such that Ann$M=$I$_\frak{sp}(\mu)$. Then $\bar\mu$ is a Shale-Weil tuple.\end{theorem}
\begin{theorem}\label{Tasp}Let $M$ be a finitely generated $(\frak{sp}(W\oplus W^*), \frak{gl}(V))$-module and I$_\frak{sp}(\mu)\subset$Ann$M$ for some Shale-Weil tuple $\bar\mu$. Then $M$ is bounded.\end{theorem}
We prove these three theorems in Subsection~\ref{Sspgl}. The primitive ideals which correspond to the Shale-Weil sequences are called Joseph ideals (see Section~\ref{SJId}).

We call a Shale-Weil $n_W$-tuple $\bar\mu$ {\it positive} if $\mu_{n_W}>0$ and {\it negative} otherwise. Let $\frak h_W$ be the $\mathbb C$-affine space of $n_W$-tuples. The automorphism \begin{center}$\sigma:\frak h_W^*\to\frak h_W^*, (\mu_1, \mu_2,...,\hspace{10pt} \mu_{n_W})\mapsto (\mu_1. \mu_2, ..., -\mu_{n_W})$\end{center} interchanges the sets of positive and negative Shale-Weil tuples. Furthermore, I$_\frak{sp}(\bar\mu)=$I$_\frak{sp}(\sigma\bar\mu)$. Set $\bar\mu_0:=(n_W-\frac{1}{2}, n_W-\frac{3}{2},..., \frac{1}{2})$. The following theorem shows that the categories of bounded modules annihilated by different ideals are equivalent. We prove it in Subsection~\ref{Sspgl}.
\begin{theorem}\label{Tspeq}a) For any two positive Shale-Weil tuples $\bar\mu_1, \bar\mu_2$ the categories of bounded $(\frak{sp}(W\oplus W^*), \frak{gl}(V))$-modules annihilated by I$_\frak{sp}(\bar\mu_1)$ and I$_\frak{sp}(\bar\mu_2)$ are equivalent.\\b) In particular, the set of simple bounded $(\frak{sp}(W\oplus W^*), \frak{gl}(V))$-modules is naturally identified with the set of pairs $(\bar\mu, M)$, where $\bar\mu$ is a positive Shale-Weil $n_W$-tuple and $M$ is a simple bounded $(\frak{sp}(W\oplus W^*), \frak{gl}(V))$-module annihilated by I$_\frak{sp}(\bar\mu_0)$.\end{theorem}
Let $\bar\mu$ be a positive Shale-Weil $n_W$-tuple. The functor $\EuScript H_{\mu}^{\sigma\mu}$ is exact and involutive. The category of modules annihilated by I$_\frak{sp}(\bar\mu)$ is stable under $\EuScript H_\mu^{\sigma\mu}$. We denote $\EuScript H_{\mu_0}^{\sigma\mu_0}$ by Inv.

The Dynkin diagram of Spin$_{2n_W}$ has a nontrivial involution (it is unique unless $n=4$) and it induces an involution $\sigma$ on Spin$_{2n_W}$. Furthermore there is a unique central element $z\in$Spin$_{2n_W}$ such that $z^2=1$ and Spin$_{2n_W}/\{1,z\}=$SO$_{2n_W}$ (for $n=4, z$
is uniquely determined by the additional requirement $\sigma(z)=z$). Ann {\it odd pair} of simple Spin$_{2n_W}$-modules is a pair $\{L^{\bar\mu}, L^{\sigma\bar\mu}\}$ of simple Spin$_{2n_W}$-modules which are conjugate by $\sigma$ and such that $z$ acts by -1 on $L^{\bar\mu}$ and $L^{\sigma\bar\mu}$. Obviously, $L^{\bar\mu}$ and $L^{\sigma\bar\mu}$ has the same dimension. Any positive Shale-Weil $n_W$-tuple $\bar\mu$ determines an odd pair $\{L^{\bar\mu}, L^{\sigma\bar\mu}\}$ of Spin$_{2n_W}$-modules (cf.~\cite{M}).

Let $\{L, L^\sigma\}$ be the unique odd pair of minimal dimension. The following theorem should be understood as a mnemonic rule corresponding to the results of Subsection~\ref{Sspgl}. It is related to the $\frak{gl}(V)$-characters of bounded $(\frak{sp}(W\oplus W^*), \frak{gl}(V))$-modules.
\begin{theorem} Let $(\bar\mu, M)$ be a pair as in Theorem~\ref{Tspeq} and let $\bar\mu M$ be the corresponding simple $(\frak{sp}(W\oplus W^*), \frak{gl}(V))$-module. Then$$\frac{[\bar\mu M:\cdot]_\frak k+[\bar\mu\mathrm{Inv}M:\cdot]_\frak k}{[M:\cdot]_\frak k+[\mathrm{Inv}M:\cdot]_\frak k}=\frac{[L^{\bar\mu}:\cdot]_\frak h+[L^{\sigma\bar\mu}:\cdot]_\frak h}{[L:\cdot]_\frak h+[L^\sigma:\cdot]_\frak h}\eqno(2.1)$$and$$\frac{[\bar\mu M:\cdot]_\frak k-[\bar\mu\mathrm{Inv}M:\cdot]_\frak k}{[M:\cdot]_\frak k-[\mathrm{Inv}M:\cdot]_\frak k}=\frac{[L^{\bar\mu}:\cdot]_\frak h-[L^{\sigma\bar\mu}:\cdot]_\frak h}{[L:\cdot]_\frak h-[L^\sigma:\cdot]_\frak h}\eqno(2.2)$$\end{theorem}
The functions $[L:\cdot]_{\frak h}, [L^\sigma:\cdot]_{\frak h},
[L^{\bar\mu}:\cdot]_{\frak h}, [L^{\sigma\bar\mu}:\cdot]_{\frak h}$
are known since the modules\begin{center}$L, L^\sigma, L^{\bar\mu},
L^{\sigma\bar\mu}$\end{center}are finite-dimensional. Therefore the
formulas (2.1), (2.2) provide two linear equations for the four
unknowns.

Let $M_1, M_2$ be non-isomorphic simple $(\frak{sp}(W\oplus W^*), \frak{gl}(V))$-modules annihilated by I$_\frak{sp}(\mu_0)$. Then both $M_1, M_2$ are multiplicity-free $\frak{gl}(V)$-modules and the functions \begin{center}$[M_1:\cdot]_{\frak{gl}(V)}$ and $[M_2:\cdot]_{\frak{gl}(V)}$\end{center} are pairwise disjoint, i.e. their product is the zero-function. Moreover, there are infinitely-many non-isomorphic simple $(\frak{sp}(W\oplus W^*), \frak{gl}(V))$-modules annihilated by I$_\frak{sp}(\mu_0)$. We relate the category of bounded \\$(\frak{sp}(W\oplus W^*), \frak{gl}(V))$-modules with a suitable category of perverse sheaves in Subsection~\ref{Sspgl}.
\begin{theorem}\label{Tspps} Let $\bar\mu$ be a positive Shale-Weil $n_W$-tuple. Then the category of $(\frak{sp}(W\oplus W^*), \frak{gl}(V))$-modules annihilated by I$_\frak{sp}(\bar\mu)$ is equivalent to the direct sum of two copies of the category of perverse sheaves on $W$ with respect to the stratification by GL$(V)$-orbits.\end{theorem}

For both cases $W=\Lambda^2V,$ S$^2V$, the category of perverse sheaves on $W$ with respect to the stratification by GL$(V)$-orbits is equivalent to a category of representations of an explicitly described quiver with relations~\cite{BG}, see also the Appendix. The simple objects of this category are enumerated by pairs $(S, Y)$, where $S$ is a GL$(V)$-orbit in $W$ and $Y$ is a simple GL$(V)$-equivariant local system on $S$, i.e. a simple representation of the fundamental group $\pi_1(S)$. As $S$ has to be GL$(V)$-spherical, this group has to be finitely generated and abelian. The category contains infinitely many non-isomorphic simple objects, but only finitely many for any fixed monodromy.

Similar work for the category of bounded weight modules has been done by D.~Grantcharov and V.~Serganova~\cite{GrS1},~\cite{GrS2}: they have found a quiver with relations, whose category of modules is equivalent to the category of bounded weight modules (see also~\cite{SM} and~\cite{GVM}).
\newpage\section{Preliminaries}\label{Spre}
\subsection{Symplectic geometry}\label{Sssg} By $\EuScript TX$ we denote the tangent bundle  of a smooth variety $X$, by T$_xX$ the tangent space to $X$ at a point $x\in X$; by T$^*_xX$ the dual to T$_xX$ space. For a smooth $G$-variety $X$ we denote by $\tau_X: \frak g\to\EuScript TX$ the canonical homomorphism. Let $Y$ be a smooth subvariety of $X$. We denote by N$_{Y/X}$ and N$_{Y/X}^*$ the total spaces of the normal and conormal bundles to $Y$ in $X$ respectively.
\begin{definition}\upshape Suppose that $X$ is a smooth variety which admits a closed nondegenerate two-form $\omega$. Such a pair $(X,\omega)$ is called a {\it symplectic variety}. If $X$ is a $G$-variety and $\omega$ is $G$-invariant, $(X, \omega)$ is called a {\it symplectic $G$-variety}.\end{definition}
\begin{example}\upshape Let $X$ be a smooth $G$-variety. Then $\mathrm T^*X$ has a one-form $\alpha_X$ defined at a point $(l,x) (l\in\mathrm T_x^*X)$ by the equality \begin{center}$\alpha_X(\xi)=l(\pi_*\xi)$\end{center} for any $\xi\in\mathrm T_{(l, x)}(\mathrm T^*X)$, where $\pi: \mathrm T^*X\to X$ is the projection. The differential $\mathrm d\alpha_X$ is a nondegenerate $G$-invariant two-form on $\mathrm T^*X$ and therefore $(\mathrm T^*X, \mathrm d\alpha_X)$ is a symplectic $G$-variety.\end{example}
\begin{example}\upshape Let $\EuScript O$ be a $G$-orbit in $\frak g^*$. Then $\EuScript O$ has a {\it{}Kostant-Kirillov-Souriau two-form } $\omega(\cdot~\!, \cdot)$ defined at a point $x\in\frak g^*$ by the equality\begin{center}$\omega_x(\tau_{\frak g^*}p|_x, \tau_{\frak g^*}q|_x)=x([p,q])$\end{center} for $p, q\in\frak g$.\end{example}
\begin{definition}\upshape Let $(X,\omega)$ be a symplectic variety. We call a subvariety $Y\subset X$\\a) {\it isotropic} if $\omega|_{\mathrm T_yY}=0$ for the generic point $y\in Y$;\\b) {\it coisotropic} if $\omega|_{(\mathrm T_yY)^{\bot_\omega}}=0$ for the generic point $y\in Y$;\\c) {\it Lagrangian} if T$_yY=($T$_yY)^{\bot_\omega}$ for the generic point $y\in Y$ or equivalently if it is both isotropic and coisotropic.\end{definition}
\begin{example}\upshape \label{TZX} Let $X$ be a smooth $G$-variety and $Y\subset X$ be a smooth $G$-subvariety. Then N$^*_{Y/X}$ is Lagrangian in T$^*X$.\end{example}
\begin{proposition}[{see for example~\cite[Lemma 1.3.27]{NG}}]\label{NG} Any closed irreducible conical (i.e. $\mathbb C^*$-stable) Lagrangian $G$-subvariety of $\mathrm T^*X$ is the closure of the total space of the conormal bundle $\mathrm N^*_{Y/X}$ to a $G$-subvariety $Y\subset X$.\end{proposition}
\begin{definition}Let $(X, \omega)$ be a $K$-symplectic variety. We say that $(X, \omega)$ is $K$-{\it isotropic} if any $K$-orbit of some open $K$-stable set $\tilde X\subset X$ is isotropic. We say that $(X, \omega)$ is $K$-{\it coisotropic} if any $K$-orbit of some open $K$-stable subset $\tilde X\subset X$ is coisotropic.\end{definition}
Let $X$ be a $G$-variety. The map\begin{center} T$^*X\times\frak g\to\mathbb C$\hspace{40pt}$((x, l), g)\mapsto l(\tau_X g|_x),$\end{center} where $g\in \frak g, x\in X, l\in$T$^*_xX$, induces a map \begin{center}$\phi_X: $T$^*X\to\frak g^*$\end{center} called the {\it{}moment map}. This map provides the following description of a nilpotent orbit $ Gu\subset\frak g^*$. Suppose that $P$ is a parabolic subgroup of $G$.
\begin{theorem}[R.~W.~Richardson~\cite{Rich}]\label{MMapSl}  The moment map \begin{center}$\phi_{G/P}:\mathrm T^*(G/P)\to\frak g^*$\end{center} is a proper morphism to the closure of some nilpotent orbit $\overline{Gu}$ and is a finite morphism over $Gu$.\end{theorem}
\begin{example}\upshape\label{Pn}Let $G=\mathrm{SL}_n$ and $G/P=\mathbb P(\mathbb C^n)$. Then $\phi_{G/P}(\mathrm T^*(G/P))$ (considered as a subset of $\frak{sl}_n$) coincides with the set of nilpotent matrices of rank $\le$1. The open SL$_n$-orbit in $\phi_{G/P}(\mathrm T^*(G/P))$ is a set of $\mathrm{SL}_n$-highest weight vectors and is contained in the closure of any nonzero nilpotent $\mathrm{SL}_n$-orbit in $\frak{sl}_n^*$.\end{example}
For $G=$SL$_n$, each moment map $\phi_{G/P}$ corresponding to a parabolic subgroup $P$ is a birational isomorphism of $\mathrm T^*(G/P)$ with the image of $\phi_{G/P}$, and one can obtain the closure of any nilpotent orbit $\overline{Gu}$ as the image of a suitable moment map $\phi_{G/P}$ (see~\cite{CM} and references therein). Assume $X=G/P$. The following computation shows that $\phi_X^*(\omega_{Gu})=$d$\alpha_{\mathrm T^*X}|_{\phi^{-1}_XGu}$:\begin{center}d$\alpha_X|_{(x, t)}(\tau_Xp, \tau_Xq)=(\tau_Xp\cdot\alpha_X(\tau_Xq))|_{(x, t)}-(\tau_Xq\cdot\alpha_X(\tau_Xp))|_{(x, t)}-\alpha_X([\tau_Xp, \tau_Xq])|_{(x, t)}=$\\$(\tau_Xq\cdot\phi^*_Xq)|_{(x, t)}-(\tau_Xq\cdot\phi^*_Xp)-(\phi^*_X[p,q])|_{(x, t)})=$\\$(\phi^*_X[p,q])|_{(x, t)}-(\phi^*_X[q, p])|_{(x, t)}-(\phi^*_X[p,q])|_{(x, t)}=x([p,q])=\omega_x([p,q])$\end{center} for any $p, q\in\frak g$ and $(x, t)\in$T$^*X\subset \frak g^*\times X$.

In what follows we call quotients of SL$(W)$ by parabolic subgroups 'partial $W$-flag varieties' (see Subsection~\ref{SGr}).
\subsection{Spherical varieties}\label{SSph} Let $V$ be a finite-dimensional $K$-module. We recall that $B$ is a Borel subgroup of $K$.
\begin{definition} Let $X$ be an irreducible $K$-variety. Then $X$ is called {\it $K$-spherical} if and only if there is an open
orbit of $B$ on $X$. A $K$-module $W$ is called $K$-spherical if it is $K$-spherical as a $K$-variety.\end{definition}
It is well known that $K$-spherical varieties have many beautiful properties. In particular, the number of $K$-orbits on such varieties is finite~\cite{VK}. A subgroup $K'\subset K$ is called {\it spherical} if the quotient $K/K'$ is $K$-spherical. For example, any symmetric subgroup $K'\subset K$ is spherical.

The following lemma is a reformulation in the terms of the present thesis of a result of \'E.~Vinberg and B.~Kimelfeld~\cite[Thm.~2]{VK}.
\begin{lemma}An irreducible quasiaffine algebraic $K$-variety $X$ is $K$-spherical if and only if the space of regular functions $\mathbb C[X]$ is a bounded $\frak k$-module. Moreover, if $X$ is $K$-spherical, then $\mathbb C[X]$ is a multiplicity-free $\frak k$-module.\end{lemma}
\begin{theorem}[D.~Panyushev {~\cite[Thm 2.1]{Pan}}]\label{Pan}Let $X$ be a smooth irreducible $K$-variety and $M$ a smooth locally closed $K$-stable subvariety. Then the generic stabilizers of the actions of $B$ on $X,\mathrm N_{M/X}$ and $\mathrm N_{M/X}^*$ are isomorphic.\label{MPan}\end{theorem}
Let $X$ be an irreducible $K$-variety. Then there exist open subset $\tilde X\subset X$ such that the stabilizers $B_x$ and $B_y$ are conjugate for all $x,y\in \tilde X$~\cite{Pan}. For all $x\in\tilde X$ there exists a unique connected reductive subgroup L$(x)\subset K_x$ such that $B_x$ is a Borel subgroup of L$(x)$~\cite{Pan} (see also~\cite{Gr}).

Let Gr$(r; V)$ be the variety of $r$-dimensional subspaces of $V$. For $r\in\{1,...,n_V-1\}$ and $x\in$Gr$(r; V)$ we denote by V$^r(x)\subset V$ the $r$-dimensional subspace which corresponds to $x$. We apply the construction of~\cite{Pan} to the case $X:=$Gr$(r; V)$. Then L$(x)$ stabilizes V$^r(x)$. Therefore the datum $(K,V,r)$ determines the modules (L$(x), $V$^r(x)$). The type of (L$(x), $V$^r(x))$ does not depend on a point $x\in\tilde X$ and therefore the datum $(K, V, r)$ determines the pair ($L, V^r$). One can compute the subgroup $L$ via a technique of doubled actions~\cite{Pan}.
\begin{definition}We denote by c$_K(X)$ the codimension in $X$ of the generic orbit of $B$.\end{definition}
\begin{remark}The variety $X$ is $K$-spherical if and only if c$_K(X)=0$.\end{remark}
\begin{lemma}\label{Lgtw}Let $P_1, P_2$ be parabolic subgroups of $G$ and $L_1$, $L_2$ be Levi subgroups of $P_1$ and $P_2$ respectively. Then \begin{center}c$_G(G/P_1\times G/P_2)=c_{L_1}(G/P_2)=c_{L_2}(G/P_1)$.\end{center}\end{lemma}
\begin{proof}Let $\tilde B$ be a Borel subgroup of $G$. The generic stabilizer for the action of $\tilde B$ on $G/L_1$ is a Borel subgroup of $L_1$. Therefore the generic stabilizer for the action of a Borel subgroup of $L_1$ on $G/P_2$ coincides with the generic stabilizer for the action of $\tilde B$ on $G/P_1\times G/P_2$. We denote this stabilizer $S$. Let $r$ be the rank of $G$. Then \begin{center}c$_{L_1}(G/P_2)=$dim$G/P_2-(\frac{\mathrm{dim}L_1+r}{2}-$dim$S$)=\\dim$(G/P_1\times G/P_2)-(\frac{\mathrm{dim}G+r}{2}-$dim$S$)=c$_G(G/P_1\times G/P_2)$.\end{center}\end{proof}
\subsection{Grassmannians}\label{SGr} Let $V$ be a finite-dimensional vector space. The set of flags $V_1\subset...\subset V_s\subset V$ with fixed dimensions $(n_1,..., n_s)$ is a homogeneous space of the group GL$(V)$, and we denote this variety by Fl$(n_1,..., n_s; V)$. We call such a variety {\it partial flag variety}. The varieties $\mathbb P(V)$ and Fl$(1; V)$ are naturally identified.

For any $r\in\{1, ..., n_V-1\}$ we denote Fl$(r; V)$ by Gr$(r; V)$.
For $r=1$, the variety Gr$(r; V)$ coincides with $\mathbb P(V)$. We
have\begin{center}Gr$(r; V)\cong$Gr$(n_V-r; V^*)$.\end{center} Let
$n_1, ..., n_s\in\{1, ..., n_V-1\}$ be numbers such that
\begin{center}$n_1<...< n_s$\end{center} and Fl$(n_1, ..., n_s; V)$
be the corresponding partial flag variety. Let $P(x)$ be the
stabilizer of a point $x\in$Fl$(n_1, ..., n_s; V)$ in SL$(V)$ and
$\frak n(x)$ be a nilpotent radical of the Lie algebra of $P(x)$.
Then $\frak n(x)\subset\frak{sl}(V)$ consists of nilpotent elements,
and \begin{center}$\cup_{x\in\mathrm{Fl}(n_1, ..., n_s; V)}\frak
n(x)\subset\frak{sl}(V)^*$\end{center} coincides with the image of
the moment map \begin{center}$\phi:$T$^*$Fl$(n_1,..., n_s;
V)\to\frak{sl}(V)^*$.\end{center}Moreover, the image of $\phi$ in
$\frak{sl}(V)^*$ is the closure of a unique nilpotent SL$(V)$-orbit.
In this way, to any partial flag variety one assigns a unique
nilpotent orbit. The condition 'to be in closure of' on the set of
nilpotent orbits induces a partial order on the set of partial flag
varieties. We will make essential use of this partial order. In what
follows we refer to one partial flag varietiy as being higher or
lower than another in terms of this partial order. Partial flag
varieties which are equivalent in terms of this partial order are
called {\it cotangent-equivalent} (see also~\cite{Kn3}).

We now describe this latter equivalence explicitly. Let Fl$_1$,
Fl$_2$ be partial $W$-flag varieties and \begin{center}$(n_1,...,
n_s), (n_1',..., n_{s'}')$\end{center} be the corresponding
dimension vectors. These vectors define the following
partitions\begin{center}$(n_1, n_2-n_1..., n_V-n_s)$ and $(n_1',
n_2'-n_1',..., n_V-n_{s'}')$ of $n_V$.\end{center}
\begin{lemma}[{~\cite[Ch.~6.2]{CM}}]\label{TtoP} Two partial flag varieties are cotangent-equivalent if and only if their corresponding partitions coincide as sets.\end{lemma}
For example, Gr$(n_1; V)$ and Gr($n_V-n_1; V$) are cotangent-equivalent.

The relation between partial flag varieties and nilpotent
SL$(V)$-orbits of $\frak{sl}(V)^*$ has been described in terms of
partitions (see~\cite{CM}); the partial order on nilpotent
SL$(V)$-orbits in $\frak{sl}(V)^*$ has been also described in terms
of partitions (see~\cite{CM}). Therefore in order to check that a
given partial flag variety is higher then another one it suffices to
check the corresponding condition on partitions. In this way we
establish in particular the following statements.\\(1) Any partial
flag variety is higher than or is cotangent-equivalent to~$\mathbb
P(V)$.\\(2) If $r_1, r_2\in\{1,..., [\frac{n_V}{2}]\}$ and
$r_1>r_2$, then Gr$(r_1; V)$ is higher than Gr($r_2; V)$.\\(3) The
subset of Grassmannians Gr$(r; V)$ is totally ordered.\\(4) If Fl is
a partial flag variety which is not cotangent-equivalent to $\mathbb
P(V)$, then Fl is higher than or is cotangent-equivalent to Gr$(2;
V)$.\\(5) Any partial $V$-flag variety is cotangent-equivalent to
one of the following\begin{center}Gr$(r; V)$,\hspace{10pt}Fl$(1, 2;
V)$,\hspace{10pt}Fl$(1, 3; V)$,\end{center}or is higher than Fl$(1,
3; V$).\\(6) Any partial $V$-flag variety is cotangent-equivalent to
one of the following\begin{center}Gr$(r; V)$,\hspace{10pt}Fl$(1,r;
V)$,\hspace{10pt}Fl$(1,2,3; V)$,\hspace{10pt} Fl$(2, 4;
V)$,\end{center}or is higher than Fl$(2, 4; V$).

\begin{proposition}\label{SphO} Suppose Fl$_1$ is a $K$-spherical variety and Fl$_2$ is lower then Fl$_1$. Then Fl$_2$ is a $K$-spherical variety.\end{proposition}
To prove the proposition we need to recall some results of I.~Losev.
\begin{theorem}[I.~Losev ~\cite{Lo}]\label{LI}Suppose $X$ is a strongly equidefectinal~\cite[Def. 1.2.5]{Lo} normal affine irreducible Hamiltonian $K$-variety. Then\begin{center} $\mathbb C(X)^K=\mathrm{Quot}(\mathbb C[X]^K)$,\end{center} where $\mathrm{Quot}(\mathbb C[X]^K)$ is the field of fractions of $\mathbb C[X]^K$.\end{theorem}
\begin{corollary}\label{Clp}Suppose $X$ is a strongly equidefectinal affine irreducible Hamiltonian $K$-variety. Then $\mathbb C(X)^K=\mathrm{Quot}(\mathbb C[X]^K)$.\end{corollary}
\begin{proof}Let $\tilde X$ be the spectrum of the integral closure of $\mathbb C[X]$ in Quot$(\mathbb C[X])$ and $\tilde X\to X$ be the canonical finite map. Then $\tilde X$ is a normal affine irreducible Hamiltonian $K$-variety~\cite{Kd}. A straightforward check shows that $\tilde X$ is strongly equideffectional. As $\mathbb C[\tilde X]^K$ is a finite extension of $\mathbb C[X]^K$ of degree 1 and Quot$(\mathbb C[\tilde X]^K)=\mathbb C(\tilde X)^K=\mathbb C(X)^K$, we have \begin{center}Quot($\mathbb C[X]^K)=\mathbb C(X)^K$.\end{center}\end{proof}
The closure of any $G$-orbit in $\frak g^*$ is a strongly equidefectinal affine irreducible Hamiltonian $K$-variety~\cite[Corollary 3.4.1]{Lo}.
\begin{theorem}\label{GrO} Let $\EuScript Z\subset\frak g^*$ and $\EuScript Z'\subset\frak g^*$ be nilpotent $G$-orbits such that $\EuScript Z'\subset\overline{\EuScript Z}$. If $\EuScript Z$ is $K$-coisotropic then $\EuScript Z'$ is $K$-coisotropic.\end{theorem}
\begin{proof} As $K$-action on $\EuScript Z$ is $K$-coisotropic, $\mathbb C(\EuScript Z)^K$ is a Poisson-commutative subfield of $\mathbb C(\EuScript Z)$~\cite[Ch. II, Prop. 5]{Vi}. By Theorem~\ref{LI} this is equivalent to the Poisson-commutativity of $\mathbb C[\EuScript Z]^K$. As $\mathbb C[\EuScript Z']^K$ is a quotient of $\mathbb C[\EuScript Z]^K$, $\mathbb C[\EuScript Z']^K$ is Poisson-commutative. Then the field $\mathbb C(\EuScript Z')^K$ is Poisson-commutative too. Therefore the $K$-action on $\EuScript Z'$ is $K$-coisotropic.\end{proof}
\begin{proof}[Proof of Proposition~\ref{SphO}]The sphericity of a $K$-variety $X$ is equivalent to the $K$-coisotropicity of the $K$-variety T$^*X$. Therefore the statement follows from Theorem~\ref{GrO}.\end{proof}
\begin{corollary} Suppose Fl$(n_1, ..., n_s; V)$ is a $K$-spherical variety. Then the variety $\mathbb P(V)$ is $K$-spherical.\end{corollary}
The following Lemma becomes a useful tool in Section~\ref{Ssphgr} and is a trivial corollary of Proposition~\ref{SphO}.
\begin{lemma}\label{Lgr2}If some partial $W$-flag variety, which is not cotangent-equivalent to $\mathbb P(W)$, is $K$-spherical, then the variety Gr$(2; W)$ is $K$-spherical.\end{lemma}
\subsection{D-modules versus $\frak g$-modules}\label{SSd_g} Let $\frak h_\frak g\subset\frak g$ be a Cartan subalgebra of $\frak g$; $\Delta\subset\frak h_\frak g$ be a root system; $\Delta^+\subset\Delta$ be a set of positive roots. Denote by $\hat{\frak h^*}$ the set of weights $\lambda$ such that $\alpha^\vee(\lambda)$ is not a strictly positive integer for any root $\alpha^\vee$ of the dual rot system $\Delta^\vee\subset\frak h_\frak g^*$.

For a fixed $\lambda$ we denote by $\EuScript D^\lambda(X)$ the
sheaf of twisted differential operators on $X$ and by D$^\lambda(X)$
its space of global sections. The algebras D$^\lambda(X)$ and
D$^\mu(X)$ are naturally identified if $\lambda$ and $\mu$ lie in a
single shifted orbit of the Weyl group~\cite{HMWJ}. Moreover, any
such orbit intersects $\hat{\frak h^*}$~\cite{HMWJ}. If
$\lambda\in\hat{\frak h^*}$, the
isomorphism\begin{center}$\tau:\mathrm U(\frak
g)/($Ker$\chi_\lambda)\to[{\sim}]\mathrm D^{\lambda}(X)$\end{center}
established in~\cite{BeBe} enables us to identify the category of
D$^{\lambda}(X)$-modules with the category of $\frak g$-modules
affording the central character $\chi=\chi_\lambda$. We have the
following functors:
$$\begin{tabular}{c|c}GSec: Sheaves$^\lambda_X\longrightarrow$ $\frak g$-modules$^\chi$&Loc: Sheaves$^\lambda_X\longleftarrow$ $\frak g$-modules$^\chi$\\$\EuScript F\to \Gamma(X, \EuScript F)$&\hspace{26pt}$M\otimes_{(1\otimes\tau)\mathrm U(g)}\EuScript D(X)\longleftarrow M$\hspace{10pt};\\\end{tabular}$$
A.~Beilinson and J.~Bernstein~\cite{BeBe} have proved that if $\lambda\in\hat{\frak h}^*$ then GSec(Loc) equals the identity.

For $\lambda\in\frak h_\frak g^*$ the sheaf $\EuScript D^\lambda(X)$ has a natural filtration  by degree \begin{center}$0\subset\EuScript O(X)\subset \EuScript D_1\subset...\subset\EuScript D(X)=\varinjlim_{i\in\mathbb Z_{\ge0}}\EuScript D_i$.\end{center} The relative spectrum of the associated graded sheaf of algebras \begin{center}gr~$\EuScript D^\lambda(X)=\oplus_{i\in\mathbb Z_{\ge0}}(\EuScript D_{i+1}/\EuScript D_i)$\end{center} is isomorphic to T$^*X$.

Let $\EuScript M$ be a quasicoherent $\EuScript D^\lambda(X)$-module
which admits $\EuScript O(X)$-coherent generating subsheaf
$\EuScript M_{gen}$ of $\EuScript M$. The associated graded sheaf
gr$\EuScript M$ with respect to the filtration
\begin{center}$0\subset\EuScript M_{gen}\subset\EuScript
D_1\EuScript M_{gen}\subset...\subset\EuScript M$\end{center} is a
gr$\EuScript D^\lambda(W)$-module. By definition, the {\it singular
support} $\EuScript V(\EuScript M)$ of $\EuScript M$ is the support
of gr$\EuScript M$ in Spec$_X$gr$\EuScript D^\lambda(X)\cong$T$^*X$.
\begin{theorem}[O. Gabber~\cite{Gab}]\label{GabD} The variety $\EuScript V(\EuScript M)$ is coisotropic in T$^*X$. In particular \begin{center}dim$\tilde{\EuScript V}\ge$dim$X$\end{center} for any irreducible component $\tilde{\EuScript V}\subset\EuScript V(\EuScript M)$. \end{theorem}
\begin{definition} The $\EuScript D^\lambda(X)$-module $\EuScript M$ is called {\it holonomic} if \begin{center}dim$\EuScript V(\EuScript M)=$dim$X$.\end{center}\end{definition}
\subsection{Associated varieties of $\frak g$-modules}\label{SSavg} The algebra U$(\frak g)$ has a natural filtration  by degree \begin{center}$0\subset\mathbb C\subset U_1\subset...\subset$U$(\frak g)=\cup_{i\in\mathbb Z_{\ge0}}U_i$.\end{center} The associated graded algebra \begin{center}gr~U$(\frak g) =\oplus_{i\in\mathbb Z_{\ge0}}(U_{i+1}/U_i)$\end{center} is isomorphic to S$(\frak g)$. The filtration $\{U_i\}_{i\in\mathbb Z_{\ge0}}$ induces the filtration on any ideal $I$ of U$(\frak g)$, namely $\{I\cap U_i\}_{i\in\mathbb Z_{\ge0}}$, hence gr$I\subset$S$(\frak g)$ is a well defined ideal. The ideal gr$I$ of the commutative algebra S$(\frak g)$ determines the variety
\begin{center}V(gr$I):=\{x\in\frak g^*\mid f(x)=0$ for all $f\in$gr$I\}$.\end{center}
In particular, if $I$=Ann$M$ for a $\frak g$-module $M$, we set GV$(M):=$V(gr$I$).
\begin{theorem}[\cite{Jo}] For a simple $\frak g$-module $M$ the variety GV$(M)$ is the closure of an orbit $Gu$, and furthermore $0\in\overline{Gu}$.\end{theorem}
Suppose that $M$ is a U$(\frak g)$-module and that a filtration \begin{center}$0\subset M_0\subset M_1\subset...\subset M=\cup_{i\in \mathbb Z_{\ge 0}}M_i$\end{center} of vector spaces is given. We say that this filtration is {\it good} if\begin{center}(1) $U_i M_j= M_{i+j}$;\hspace{40pt}(2) dim~$M_i<\infty~$for all $i\in\mathbb Z_{\ge 0}$.\end{center}
Such a filtration arises from any finite-dimensional space of generators $M_0$. The corresponding associated graded object gr$M=\oplus_{i\in\mathbb Z_{\ge0}}M_{i+1}/M_i$ is a module over gr~U($\frak{g})\cong$S$(\frak g)$, and we set
\begin{center}J$_M:=\{s\in $S$(\frak g)$ $\mid$ there exists $k\in \mathbb Z_{\ge0}$ such that $s^km=0$ for all $m\in $gr$M\}.$\end{center}
In this way we associate to any $\frak g$-module $M$ the variety
\begin{center}V($M):=\{x\in\frak g^*\mid f(x)=0$ for all $f\in$J$_M\}$.\end{center}
It is easy to check that the module gr$M$ depends on the choice of good filtration, but the ideal J$_M$  and the variety V($M)$ does not. Indeed, let $M_0, M_0'$ be different generating spaces of $M$ and $\{M_i\}_{i\in\mathbb Z_{\ge0}}, \{M_i'\}_{i\in\mathbb Z_{\ge0}}$ be the corresponding filtrations of $M$. Then there exist $r, s\in\mathbb Z_{\ge0}$ such that $M_0\in M_r'$ and $M_0'\in M_s$. Let $f\in$S$(\frak g)$ be an element of the degree $d$ such that \begin{center}$fM_i\subset M_{i+d-1}$ for all $i\in\mathbb Z_{\ge0}$.\end{center} Then \begin{center}$f^{r+s+1}M_0'\subset f^{r+s+1}M_s\subset$\\$M_{s+(d-1)(r+s+1)}\subset M_{s+r+(d-1)(r+s+1)}'=M_{(s+r+1)d-1}'$.\end{center} Therefore J$^{s+r+1}_M\subset$J$'_M$. In the same way J$_M'^{s+r+1}\subset$J$_M$, and therefore \begin{center}V$(M)$=V$(M)'$.\end{center}
\begin{theorem}[Bernstein's theorem~{\cite[p. 118]{KL}}]\label{Tber} Let $M$ be a finitely generated $\frak g$-module. Then dim~V$(M)\ge\frac{1}{2}$dim~GV($M$).\end{theorem}
Let $M$ be a finitely generated $\frak g$-module which affords a generalized central character.
\begin{theorem}[O. Gabber~\cite{Gab}]\label{Gab} Let $\tilde V$ be an irreducible component of V$(M)$ and $\EuScript Z$ be the unique open $G$-orbit of \begin{center}$G\tilde V:=\{x\in\frak g^*\mid x=gv$ for some $g\in G$ and $v\in\tilde V$\}.\end{center} Then $\tilde V\cap \EuScript Z$ is a coisotropic subvariety of $\EuScript Z$. In particular \begin{center}dim$\tilde V\ge\frac{1}{2}$dim$G\tilde V$.\end{center}\end{theorem}
\begin{definition}\label{Dhol} A simple $\frak g$-module $M$ is called {\it holonomic} if \begin{center}dim$\tilde V=\frac{1}{2}$dim$G\tilde V$\end{center} for any irreducible component $\tilde V$ of V$(M)$.\end{definition}
\begin{corollary}[S.~Fernando~\cite{F}] The vector space \begin{center}V$(M)^\bot=\{g\in\frak g\mid v(g)=0$ for all $v\in$V$(M)$\}\end{center} is a Lie algebra and V$(M)$ is a
V$(M)^\bot$-variety.\end{corollary}
\begin{theorem}[S.~Fernando~{\cite[Cor. 2.7]{F}}, V.~Kac~\cite{Kc2}]Set \begin{center}$\frak g[M]:=\{g\in\frak g\mid\mathrm{dim}(\mathrm{span}_{i\in\mathbb Z_{\ge 0}}\{ g^im\})< \infty$ for all $m\in M\}$.\end{center} Then $\frak g[M]$ is a Lie algebra and $\frak g[M]\subset\mathrm V(M)^\bot$.\end{theorem}
\begin{corollary}For a $(\frak g, \frak k)$-module $M$, V$(M)\subset\frak k^\bot$ and V$(M)$ is a $K$-variety.\end{corollary}
Let $M$ be a $(\frak g, \frak k)$-module and $M_0$ be a $\frak k$-stable finite-dimensional space of generators of $M$; J$_M,$ gr$M$ be the corresponding objects constructed as above. Consider the S$(\frak g)$-modules\begin{center}  J$_M^{-i}\{0\}:=\{m\in$gr$M\mid j_1...j_im=0$ for all $j_1,...,j_i\in$J$_M~\}$.\end{center} One can easily see that these modules form an ascending filtration\begin{center}$0\subset$J$_M^{-1}\subset...\subset$gr$M$\end{center}such that \begin{center}$\cup_{i=1}^\infty$J$_M^{-i}\{0\}=$gr$M$.\end{center} Since S$(\frak g)$ is a N\"otherian ring, the filtration stabilizes, i.e. J$_M^{-i}\{0\}=$gr$M$ for some $i$. By $\overline{\mathrm{gr}}M$ we denote the corresponding graded object. By definition, $\overline{\mathrm{gr}}M$ is an S$(\frak g)/$J$_M$-module. Suppose that $f\overline{\mathrm{gr}}M=0$ for some $f\in$S$(\frak g)$. Then $f^i$gr$M=0$ and hence $f\in$J$_M$. This proves that the annihilator of $\mathrm{\overline{gr}}M$ in $\mathrm S(\frak g)/\mathrm J_M$ equals zero.

As $M$ is a finitely generated $\frak g$-module, the S$(\frak
g)$-modules gr$M$ and $\mathrm{\overline{gr}}M$ are finitely
generated. Let $\tilde M_0$ be a $\frak k$-stable finite-dimensional
space of generators of $\mathrm{\overline{gr}}M$. Then there is a
surjective homomorphism \begin{center}$\psi: \tilde
M_0\otimes_\mathbb C$(S$(\frak
g)/$J$_M)\to\overline{\mathrm{gr}}M$.\end{center}
Set\begin{center}Rad$M:=\{m\in\overline{\mathrm{gr}}M|$ there exists
$f\in$S$(\frak g)/$J$_M$ such that $fm=0$ and  $f\ne
0\}$.\end{center} The space Rad$M$ is a $\frak k$-stable S$(\frak
g)$-submodule of $\overline{\mathrm{gr}}M$ and $\tilde
M_0\not\subset$Rad$M$. The homomorphism $\psi$ determines an
injective homomorphism
\begin{center}$\hat\psi:$ S$(\frak g)/$J$_M\to\tilde
M_0^*\otimes_\mathbb C\overline{\mathrm{gr}}M$.\end{center}
\begin{proposition}\label{Pbm}a) The module $M$ is bounded if and only if all irreducible components of $\mathrm V(M)$ are $K$-spherical.\\ b) If the equivalent conditions of a) are satisfied, then any irreducible component $\tilde V$ of $\mathrm V(M)$ is a conical Lagrangian subvariety of $G\tilde V$.\end{proposition}
\begin{proof}a) Suppose that all irreducible components of $\mathrm V(M)$ are $K$-spherical. Then \begin{center}$\tilde M_0\otimes_\mathbb C$(S$(\frak g)/$J$_M$)\end{center} is a bounded $\frak k$-module. Therefore $\overline{\mathrm{gr}}M$ is bounded, which implies that $M$ is a bounded $\frak k$-module too.

Assume now that a $\frak g$-module $M$ is $\frak k$-bounded. Then S$(\frak g)/$J$_M$ is a $\frak k$-bounded module and all irreducible components of $\mathrm V(M)$ are $K$-spherical. This completes the proof of a).\\
b) Let $\tilde V\subset\mathrm V(M)$ be an irreducible component and
$x\in\tilde V$ be a generic point. As $x\in\frak k^\bot$, we have
\begin{center}$x([k_1, k_2])=\omega_x(\tau_{\frak g^*}k_1|_x,
\tau_{\frak g^*}k_2|_x)=0$\end{center} for all $k_1, k_2\in\frak k$.
Therefore any $K$-orbit in $\frak k^\bot$ is isotropic. As $\tilde
V$ is a spherical variety, $\tilde V$ has an open $K$-orbit.
Therefore $\tilde V$ is Lagrangian in $G\tilde V$ and  this
completes the proof of b).\end{proof}
\begin{corollary}[{\cite{PS}}]\label{Cps} Let $M$ be a finitely generated bounded $(\frak g, \frak k)$-module. Then dim$\frak b_\frak k \ge\frac{1}{2}$dim~GV$(M)$.\end{corollary}
\begin{proof}As $M$ is bounded, V$(M)$ is $K$-spherical and therefore\\ dim$\frak b_\frak k\ge$dim~V$(M)$. We have dim~V$(M)\ge\frac{1}{2}$dim~GV$(M)$.\end{proof}
\begin{lemma}[\cite{VP}] Let $\hat X$ be an affine $K$-variety. Then $\mathbb C[X]$ has finite type as a $\frak k$-module if and only if $X$ contains only finite number of the closed $K$-orbits. In the latter case any irreducible component of $X$ contains precisely one closed $K$-orbit.\end{lemma}
\begin{lemma}A finitely generated $(\frak g, \frak k)$-module $M$ has finite type over $\frak k$ if and only if the variety V$(M)$ contains only finitely many closed $K$-orbits. In this case the unique closed orbit of V$(M)$ is the point 0.\end{lemma}
\begin{proof}If V$(M)$ contains the unique closed $K$-orbit then $\mathbb C[$V$(M)]$ is a $\frak k$-module of finite type and therefore $\tilde M_0\otimes_\mathbb C($S$(\frak g)/$J$_M$) and $\overline{\mathrm{gr}}M$ are $\frak k$-modules of finite type.

Assume that $M$ is a $\frak k$-module of finite type. Then S$(\frak g)/$J$_M$ is a $\frak k$-module of finite type and V($M$) contains only finite number of the closed $K$-orbits. On the other hand, V$(M)$ is $\mathbb C^*$-stable and hence any irreducible component of V($M$) contains 0.\end{proof}
\subsection{Other faces of the support variety}\label{SSof}In this section $X$ is a variety of Borel subalgebras of $\frak g$. We recall the singular support $\EuScript V(\EuScript M)\subset$T$^*X$ of any coherent $\EuScript D^\lambda(X)$-module $\EuScript M$  is defined. The correspondence between $\frak g$-modules and $\EuScript D^\lambda(X)$-modules allows us to assign interesting geometric objects to a $\frak g$-module.

Let $M$ be a finitely generated $\frak g$-module which affords a central character $\chi$ and $\lambda\in\hat{\frak h}^*$ (see Subsection~\ref{SSd_g}) be a weight such that $\chi=\chi_\lambda$.
\begin{definition} The {\it singular support} of  $M$ is the variety $\EuScript V(M):=\EuScript V($Loc$M)$.\end{definition}
\begin{definition} The {\it support variety} L$(M)$ of $M$ is the projection of $\EuScript V(M)$ to $X$.\end{definition}
Let $\phi_X:$T$^*X\to\frak g^*$ be the moment map. D.~Barlet and M.~Kashiwara~\cite{BK}, have proved that \begin{center}V$(M)=\phi_X(\EuScript V($Loc$M))$\end{center} (see also~\cite{BoBry}). Therefore we have a diagram
\begin{center}\hspace{40pt}$\xymatrix{&\EuScript V(M)\subset\mathrm T^*X\ar[ld]^{\phi_X} \ar[rd]^{\mathrm{pr}} &\\\mathrm{V(}M\mathrm{)}\subset\frak g^*&&\mathrm L(M)\subset X\hspace{39pt}.}$\end{center}
\subsection{Hilbert-Mumford criterion}\label{SSHM} Let $X$ be an affine $K$-variety, $V$ be a $K$-module.
\begin{theorem}[Hilbert-Mumford~\cite{VP}] The closure of any $K$-orbit $\overline{Kx}\subset X$ contains a unique closed orbit $K\overline x\subset X$. There exists a group homomorphism $\mu:\mathbb C^*\to K$ such that $\lim\limits_{t\to 0}\mu(t)x=\bar x\in K\overline x$.\end{theorem}
The {\it null cone} N$_K(V):=\{x\in V\mid 0\in \overline{Kx}\}$ is a closed algebraic subvariety of $V$~\cite{VP}.
\begin{theorem}[\cite{VP}]Fix $x\in V$. Then $0\in \overline{Kx}$ if and only if there exists a rational semisimple element $h\in\frak k$ such that $x\in V^{>0}_h$; here $V^{> 0}_h$ is the direct sum of $h$-eigenspaces in $V$ with positive eigenvalues.\end{theorem}
\begin{corollary}[\cite{VP}]\label{CHM} There exists a finite set $H\subset\frak k$ of rational semisimple elements such that N$_K(V):=\cup_{h\in H}KV_h^{>0}$, where \begin{center}$KV_h^{>0}:=\{v\in V\mid v=kv_h$ for some $k\in K$ and $v_h\in V_h^{>0}$\}.\end{center}\end{corollary}
\subsection{A monoid of projective functors}\label{SSpro} Let $\frak b_\frak g$ be a Borel subalgebra of $\frak g$ and let $\frak h_\frak g$ be a Cartan subalgebra of $\frak b_\frak g$. Let $\Delta\subset \frak h_\frak g^*$ be the root system of $\frak g$, $\Delta^+\subset\Delta$ be the set of positive roots of $\frak b_g$, $\Pi\subset \Delta^+$ be the set of simple roots. We denote by s$_\alpha$ the reflection of $\frak h_\frak g^*$ with respect to $\alpha$ for $\alpha\in\Delta$, and $W^\frak g$ is the group generated by  $\{$s$_\alpha\}_{\alpha\in\Delta}$, i.e. the Weyl group of $\frak g$. Let $\alpha^\vee\in \frak h_\frak g$ be the coroot of $\alpha\in\Delta$, and \begin{center}$\rho:=\frac{1}{2}\Sigma_{\alpha\in\Delta^+}\alpha$.\end{center}

We introduce a partial order on $\frak h_\frak g^*$ (cf.~\cite[1.5]{BeG}). If $\phi, \psi\in\frak h_\frak g^*$, $\gamma\in\Delta^+$ we put $\phi<^\gamma\psi$ whenever $\phi=$s$_\gamma\psi$ and $\gamma^\vee(\psi)\in\mathbb Z_{>0}$. We put $\phi<\psi$ whenever there exists sequences of weights $\phi_0,..., \phi_n$ and of roots $\gamma_0,..., \gamma_n$ such that \begin{center}$\phi<^{\gamma_0}\phi_0<...<^{\gamma_n}\phi_n=\psi$.\end{center} We write $\phi\le \psi$ whenever $\phi=\psi\vee\phi<\psi$. A weight is called {\it dominant} if it is maximal with respect to partial order $<$.

We denote by W$^\frak g_\phi$ the stabilizer in W$^\frak g$ of $\phi\in\frak h_\frak g^*$. By definition, $\phi\in\frak h_\frak g^*$ {\it integral} if $\alpha^\vee(\phi)\in\mathbb Z$ for any $\alpha\in\Delta$. Furthermore, $\phi\in\frak h_\frak g^*$ is {\it regular} if W$^\frak g_\phi=\{e\}$. We call a pair of weights $(\phi, \psi)\in\frak h_\frak g^*\times\frak h_\frak g^*$ {\it correctly ordered} if $\phi$ is dominant, $\phi-\psi$ is an integral weight and $\psi\le w\psi$ for all $w\in$W$^\frak g_\phi$. In addition, we define two weights $\phi_1, \phi_2$ to be {\it equivalent} if $\phi_1-\phi_2$ is integral and W$^\frak g_{\phi_1}=$W$^\frak g_{\phi_2}$.

Let $\EuScript O$ be the category of finitely generated $\frak g$-modules such that the action of $\frak b_\frak g$ is locally finite and the action of $\frak h_\frak g$ is semisimple. Let $\lambda\in\frak h_\frak g^*$ be a weight. We denote by M$_\lambda$ the Verma module with the highest weight $\lambda-\rho$, and by L$_\lambda$ the unique simple quotient of M$_\lambda$. Both M$_\lambda$ and L$_\lambda$ are objects of the category $\EuScript O$. We denote by P$_\lambda$ a minimal projective cover of L$_\lambda$ in $\EuScript O$. If $\lambda$ is dominant, then M$_\lambda\simeq$P$_\lambda$.

The modules M$_\lambda$, L$_\lambda$, P$_\lambda$ afford the same generalized central character and we denote this central character by $\chi_\lambda$. Let $I$ be a primitive ideal of U$(\frak g)$. By a famous theorem of Duflo~\cite{Dix}, $I$=Ann~L$_\lambda$ for some weight $\lambda\in\frak h_\frak g^*$.

Let $\chi$ be a central character. The set of weights $\lambda$ such that M$_\lambda$ are annihilated by Ker$\chi$ is nonempty and we denote this set W$^\frak g(\chi)$. The set W$^\frak g(\chi)$ is an orbit of W$^\frak g$. Therefore we can identify the set of central characters with the set of W$^\frak g$-orbits in $\frak h_\frak g^*$.

The Weyl group W$^\frak g$ acts on $\frak h_\frak g^*\times\frak h_\frak g^*$:\begin{center} $w((\lambda_1, \lambda_2))\mapsto (w(\lambda_1), w(\lambda_2))$.\end{center} Any W$^\frak g$-orbit in $\frak h_\frak g^*\times\frak h_\frak g^*$ contains a correctly ordered representative.

Let $E$ be a finite-dimensional $\frak g$-module. We have the functor\begin{center}F$_E:\frak g-$mod$\to\frak g-$mod$,\hspace{10pt} M\mapsto E\otimes M$.\end{center}The restriction of F$_E$ to the category of $\frak g$-modules affording a given generalized central character $\chi$ is a direct sum of a finite number of indecomposable exact functors. Following~\cite{BeG}, we call a direct summand of F$_E$ a {\it projective functor}.

The set of indecomposable projective functors arising from all finite-dimensional $\frak g$-modules $E$ is naturally identified with the set of W$^\frak g$-orbits in $\frak h_\frak g^*\times\frak h_\frak g^*$~\cite[3.3]{BeG}. Let $(\phi, \psi)\in\frak h_\frak g^*\times\frak h_\frak g^*$ be a correctly ordered pair and $\EuScript H_\phi^\psi$ be the corresponding projective functor. Then \begin{center}$\EuScript H_\phi^\psi$M$_\phi=$P$_\psi$\end{center} and the functors H$_\phi^\psi$ and H$_\psi^\phi$ are adjoint one to each other. Assume that $\phi-\psi$ is integral and W$^\frak g_\phi=$W$^\frak g_\psi$. Then \begin{center}$(\EuScript H_\phi^\psi, \EuScript H_\psi^\phi)$\end{center} is a pair of mutually inverse equivalences of the respective categories of $\frak g$-modules affording the generalized central characters $\chi_\phi$ and  $\chi_\psi$.

Let PFunc$(\chi_\lambda$) be the set of projective functors as above which preserve the category of $\frak g$-modules affording the generalized central character $\chi_\lambda$. Obviously PFunc$(\chi_\lambda)$ carries an additive structure and is a monoid with respect to it. We denote by PF$\overline{\mathrm{unc}}(\chi_\lambda)$ the minimal abelian group which contains PFunc$(\chi_\lambda$). As the composition of projective functors is a projective functor, PF$\overline{\mathrm{unc}}(\chi_\lambda$) carries a multiplicative structure. Therefore PF$\overline{\mathrm{unc}}(\chi_\lambda)$ is a unitary ring.
\begin{proposition}\label{Ppfsn}Let $\lambda$ be a regular dominant integral weight. Then the ring PF$\overline{\mathrm{unc}}(\chi_\lambda)$ is isomorphic to the ring $\mathbb Z[$W$^\frak g$].\end{proposition}
\begin{proof}Let $\EuScript O_\lambda$ be the subcategory of $\EuScript O$ consisting of $\frak g$-modules which afford the generalized central character $\chi_\lambda$. Let $[\EuScript O]_\frak g$ be the free abelian group generated by the isomorphism classes of simple $\frak g$-modules of the category $\EuScript O$, and let $[\EuScript O_\lambda]_\frak g$ be the subgroup generated by the simple objects of $\EuScript O_\lambda$. Any object $M$ of $\EuScript O$ has finite length and therefore determines [$M]\in [\EuScript O]_\frak g$.

Let $\EuScript H_1, \EuScript H_2\in$PFunc$(\chi_\lambda)$. Then $\EuScript H_1=\EuScript H_2$ if and only if \begin{center}$[\EuScript H_1($M$_\lambda)]=[\EuScript H_2($M$_\lambda)]$.\end{center} For any $w\in$W$^\frak g$, we have $\EuScript H_\lambda^{w\lambda}\in$PFunc$(\chi_\lambda)$ and $\EuScript H_\lambda^{w\lambda}$M$_\lambda=$P$_{w\lambda}$. The elements $\{[$M$_{w\lambda}]\}_{w\in\mathrm W^\frak g}$\} form a basis of the free abelian group $[\EuScript O_\lambda]_\frak g$, and therefore $[\EuScript O_\lambda]_\frak g$ is identified with $\mathbb Z^{\mathrm W^\frak g}$ as a set. We introduce an action of W$^\frak g$ on $[\EuScript O_\lambda]_\frak g$ via the formula\begin{center}$(w, [$M$_{w'\lambda}])\mapsto $[M$_{ww'\lambda}]$ for any $w', w\in$W$^\frak g$.\end{center}Any functor $\EuScript H\in$PFunc$(\chi_\lambda)$ commutes with this action. As $\{[$P$_{w\lambda}]\}_{w\in\mathrm W^\frak g}$ is a basis of the free abelian group $[\EuScript O_\lambda]_\frak g$, the map \begin{center}PF$\overline{\mathrm{unc}}(\chi_\lambda)\to\mathbb Z[$W$^\frak g$]$=\mathbb Z^{\mathrm W^\frak g}$,\end{center} given by $\EuScript H\mapsto{}{}[\EuScript H($M$_\lambda)]$ on generators, is an isomorphism.\end{proof}
\begin{theorem}[\cite{BeG}]\label{Tweq} Let $\lambda_1, \lambda_2$ be dominant weights such that the difference $\lambda_1-\lambda_2$ is integral and W$^\frak g_{\lambda_1}=$W$^\frak g_{\lambda_2}$, i.e. $\lambda_1$ is equivalent to $\lambda_2$. Then\\a) the categories of $\frak g$-modules which afford the generalized central characters $\chi_{\lambda_1}$ and $\chi_{\lambda_2}$ are equivalent;\\b) for a given $w\in$W$^\frak g$, the categories of $\frak g$-modules annihilated by Ann~L$_{w\lambda_1}$ and Ann~L$_{w\lambda_2}$ are equivalent.\end{theorem}
The subcategories of locally finite $\frak k$-modules and bounded $\frak k$-modules are stable under this equivalence.
\subsection{Bounded weight modules}\label{SM} In the joint work~\cite{PS} I.~Penkov and V.~Serganova proved that 'boundness' is not a property of a module $M$ but of the ideal
Ann$M$. More precisely, this means the following.
\begin{theorem}\label{Trel}Let $M$ and $N$ be simple $(\frak g, \frak k)$-modules such that Ann$M$=Ann$N$. Suppose that the function $[M:\cdot]_\frak k$ is uniformly bounded by a constant $C_M$. Then the function $[N:\cdot]_\frak k$ is uniformly bounded by $C_M$ (see also Theorem~\ref{Tcob}).\end{theorem}
To classify bounded modules $M$ we should first classify 'bounded ideals' Ann$M$, i.e. ideals for which exist at least one bounded module. It is well known that the two-sided ideals of U$(\frak g)$ containing a fixed maximal ideal in Z$(\frak g)$ are closely related to category $\EuScript O$~\cite{BeG}. Therefore one may expect that a classification of bounded ideals has something in common with the classification of Verma modules L$_\lambda$ whose weight multiplicities are bounded.

We fix the following notation:\\
--- $\{$e$_i\}_{i\le n_W} \subset W$ is a basis of W, \{e$_i^*\}_{i\le n_W}\subset W^*$ is the dual basis;\\
--- $\{$e$_{i,j}\}_{i\le n_W , j\le n_W}\subset\frak{gl}(W)$ stands for the elementary matrix e$_i\otimes$e$_j^*$;\\
--- $\varepsilon_i:=$e$_{i,i}-\frac{1}{n_W}($e$_{1,1}+$e$_{2,2}+...+$e$_{n_W, n_W}$) for $i\le n_W$;\\
--- $\frak h_W:=$span$\langle$e$_{i,i}\rangle_{i\le n_W}$.
\subsubsection{The case $\frak g=\frak{sl}(W)$}\label{SSbmsl}We identify all weights of $\frak h^{\frak{sl}}_W:=\frak h_W\cap\frak{sl}(W)$ with the set of $n_W$-tuples modulo the equivalence relation\begin{center}
$(\lambda_1, \lambda_2,...,
\lambda_{n_W})\leftrightarrow(\lambda_1+k, \lambda_2+k,...,
\lambda_{n_W}+k).$\end{center} We fix the set of positive roots of
$\frak{sl}(W)$ in $\frak h_W^\frak{sl}$ for which $\rho$ is
identified with $(n_W,..., 1)$ and $\varepsilon_i$ is identified
with the $n_W$-tuple $(i$ zeros , 1, $n_W-i-1$ zeros). Let
$\bar\lambda$ be an $n_W$-tuple and $\lambda$ be the corresponding
weight. We put I$(\lambda):=$Ann~L$_\lambda$. Duflo's
theorem~\cite{Dix} claims now that, if $I$ is any primitive ideal of
U$(\frak{sl}(W))$, then \begin{center}$I$=I$(\lambda)$\end{center}
for some $n_W$-tuple $\bar\lambda$.
\begin{theorem}[\cite{M}]\label{Tsmo} The $(\frak{sl}(W), \frak h^{\frak{sl}}_W)$-module L$_\lambda$ is $\frak h^{\frak{sl}}_W$-bounded if and only if $\lambda$ is a semi-decreasing tuple.\end{theorem}
The bounded $(\frak{sl}(W), \frak h^{\frak{sl}}_W)$-modules split into coherent families~\cite{M}, and all simple modules from one coherent family share an annihilator. Let $\bar\lambda$ be a semi-decreasing tuple. Assume that $\bar\lambda$ is not regular integral, i.e. it is singular integral or semi-integral. Then there exists precisely one coherent family annihilated by Ker$\chi_\lambda$~\cite{M}. Therefore all simple bounded $(\frak{sl}(W), \frak h^{\frak{sl}}_W)$-modules which afford the central character $\chi_\lambda$ have the same annihilator.

V.~Serganova and D.~Grantcharov~\cite{GrS1} note that the lemma below follows from the results of O.~Mathieu~\cite{M}. See also Theorem~\ref{Tweq}.
\begin{lemma}\label{Lsoch} Assume $\bar\lambda$ is a semi-decreasing tuple and that $\bar\lambda$ is not regular integral. Then there exists $s\in\mathbb C$ such that \begin{center}e$^{2\pi is}=$m$(\lambda)$\end{center} and the categories of $\frak{sl}(W)$-modules annihilated by \begin{center}I$(\lambda)$ and I$(s, n_W-1,...,1)$\end{center} are equivalent.\end{lemma}
The subcategories of locally finite $\frak k$-modules and bounded $\frak k$-modules are stable under this equivalence.

Assume that $\bar\lambda$ is regular integral. Let ord$(\bar\lambda)$ be a decreasing $n_W$-tuple which coincides as a set with $\bar\lambda$. For $k\in\{1,..., n_W-1\}$ we denote by s$_k$ the permutation $(k, k+1)\in$S$_{n_W}$. There exist precisely $n_W-1$ different coherent families annihilated by Ker$\chi_\lambda$ and any such family contains L$_{\mathrm s_k\mathrm{ord}(\lambda)}$ for some $k\in\{1,..., n_W-1\}$. Therefore \begin{center}I$(\lambda)$=I(s$_k$ord$(\lambda))$\end{center} for some $k\in\{1,..., n_W-1\}$~\cite{M}.

Consider the Weyl algebra of differential operators D($W$) on $W$. This algebra is generated by e$_i$ and $\partial_{\mathrm e_i}$ for $i\le n_W$. Set \begin{center}$E:= $e$_1\partial_{\mathrm e_1}+...+$e$_{n_W}\partial_{\mathrm e_{n_W}}$.\end{center}The operator [$E, \cdot$]: D$(W)\to$D$(W$) is semisimple, and D$(W$) splits into a direct sum of eigenspaces \begin{center}\{D$^i(W)\}_{i\in\mathbb Z}$\end{center} with respect to this operator.

The homomorphism of Lie algebras \begin{center}$\frak{gl}(W)\to$D$^0(W),\hspace{40pt}$e$_{i, j}\mapsto$e$_i\partial_{\mathrm e_j}$\end{center}induces a surjective homomorphism \begin{center}$\phi:$U$(\frak{gl}(W))\to$D$^0(W)$.\end{center} The element $E-t$ generates a two-sided ideal in D$^0(W)$. We denote the corresponding quotient by D$^t\mathbb P(W)$. Furthermore, $\phi$ induces a surjective homomorphism \begin{center} $\phi_t:$U$(\frak{sl}(W))\to$D$^t\mathbb P(W)$\end{center}and\begin{center}Ker$\phi_t=$I$(t, n_W-1,...,1)$,\end{center}see~\cite{M}.
\subsubsection{The case $\frak g=\frak{sp}(W\oplus W^*)$}We note that $\frak h_W$ is a Cartan subalgebra of $\frak{sp}(W\oplus W^*)$. We identify the weights $\frak h_W^*$ with $n_W$-tuples. We fix the set of positive roots of $\frak{sp}(W\oplus W^*)$ in $\frak h_W$ for which $\rho$ is identified with $(n_W,..., 1)$. Let $\bar\mu$ be an $n_W$-tuple and $\mu$ be the corresponding weight. We denote Ann~L$_\mu$ by I$_{\frak{sp}}(\mu)$. Let $I$ be a primitive ideal of U$(\frak{sp}(W\oplus W^*))$. Then $I$=I$_\frak{sp}(\mu)$ for some $n_W$-tuple $\bar\mu$.
\begin{theorem}[\cite{M}]\label{Tbwm} The $(\frak{sp}(W\oplus W^*), \frak h_W)$-module L$_\mu$ is $\frak h_W$-bounded if and only if $\bar\mu$ is a Shale-Weil tuple.\end{theorem}
The bounded $(\frak{sp}(W\oplus W^*), \frak h_W)$-modules split into coherent families~\cite{M} and all simple modules from one coherent family share an annihilator. Let $\bar\mu$ be a Shale-Weil tuple. There exists precisely one coherent family annihilated by Ker$\chi_\mu$~\cite{M}. Therefore all simple bounded $(\frak{sp}(W\oplus W^*), \frak h_W)$-modules affording the central character $\chi_\mu$ have the same annihilator. The modules L$_\mu$ and L$_{\sigma\mu}$ belong to the same coherent family. Recall that $\mu_0$ is the $n_W$-tuple $(n_W-\frac{1}{2}, n_W-\frac{3}{2},..., \frac{1}{2})$.
\begin{lemma}\label{L332}Let $\bar\mu$ be a Shale-Weil $n_W$-tuple. Then the categories of U$(\frak{sp}(W\oplus W^*))$-modules annihilated by \begin{center}I$_\frak{sp}(\mu)$ and I$_\frak{sp}(\mu_0)$\end{center} are equivalent.\end{lemma}
\begin{proof}This is a simple corollary of Theorem~\ref{Tweq} b).\end{proof}
The subcategories of locally finite $\frak k$-modules and bounded
$\frak k$-modules are stable under the equivalence of
Lemma~\ref{L332}.

The spaces \{D$^i(W)$\}$_{i\in\mathbb Z}$ define a $\mathbb Z$-grading of an algebra D$(W)$. Therefore the spaces \begin{center}\{$\oplus_{i\in\mathbb Z}$D$^{2i}(W), \oplus_{i\in\mathbb Z}$D$^{2i+1}(W)$\}\end{center} define a $\mathbb Z_{2}$-grading of the algebra D$(W)$. We set \begin{center}D$^{\bar 0}(W):=\oplus_{i\in\mathbb Z}$D$^{2i}(W)$,\hspace{10pt} D$^{\bar1}(W):=\oplus_{i\in\mathbb Z}$D$^{2i+1}(W)$.\end{center}
The space span $\langle$e$_i$e$_j, $e$_i\partial_{\mathrm e_j}, \partial_{\mathrm e_i}\partial_{\mathrm e_j}, 1\rangle_{i, j\le n_W}$ is a Lie algebra with respect to commutator. This Lie algebra is isomorphic to $\frak{sp}(W\oplus W^*)\oplus\mathbb C$. We note that span$\langle$e$_i\partial_{\mathrm e_j}\rangle_{i, j\le n_W}$ is a Lie subalgebra and is isomorphic to $\frak{gl}(W)$. This defines a homomorphism of associative algebras\begin{center}$\phi_\frak{sp}:$U$(\frak{sp}(W\oplus W^*))\to$D$(W)$,\end{center}and I$(\mu_0)$=Ker$\phi_\frak{sp}$, D$^{\bar0}(W)$=Im$\phi_\frak{sp}$.
\subsection{A lemma on S$_{n_W}$-modules}\label{SSsn} The set of isomorphism classes of simple S$_{n_W}$-modules is naturally identified with the set of partitions \begin{center}$m_1\ge m_2\ge ...\ge m_s (m_1+m_2+...+m_s=n_W)$\end{center} of $n_W$. Let $\mathbb C^{n_W-1}$ be the natural $(n_W-1)$-dimensional module of S$_{n_W}$ and $\mathbb C$ be the trivial one-dimensional module. Recall that s$_i$ is the simple transposition $(i, i+1)$. Let $R$ be an S$_{n_W}$-module. For any $i\in\{1,..., n_W-1\}$ we denote by $R_{\bar i}$ the vector space \begin{center}$\{v\in R\mid$~s$_jv=v$ for all $j\ne i\}$.\end{center}
\begin{lemma}\label{Lsn} Suppose that $R$ is an S$_{n_W}$-module such that \begin{center}$R=+_{i\le n_W-1}R_{\bar i}$.\end{center} Then $R$ is a direct sum of several copies of $\mathbb C^{n_W-1}$ and $\mathbb C$.\end{lemma}
\begin{proof} Without loss of generality we assume that $R$ is irreducible. If $R_{\bar 1}$ equals zero, the sum $+_{i\le n_W-1}R_{\bar i}$ is s$_1$-invariant and therefore $R\cong \mathbb C$. If $R_{\bar 1}\ne0$, the restriction of $R$ to S$_{n_W-1}$ contains $\mathbb C$ as a simple submodule. Therefore $R$ is \begin{center}$\mathbb C^{n_W-1}$, or $\mathbb C$,\end{center}see~\cite{FH}.\end{proof}
\subsection{Regular singularities}\label{SSrs} Let $C$ be a smooth connected affine curve. We recall that $\mathbb C[C]$ is the algebra of regular functions of $C$ and D$(C)$ is the algebra of regular differential operators on $C$. Let $(x)\subset\mathbb C[C]$ be a maximal ideal and $x\in C$ be the corresponding point. Set\begin{center}D$^{\ge0}_x(C):=\{D\in$D$(C)\mid Df\subset (x)$ for all $f\in (x)\}$.\end{center}We say that a D$(C)$-module $F$ has {\it regular singularities at $x$} if for any finite-dimensional space $F_0\subset F$, D$^{\ge0}_x(C)F_0$ is a finitely generated $\mathbb C[C]$-module.

Let $C^+$ be a unique  compact connected curve which contains $C$ as an open subset. We say that a D$(C)$-module $F$ has {\it regular singularities} if for any point $x\in C^+$ there exists an affine open set $C(x)\subset C^+$ such that $x\in C(x)$ and $F\otimes_{\mathbb C[C]}\mathbb C[C\cap C(x)]$ considered as a D$(C(x))$-module by extension from $C\cap C(x)$ to $C(x)$ has regular singularities at $x$.
\begin{definition}Let $X$ be a smooth variety and $\EuScript F$ be a coherent $\EuScript D(X)$-module. We say that $\EuScript F$ has {\it regular singularities} if $\EuScript F$ has regular singularities after restriction to any smooth connected affine curve.\end{definition}
\newpage\section{Holonomicity of $(\frak g, \frak k)$-modules of finite type}\label{SHol}
We are now going to prove Theorem~\ref{TGeo} stated in Section~\ref{Stro}.
\begin{proof}[ Proof of Theorem~\ref{TGeo}] Without loss of generality we assume that $\frak g$ is semisimple. We identify $\frak g$ and $\frak g^*$ by use of the Killing form. Let $h\in\frak k$ be a rational semisimple element. We denote by $\frak g_h$ the direct sum of $h$-eigenspaces of $\frak g$ with nonnegative eigenvalues; by $G_h\subset G$ the parabolic subgroup with the Lie algebra $\frak g_h$. Set $Z_G:=G/G_h$. In a similar way we define $\frak k_h, K_h, Z_K$. We denote by $\frak n_h$ the nilpotent radical of $\frak g_h$. Let $e\subset K$ be the unit element. The orbit $Ke\subset Z_G$ is isomorphic to $Z_K$. Put\\$\bullet~G\frak n_h:=\{x\in\frak g\mid x=gn$ for some $n\in\frak n_h, g\in G$\},\\$\bullet~K\frak n_h:=\{x\in\frak g\mid x=kn$ for some $n\in\frak n_h, k\in K$\},\\$\bullet~K\frak n_h\cap\frak k^\bot:=\{x\in\frak g\mid x=kn$ for some $n\in\frak n_h\cap\frak k^\bot, k\in K$\}.\\
Let $\phi:$~T$^*Z_G\to\frak g^*$ be the moment map. It is easy to see that $G\frak n_h$ coincides with $\phi($T$^*Z_G)$, $K\frak n_h$ coincides with $\phi($T$^*Z_G|_{Z_K})$, $K\frak n_h\cap\frak k^\bot$ coincides with $\phi($N$_{Z_K/Z_G}^*)$:
$$\xymatrix{\mathrm T^*Z_G\ar@{->>}[d]^\phi\ar@{<-^)}[r]&\mathrm T^*Z_G|_{Z_K}\ar@{->>}[d]^\phi\ar@{<-^)}[r]&\mathrm N^*_{Z_K/Z_G}\ar@{->>}[d]^\phi\\G\frak n_h\ar@{<-^)}[r]&K\frak n_h\ar@{<-^)}[r]&K\frak n_h\cap\frak k^\bot\hspace{10pt}.}$$
As the variety N$^*_{Z_K/Z_G}$ is isotropic in T$^*Z_G$, the image $\phi($N$^*_{Z_K/Z_G}$) is isotropic in $G(\frak n_h\cap\frak k^\bot)$ and any subvariety $\tilde V\subset K\frak n_h\cap\frak k^\bot$ is isotropic in $G\tilde V$. Therefore by Corollary~\ref{CHM} any subvariety $\tilde V\subset\frak k^\bot\cap$N$_K(\frak g^*)$ is isotropic in $G\tilde V$.\end{proof}
We introduce the following notation:\\$\bullet$~~~~V$^._{\frak g, \frak k}$ is the set of all irreducible components of  intersections of N$_K(\frak k^\bot)$ with all possible $G$-orbits of N$_G(\frak g^*)$.\\$\bullet$ $\EuScript V^._{\frak g, \frak k}$ is the set of all possible irreducible components of the preimages of elements of V$^._{\frak g, \frak k}$ under the moment map T$^*X\to\frak g^*$.\\$\bullet$ $\mathrm L^._{\frak g, \frak k}$ is the set of images in $X$ of all elements of $\EuScript V^._{\frak g, \frak k}$.

Let $M$ be a finitely generated $\frak g$-module which affords a central character $\chi$ and $\lambda\in\hat{\frak h}^*$ (see Subsection~\ref{SSd_g}) be a weight such that $\chi=\chi_\lambda$.
\begin{theorem}\label{Tisohol} If $M$ is a $(\frak g, \frak k)$-module of finite type, then\\a) the irreducible components of V$(M)$ are contained in V$^._{\frak g, \frak k}$;  the irreducible components of $\EuScript V(M)$ are contained in $\EuScript V^._{\frak g, \frak k}$, the irreducible components of L$(M)$ are contained in L$^._{\frak g, \frak k}$.\\b) The module Loc$M$ is holonomic. If $M$ is simple, $M$ is holonomic.\end{theorem}
\begin{proof}Let $\tilde V\subset$V$(M)$ be an irreducible component and $\EuScript Z$ be the closure of $G\tilde V$ in $\frak g^*$. By Theorem~\ref{Gab} the variety $\tilde V$ is coisotropic. On the other hand \begin{center}$\tilde V\subset$N$_K(\frak k^\bot)\cap \EuScript Z$,\end{center} and therefore $\tilde V$ is isotropic. Hence $\tilde V$ is Lagrangian and is an irreducible component of $\EuScript Z\cap$N$_K(\frak g^*)\cap\frak k^\bot$.

As intersections of V$(M)$ with any irreducible component of $M$ are isotropic, \begin{center}dim$\EuScript V$(Loc$M)\le$dim$X$.\end{center} Therefore Loc$M$ is holonomic.

Assume that $M$ is simple. As all irreducible components of V$(M)$ are Lagrangian, $M$ is holonomic.\end{proof}
\begin{corollary}\label{Cholsup} The statements of Theorem~\ref{THol} and Theorem~\ref{TSup} follow from Theorem~\ref{Tisohol}.\end{corollary}
\begin{corollary}\label{Cghol}Let $M$ be a simple $(\frak g, \frak k)$-module of finite type and $\tilde V$ be an irreducible component of V$(M)$. Then \begin{center}dim$\tilde V=\frac{1}{2}$dim~GV$(M)$ and $G\tilde V$=GV$(M)$.\end{center}\end{corollary}
\begin{proof}As $\tilde V$ is Lagrangian in $G\tilde V$, \begin{center}dim$\tilde V=\frac{1}{2}$dim$G\tilde V$\end{center} (see also~\cite{GL}). On the other hand $G\tilde V\subset$GV$(M)$, and \begin{center}dim$\tilde V\ge\frac{1}{2}$dim~GV$(M)$.\end{center} As GV$(M)$ is irreducible, GV$(M)$ coincides with $G\tilde V$.\end{proof}
\newpage\section{Bounded subalgebras of $\frak{sl}(W)$}\label{Sbgk}
We now prove Theorem~\ref{Tbgk} stated in the Introduction. In this
section we assume that $\frak g=\frak{sl}(W)$ for some
finite-dimensional vector space $W$. First we prove
Theorem~\ref{Tcob}.
\begin{proof}[Proof of Theorem~\ref{Tcob}]Let $\EuScript Z$ be an SL$(V)$-orbit open in GV$(M)$. Assume $M$ is a bounded $(\frak{sl}(W), \frak k)$-module. Let $\tilde V$ be an irreducible component of V$(M)$. Then $\tilde V\cap\EuScript Z$ is an open subset of $\tilde V$ and is a conical Lagrangian subvariety of $\EuScript Z$. By the discussion following Example~\ref{Pn}, $\EuScript Z$ is $K$-birationally isomorphic to T$^*$Fl for some partial $W$-flag variety Fl. As $\tilde V$ is a conical Lagrangian subvariety of $\EuScript Z$, $\tilde V$ is birationally isomorphic to N$^*_{Z/\mathrm{Fl}}$ for some smooth subvariety $Z\subset$Fl (Proposition~\ref{NG}). As $\tilde V$ is $K$-spherical, N$^*_{Z/\mathrm{Fl}}$ is $K$-spherical. Therefore Fl is $K$-spherical (Theorem~\ref{Pan}) and T$^*$Fl is $K$-coisotropic. Hence $\EuScript Z$ is $K$-coisotropic.

Assume GV$(M)$ is $K$-coisotropic. Let $\tilde V$ be an irreducible component of V$(M)$ and $\tilde{\EuScript Z}$ be a $G$-orbit open in $G\tilde V$. By the discussion following Example~\ref{Pn}, $\tilde{\EuScript Z}$ is $K$-birationally isomorphic to T$^*$Fl for some partial $W$-flag variety Fl. As $\tilde{\EuScript Z}\subset\overline{\EuScript Z}$, the variety $\tilde{\EuScript Z}$ is $K$-coisotropic by Theorem~\ref{GrO}. Therefore the variety Fl is $K$-spherical and in particular has finitely many $K$-orbits. As $\tilde V\cap\tilde{\EuScript Z}\subset\frak k^\bot$, $\tilde V\cap\tilde{\EuScript Z}$ is isomorphic to an irreducible subvariety of the total space of the conormal bundle to a $K$-orbit in Fl. Therefore \begin{center}dim$\tilde V\le$dim~Fl.\end{center} On the other hand, $\tilde V\cap\tilde{\EuScript Z}$ is coisotropic in $\tilde{\EuScript Z}$ and therefore \begin{center}dim$\tilde V\ge$dim~Fl.\end{center} Therefore dim$\tilde V$=dim~Fl and $\tilde V$ is birationally isomorphic to the total space of the conormal bundle to a $K$-orbit in Fl. Hence $\tilde V$ is $K$-spherical (Theorem~\ref{Pan}).

As all irreducible components of V$(M)$ are $K$-spherical, $M$ is a bounded $K$-module.\end{proof}
\begin{theorem}\label{M_to_Gr} If there exists a simple infinite-dimensional bounded $(\frak{sl}(W), \frak k)$-module $M$, then $\mathrm{Gr}(r; W)$ is a spherical $K$-variety for some $r\in\{1,..., n_W-1\}$.\end{theorem}
\begin{proof}By Theorem~\ref{Tcob} the variety GV$(M)$ is $K$-coisotropic. By the discussion following Example~\ref{Pn} the variety GV$(M)$ is $K$-birationally isomorphic to T$^*$Fl for some partial flag variety Fl. Hence Fl is $K$-spherical and Gr$(r; W)$ is a $K$-spherical variety for some $r$.\end{proof}
\begin{theorem}\label{Gr_to_M}Let Fl be a partial $W$-flag variety. Assume that Fl is $K$-spherical. Then there exists a simple infinite-dimensional multiplicity-free $(\frak{sl}(W), \frak k)$-module.\end{theorem}
\begin{proof}Theorem 6.3 in ~\cite{PS} proves the existence of a simple infinite-dimensional multiplicity-free $(\frak{sl}(W), \frak k)$-module under the assumption that there exists a partial $W$-flag variety for which $K$ has a proper closed orbit on Fl such that the total space of its conormal bundle is $K$-spherical. By~Theorem~\ref{Pan} this latter condition is equivalent to the $K$-sphericity of Fl. It remains to consider the case when Fl has no proper closed $K$-orbits on Fl, i.e. $K$ has only one orbit on Fl. If $K$ has an open orbit on Fl then $[\frak k,\frak k]\cong\frak{sp}(W)$~\cite{On}. However, in this last case Gr$(2; W)$ has a proper closed $K$-orbit and is $K$-spherical.\end{proof}
We have thus proved the following weaker version of Theorem~\ref{Tbgk}.
\begin{corollary}\label{Gr1}A pair $(\frak{sl}(W), \frak k)$ admits an infinite-dimensional bounded simple $(\frak{sl}(W), \frak k)$-module if and only if $\mathrm{Gr}(r; W)$ is a spherical $K$-variety for some $r$.\end{corollary}
\begin{proof} The statement  follows directly from Theorems~\ref{Gr_to_M} and ~\ref{M_to_Gr}.\end{proof}
\begin{corollary}[see also ~\cite{PS}, Conjecture 6.6]If there exists a bounded simple infinite-dimensional $(\frak{sl}(W), \frak k)$-module, then there exists \\a multiplicity-free simple infinite-dimensional $(\frak{sl}(W), \frak k)$-module.\end{corollary}
\begin{proof}The statement follows directly from Corollary~\ref{Gr1}.\end{proof}
As $\mathbb P(W)$ is lower than or cotangent-equivalent to any other
Grassmannian (see discusion following Lemma~\ref{TtoP}), $r$ in
Corollary~\ref{Gr1} can be chosen to equal 1 (see
Proposition~\ref{SphO}). This completes the proof of
Theorem~\ref{Tbgk}.

All finite-dimensional $\frak k$-modules $W$ such that $\mathbb P(W)$ is a $K$-spherical variety are known from the work of C.~Benson and G.~Ratcliff~\cite{BR} (see also~\cite{Kc} and~\cite{Le}). The list of respective pairs $(\frak k, W)$ is reproduced in the Appendix. Theorem~\ref{Tbgk} implies the following.
\begin{corollary} The list of pairs $(\frak{sl}(W), \frak k)$ for which $\frak k$ is reductive and bounded in $\frak{sl}(W)$ coincides with the list of C.~Benson and G.~Ratcliff reproduced in the Appendix.\end{corollary}
\newpage\section{Spherical partial flag varieties}\label{Ssphgr}
We recall that $K$ is a reductive Lie group. Let $V$ be a finite-dimensional $K$-module. We denote by $\frak k_V$ the image of $\frak k$ in End$(V)$.
\begin{definition} A $K$-module $V$ is called {\it weakly irreducible} if it is not a proper direct sum $V_1\oplus V_2$ of two $\frak k$-submodules such that \begin{center}$[\frak k_{V_1}, \frak k_{V_1}]\oplus [\frak k_{V_2}, \frak k_{V_2}]=[\frak k_V, \frak k_V]$.\end{center}\end{definition}
All spherical representations are classified in the work~\cite{BR} (see also~\cite{Le} and ~\cite{Kc}) and we now recall this classification. According to~\cite{BR}, a $K$-module $V$ is $K$-spherical if and only if the pair $([\frak k_V, \frak k_V], V)$ is a direct sum of pairs $(\frak k_i, V_i)$ listed in the Appendix (cf.~\cite{BR}) and in addition \begin{center}$(\frak k_V+_i\frak c_i)=$N$_{\frak{gl}(V)}(\frak k_V+_i\frak c_i)$\end{center} for certain abelian Lie algebras $\frak c_i$ attached to $(\frak k_i, V_i)$ where \begin{center}N$_{\frak{gl}(V)}(\frak k_V+_i\frak c_i)$\end{center} is the normalizer of $\frak k_V+_i\frak c_i$ inside $\frak{gl}(V)$.

Assume that $V$ is a direct sum of two simple modules, $V=V_1\oplus V_2$. For $a,b\in\mathbb C$ we denote by $h_{a,b}$ the rational semisimple element such that \begin{center}$h_{a,b}|_{V_1}=a$Id and $h_{a,b}|_{V_2}=b$Id,\end{center} where Id is the identity map. We use similar notation if $V$ is semisimple of length 1 or 3.
\begin{theorem}\label{KlGr}a) Assume that a partial flag variety Fl$(n_1,..., n_s; V)$ is not cotangent-equivalent to $\mathbb P(V)$. If the partial flag variety Fl$(n_1,...,n_s; V)$ is $K$-spherical, then Fl$(n_1,..., n_s; V)$ is cotangent-equivalent to Fl$(n_1',...,n_s'; V)$ for some datum \begin{center}$(n_1',...,n_s'; [\frak k_V, \frak k_V], V)$\end{center} which appears in the following list.\\
\setcounter{AP}{1}I) Case $s=1$ ('Grassmannians').\\
\Roman{AP}-1) $(r; \frak{sl}_n, \mathbb C^n); (r; \frak{so}_n, \mathbb C^n) (n\ge 3); (r; \frak{sp}_n, \mathbb C^n)$.\\
\Roman{AP}-2-1-1) $(2; \frak{sp}_n\oplus\frak{sl}_m, \mathbb C^n\oplus\mathbb C^m) (m\ge1);$\\
\Roman{AP}-2-1-2) $(2; \frak{sp}_n\oplus\frak{sp}_m, \mathbb C^n\oplus\mathbb C^m);$\\
\Roman{AP}-2-2) $(3; \frak{sl}_n\oplus\frak{sp}_m, \mathbb C^n\oplus\mathbb C^m) ( n\ge 1);$\\
\Roman{AP}-2-3) $(r; \frak{sp}_n, \mathbb C^n\oplus\mathbb C);$\\
\Roman{AP}-2-4) $(r; \frak{sl}_n\oplus\frak{sp}_4, \mathbb C^n\oplus\mathbb C^4) (n\ge 1)$;\\
\Roman{AP}-2-5) $(r; \frak{sl}_n\oplus\frak{sl}_m, \mathbb C^n\oplus\mathbb C^m) (n,m\ge 1)$;\\
\Roman{AP}-3-1-1) $(2; \frak{sl}_n\oplus\frak{sl}_m\oplus\frak{sl}_q, \mathbb C^n\oplus\mathbb C^m\oplus\mathbb C^q) (m,n,q\ge 1)$;\\
\Roman{AP}-3-1-2) $(2; \frak{sl}_n\oplus\frak{sl}_m\oplus\frak{sp}_q, \mathbb C^n\oplus\mathbb C^m\oplus\mathbb C^q) ( m\ge 1, n\ge 1)$;\\
\Roman{AP}-3-1-3) $(2; \frak{sl}_n\oplus\frak{sp}_m\oplus\frak{sp}_q,\mathbb C^n\oplus\mathbb C^m\oplus\mathbb C^q) ( n\ge 1)$;\\
\Roman{AP}-3-1-4) $(2; \frak{sp}_n\oplus\frak{sp}_m\oplus\frak{sp}_q, \mathbb C^n\oplus\mathbb C^m\oplus\mathbb C^q)$;\\
\Roman{AP}-3-2) $(r; \frak{sl}_n\oplus\frak{sl}_m,\mathbb C^n\oplus\mathbb C^m\oplus\mathbb C) (n,m\ge 1)$.\\
\setcounter{AP}{2}II) Case $s\ge 2$.\\
\Roman{AP}-1-1) $(n_1,..., n_s; \frak{sl}_n, \mathbb C^n)$;\\
\Roman{AP}-1-2) $(1,2,3; \frak{sp}_n, \mathbb C^n);$\\
\Roman{AP}-1-3) $(1,r; \frak{sp}_n, \mathbb C^n)$;\\
\Roman{AP}-2-1) $(n_1,..., n_s; \frak{sl}_n,\mathbb C^n\oplus\mathbb C)$;\\
\Roman{AP}-2-2) $(1, r; \frak{sl}_n\oplus\frak{sl}_m, \mathbb C^n\oplus\mathbb C^m) (m,n\ge 1)$;\\
\Roman{AP}-2-3) $(r_1, r_2; \frak{sl}_2\oplus\frak{sl}_n, \mathbb C^2\oplus\mathbb C^n) (n\ge 1)$;\\
\Roman{AP}-2-4) $(1,2; \frak{sl}_n\oplus\frak{sp}_m, \mathbb  F^n\oplus\mathbb C^m) (n\ge 1)$;\\
\Roman{AP}-2-5) $(1,2; \frak{sp}_n\oplus\frak{sp}_m, \mathbb C^n\oplus\mathbb C^m)$.\\
b) For all data $(n_1,..., n_s; \frak k', V)$ from this list there exists a reductive subgroup $K\subset$GL$(V)$ with Lie algebra $\frak k$ such that $\frak k'=[\frak k, \frak k]$ and Fl$(n_1,...,n_s; V)$ is a $K$-spherical variety.\end{theorem}
The rest of this section is devoted to the proof of Theorem~\ref{KlGr}.
\subsection{Simplifications} Let $V_s$ and $V_b$ be $K$-modules of dimensions $n_s, n_b$ such that $n_s, n_b\ge 1$. We start with the following remark. The varieties \begin{center}Hom$(V_s, V_b)$ and Gr$(n_s; V_s\oplus V_b)$\end{center} are birationally isomorphic. This shows that the $K$-module Hom$(V_s, V_b)$ is $K$-spherical if and only if Gr$(n_s; V_s\oplus V_b)$ is $K$-spherical.
\begin{theorem}\label{PTr} Fix $r\in\{1,..., n_b\}$. Let $(L, V_b^r)$ be the pair determined by the datum $(K, V_b, r)$ (see the discussion following Theorem~\ref{Pan}). Then the variety Gr$(r; V)$ is $K$-spherical if and only if the module Hom$(V_b^r, V_s)$ is $L$-spherical and the variety Gr$(r; V_b)$ is $K$-spherical.\end{theorem}
\begin{proof}Assume that Gr$(r; V)$ is $K$-spherical. Then Gr$(r; V_b)$ is $K$-spherical and Gr$(r; V_b^r\oplus V_s)$ is $L$-spherical. The variety Gr$(r; V_b^r\oplus V_s)$ is $L$-spherical if and only if Hom$(V_b^r, V_s)$ is an L-spherical module.

Assume that the module Hom$(V_b^r, V_s)$ is $L$-spherical and the variety Gr$(r; V_b)$ is $K$-spherical. Then Gr$(r; V_b^r\oplus V_s)$ is an $L$-spherical variety. Therefore the variety Gr$(r; V)$ is $K$-spherical.\end{proof}
\begin{corollary}\label{GG} Suppose that the variety Gr$(r; V)$ is $K$-spherical for some $r\in\{2,..., n_b\}$. Then\\a) Hom$(\mathbb C^r, V_s)$ is a spherical GL$_r\times$K-module;\\b) if $n_b\ne 2$, then the variety Fl$(1,2; V_b)$ is $K$-spherical.\end{corollary}
\begin{proof}The first statement is obvious. By Lemma~\ref{Lgr2} we can assume that $r=2$. Let $(L, V_b^2)$ be the pair determined by the datum $(K, V_b, 2)$. As Gr$(2; V_b^2\oplus V_s)$ is an $L$-spherical variety, $\mathbb P(V_b^2\oplus V_s)$ is an $L$-spherical variety by Proposition~\ref{SphO}. Therefore Fl$(1,2; V_b)$ is a $K$-spherical variety.\end{proof}
The following lemma is a first approximation to Theorem~\ref{KlGr}.
\begin{lemma} Let $W$ be a $K$-module. Suppose that $W\otimes\mathbb C^r$ is a spherical $K\times$GL$_r$-module for some $r\in\mathbb Z_{\ge2}$. Then one of the following possibilities holds.\\1) $r=2$ and the datum $([\frak k_W, \frak k_W], W)$ appears in the following list:\\$(\frak{sl}_n\oplus\frak{sl}_m,\mathbb C^n\oplus\mathbb C^m) (n,m\ge 1),\hspace{10pt} (\frak{sl}_n\oplus\frak{sp}_{2m}, \mathbb C^n\oplus\mathbb C^{2m}) (n\ge 1),$\\$(\frak{sp}_{2n}\oplus\frak{sp}_{2m}, \mathbb C^{2n}\oplus\mathbb C^{2m}),\hspace{10pt} (\frak{sp}_{2n}, \mathbb C^{2n}),\hspace{10pt}(\frak{sl}_n,\mathbb C^n)$.\\2) $r=3$ and $(\frak k_W\cap\frak{sl}(W), W)\cong(\frak{sp}_{2n}, \mathbb C^{2n})$.\\3) $r\ge 3$ and the datum $([\frak k_W, \frak k_W], W)$ appears in the following list:\\ $(\frak{sl}_n, \mathbb C^n\oplus\mathbb C) (n \ge 1),\hspace{10pt}(\frak{sl}_n,\mathbb C^n) (n\ge 1), \hspace{10pt} (\frak{sp}_4, \mathbb C^4)$.\end{lemma}
\begin{proof}The result follows directly from the tables of~\cite{BR}.\end{proof}
Let $0<n_1<...< n_s<n_W$ and $0<n_1'<...<n_{s'}'<n_W$ be sequences of integers. By Lemma~\ref{Lgtw} \begin{center}Fl$(n_1, ..., n_s; W)\times$Fl$(n_1',..., n_{s'}'; W)$\end{center} is GL$(W)$-spherical if and only if the variety Fl$(n_1',..., n_{s'}'; W)$ is $K$-spherical, where \begin{center}$K=$GL$_{n_W-n_s}\times...\times$GL$_{n_1}$\end{center} is a Levi subgroup of a parabolic subgroup of GL$(W)$. Let \begin{center}$(d_1,..., d_{s+1})$ and $(d_1',..., d_{s'+1}')$\end{center} be the corresponding to $(n_1,..., n_s)$ and $(n_1',..., n_{s'}')$ partitions of $n_W$.

The following theorem is a particular case of a result of E.~Ponomareva~\cite{Po}; she classified $G$-spherical products of two compact $G$-homogeneous spaces for arbitrary algebraic groups $G$.
\begin{theorem}[see also~\cite{Pan3}]\label{Tpon}The variety \begin{center}Fl$(n_1,..., n_s; W)\times$Fl$(n_1',...,n_{s'}'; W)$\end{center} is GL$(W)$-spherical if and only if the (unordered) pair of sets \begin{center}$(\{d_1,..., d_{s+1}\}, \{d_1',...,d_{s'+1}'\})$\end{center} appears in the following list:$\hspace{10pt}(\{p_1, p_2\}, \{q_1, q_2\});\hspace{10pt}(\{p_1, p_2\}, \{1, q_1, q_2\});$\\$(\{2, p_1\}, \{q_1, q_2, q_3\});\hspace{10pt}(\{1, p_1\}, \{q_1,..., q_{s'+1}\}).$\end{theorem}
\begin{proof}A straightforward computation shows that the pair \begin{center}$(\{1, 1, n\}, \{1, 1, n\})$\end{center} does not yield a spherical variety for $n\ge 1$ (see also Theorem~\ref{PTr}). A straightforward computation shows that the pair \begin{center}$(\{3, n+1\}, \{2, 2, n\})$\end{center} does not yield a spherical variety for $n\ge 2$ (see also Theorem~\ref{PTr}). The rest of the proof is an exercise to Proposition~\ref{SphO}. \end{proof}
\subsection{Simple spherical modules}\label{SSssph}Let $W$ be a simple $K$-module. We recall that by Lemma~\ref{Lgr2} if some partial $W$-flag variety, which is not cotangent-equivalent to $\mathbb P(W)$, is $K$-spherical, then the variety Gr$(2; W)$ is $K$-spherical too.
\begin{lemma}\label{TensP} Let Gr$(2; \mathbb C^m\otimes\mathbb C^n)$ be an SL$_n\times$SL$_m$-spherical variety for some $m,n\in\mathbb Z_{\ge 2}$. Then $m=n=2$.\end{lemma}
\begin{proof}Without loss of generality we assume that $m\ge n$. Assume that $n=2$. Then \begin{center}dim$(\frak b_{\frak{sl}_m}\oplus\frak b_{\frak{sl}_2})\ge$dim~Gr$(2; \mathbb C^m\otimes\mathbb C^2)$.\end{center} We have $\frac{m(m+1)}{2}+1\ge 2(2m-2)$ and therefore $(m-2)(m-5)\ge 0$. Hence $m\in\{2,5,6,7,...\}$.

Assume $m\ge 5$. As Gr$(2; \mathbb C^2\otimes\mathbb C^m)$ is SL$_2\times$SL$_m$-spherical, it must be that Gr$(2; \mathbb C^2\otimes\mathbb C^4)$ is SL$_2\times$SL$_4$-spherical. This is false by dimension reasons.

Assume $n\ge 3$. Then\begin{center}$\frac{14}{18}(m^2+n^2)\ge\frac{m^2+m}{2}-1+\frac{n^2+n}{2}-1+4\hspace{40pt}.$\end{center}As dim$(\frak b_{\frak{sl}_m}\oplus\frak b_{\frak{sl}_n})\ge$dim~Gr$(2; \mathbb C^m\otimes\mathbb C^n)$, we have \begin{center}$\frac{m^2+m}{2}-1+\frac{n^2+n}{2}-1+4\ge 2mn$,\\ $\frac{14}{18}(m^2+n^2)\ge 2mn$ and $m>2n$.\end{center} As Gr$(2; \mathbb C^n\otimes\mathbb C^m)$ is SL$_n\times$SL$_m$-spherical, we have Gr$(2; \mathbb C^n\otimes\mathbb C^{2n})$ is SL$_n\times$SL$_{2n}$-spherical. This implies that SL$_n\times$SL$_2$ is a spherical subgroup of SL$_{2n}$. Hence $n=2$~\cite{Kr}.\end{proof}
\begin{theorem}\label{Tsgr2}Suppose the variety Gr$(2; W)$ is $K$-spherical. Then the pair $(\frak k_W\cap\frak{sl}(W), W)$ is isomorphic to one of the following 3 pairs:\begin{center}$(\frak{sl}(W), W)$, $(\frak{so}(W), W)$, $(\frak{sp}(W), W)$.\end{center}\end{theorem}
\begin{proof}Without loss of generality we assume that $\frak k=\frak k_W\subset\frak{gl}(W)$ and that $\frak k$ is a semisimple Lie algebra. We recall that $B$ is a Borel subgroup of $K$.

A straightforward computation shows that Gr$(2; W)$ is SL$(W)$-spherical, and also SO$(W)$-spherical and SP$(W)$-spherical. In the rest of the proof we assume that $(\frak k_W, W)$ is not isomorphic to \begin{center}$(\frak{sl}(W), W)$, $(\frak{so}(W), W)$ and $(\frak{sp}(W), W)$.\end{center}
1) Assume that $\frak k$ is a simple Lie algebra. If there is an open orbit of $B$ on Gr$(2; W)$ then there is an open orbit of $B\times$GL$_2$ on $W\otimes\mathbb C^2$. The simple modules with an open orbit of a reductive group  are classified in~\cite{SK}. In Table~1 below we reproduce, following~\cite{SK}, all simple $K\times$GL$_2$-modules with an open orbit of $K\times$GL$_2$ such that $W$ is a $K$-spherical module.
\begin{center}{\bf Table 1.} The pairs $([\frak k_W, \frak k_W], W)$ such that $W\otimes\mathbb C^2$ has an open orbit of $K\times$GL$_2$ and $W$ is a spherical $K$-module.\\
$\begin{tabular}{|c|l|c|c|}\hline No&The pair $(\frak k_W, W)$&dim$(\frak b_\frak k\oplus\frak{gl}_2)$&2$n_W$\\
\hline{\bf1}&$(\frak{sl}_n,\mathbb C^n) (n \ge 2)$&$\frac{n(n+1)}{2}+3$&$2n$\\
\hline{\bf2}&$(\frak{so}_n,\mathbb C^n) (n \ge 3)$&$\frac{1}{2}(\frac{n(n-1)}{2}+[\frac{n}{2}])$&$2n$\\
\hline{\bf3}&$(\frak{sp}_{2n}, \mathbb C^{2n}) (n\ge2)$&$n^2+n$&$4n$\\
\hline{\bf4}&$(\frak{sl}_3,$ S$^2\mathbb C^3)$&9&12\\
\hline{\bf5}&$(\frak{sl}_{2n+1}, \Lambda^2\mathbb C^{2n+1})$&$2n^2+3n+4$&$4n^2+2n$\\
\hline{\bf6}&$(\frak{sl}_6, \Lambda^2\mathbb C^6)$&24&30\\
\hline{\bf7}&$(\frak{so}_7,$ Spin$(\mathbb C^7))$&16&16\\
\hline{\bf8}&$(\frak{so}_{10},$ Spin$(\mathbb C^{10}))$&29&32\\
\hline{\bf9}&(G$_2, \mathbb C^7)$&12&14\\
\hline{\bf10}&(E$_6, \mathbb C^{27}$)&42&54\\\hline
\end{tabular}$\end{center}
Here Spin($F$) is any spinor module of SO($F$), and $\mathbb C^7$ in case {\bf9} is the unique faithful G$_2$-module of minimal dimension, $\mathbb C^{27}$ in case {\bf10} is a faithful E$_6$-module of minimal dimension. If $\frak b_\frak k\oplus\frak{gl}_2$ has an open orbit on $W\otimes\mathbb C^2$, then dim$(\frak b_\frak k\oplus\frak{gl}_2)\ge 2n_W$.

Under the assumption that $\frak k$ is simple we complete the proof by the following case-by-case considerations.\\
Case {\bf5}. The computation $4n^2+2n-(2n^2+3n+4)=2n^2-n-4=n^2-4+n(n-1)>0 (n\ge 2)$ shows that dim$(\frak b_\frak k\oplus\frak{gl}_2)< 2n_W$.\\Case {\bf7}. A generic isotropy subalgebra for the action of $\frak{so}_7\oplus\frak{gl}_2$ on Spin$(\mathbb C^7)\otimes\mathbb C^2$ is isomorphic to $\frak{gl}_3$~\cite{SK}. All spherical subgroups of SO$_7$ are not quotients of SL$_3\times\mathbb C^*$~\cite{Kr}. Therefore Spin$(\mathbb C^7)\otimes \mathbb C^2$ has no open B$(\mathrm{SO}_7)\times$GL$_2$-orbit.\\
For the cases {\bf4, 6, 8, 9, 10} from Table~1 we have dim$(\frak b_\frak k\oplus\frak{gl}_2)<2n_W$ and therefore Gr$(2; W)$ has no open $B$-orbit.\\
2) Let $\frak k$ be a direct sum $\frak k_1\oplus\frak k_2$ of 2 noncommutative ideals. Since $W$ is a simple $K$-module, $W$ is isomorphic to $V_1\otimes V_2$, where $V_i (i=1,2)$ are simple $\frak k_i$-modules such that $\frak k_{V_i}=\frak k_i$. As Gr$(2; V_1\otimes V_2)$ is $G$-spherical, Gr$(2; V_1\otimes V_2)$ is SL$(V_1)\times$SL$(V_2)$-spherical too. Since $n_{V_{1,2}}\ge 2$, the only possibility for $(\frak k, W)$ is $(\frak{so}_4, \mathbb C^4)$ by Lemma~\ref{TensP}. In this case Gr$(2; W)$ is a spherical $K$-variety.\end{proof}
\begin{remark}\label{Rsphgr} Fix $r\in\{1,..., n_W-1\}$. Then\\a) the variety Gr$(r; W)$ is SL$(W)$-spherical;\\b) the
variety Gr$(r; W)$ is SO$(W)$-spherical;\\c) the variety Gr$(r; W)$ is SP$(W)$-spherical.\end{remark}
\begin{lemma}\label{SimG}Suppose that Fl$(1,2; W)$ is a $K$-spherical variety. Then the pair $(\frak k_W\cap\frak{sl}(W), W)$ is isomorphic to one of the following pairs:\begin{center}$(\frak{sl}(W), W)$, $(\frak{sp}(W), W)$.\end{center}\end{lemma}
\begin{proof}As Fl$(1,2; W)$ is a $K$-spherical variety, the variety Gr$(2; W)$ is $K$-spherical. Therefore $(\frak k\cap\frak{sl}(W), W)$ is isomorphic to \begin{center}$(\frak{sl}(W), W)$, to $(\frak{so}(W), W)$, or to $(\frak{sp}(W), W)$.\end{center} A generic isotropy subgroup for
the action of SO$_n$ on Fl$(1,2; \mathbb C^n)$ is isomorphic to SO$_{n-2}$. Such a subgroup is not spherical in SO$_n$ for all $n\ge 3$~\cite{Kr}.\end{proof}
\begin{corollary}\label{Cpflso} The only partial flag varieties which are SO$(W)$-spherical are Gr$(r; W)$ for all $r\in\{1,...,n-1\}$.\end{corollary}
\begin{proof}This is a simple corollary of Proposition~\ref{SphO} and of the discussion preceeding this proposition.\end{proof}
\begin{remark}\label{Rpflsl}All partial $W$-flag varieties are SL$(W)$-spherical.\end{remark}
\begin{lemma}Assume that $2|n_W$ and $n_W\ge6$. Then the variety Fl$(2,4; W)$ is not SP$(W)$-spherical.\end{lemma}
\begin{proof}To the datum (SP$(W), W, 4)$ one assigns the datum (SL$_2\times$SL$_2, \mathbb C^2\oplus\mathbb C^2)$. Therefore the variety Fl$(2,4; W)$ is SP$(W)$-spherical if and only if the variety Gr$(2; \mathbb C^2\oplus\mathbb C^2)$ is SL$_2\times$SL$_2$-spherical. However, Gr$(2; \mathbb C^2\oplus\mathbb C^2)$ is not SL$_2\times$SL$_2$-spherical.\end{proof}
We recall that any partial $W$-flag variety is cotangent-equivalent
to one of the following\begin{center}Gr$(r;
W)$,\hspace{10pt}Fl$(1,r; W)$,\hspace{10pt}Fl$(1,2,3;
W)$,\hspace{10pt} Fl$(2, 4; W)$,\end{center}or is higher than Fl$(2,
4; W$), see subsection~\ref{SGr}.
\begin{corollary}\label{Cpflsp}The only SP$(W)$-spherical partial flag varieties, up to cotangent-equivalence, are\begin{center}Gr$(r; W)$,\hspace{10pt}Fl$(1,r; W)$,\hspace{10pt}Fl$(1,2,3; W)$.\end{center}\end{corollary}In particular, all partial $\mathbb C^4$-flag varieties are SP$_4$-spherical.
\subsection{Weakly irreducible spherical modules} Let $W$ be a weakly irreducible, but not irreducible, $K$-spherical module. Then $W$ is a direct sum of two nonzero simple $K$-modules $W_s\oplus W_b$ of dimensions $n_s, n_b$. Without loss of generality we assume that $n_b\ge n_s$.
\begin{lemma} The variety Gr$(2; W)$ is not $K$-spherical.\end{lemma}
\begin{proof} Assume $n_b=2$. Then the pair ($[\frak k, \frak k], W)$ is isomorphic to the pair $(\frak{sl}_2, \mathbb C^2\oplus\mathbb C^2)$. The variety Gr$(2; \mathbb C^2\oplus\mathbb C^2)$ is not GL$_2$-spherical for dimension reasons.

Assume $n_b>2$. Then the variety Fl$(1,2; W_b)$ is $K$-spherical by Corollary~\ref{GG}. Therefore $(\frak k_{W_b}\cap\frak{sl}(W_b), W_b)$ is isomorphic to $(\frak{sl}(W_b), W_b)$ or to $(\frak{sp}(W_b), W_b)$. Then the pair $([\frak k, \frak k], W)$ is isomorphic to \begin{center}$(\frak{sl}_n, \mathbb C^n\oplus\mathbb C^n)$, to $(\frak{sl}_n, \mathbb C^n\oplus(\mathbb C^n)^*)$, or to $(\frak{sp}_{2n}, \mathbb C^{2n}\oplus\mathbb C^{2n})$\end{center} (see the tables in~\cite{BR}). Note, that $n_s=n_b$ in all three cases.

Consider the pair $(\frak{sl}_n,\mathbb C^n\oplus\mathbb C^n)$. If the variety Gr$(2; W)$ is $K$-spherical, Gr$(2; \mathbb C^n\otimes\mathbb C^2)$ must be GL$_n\times$GL$_2$-spherical. This is not the case by Lemma~\ref{TensP}.

Next we consider the pair $(\frak{sp}_{2n}, \mathbb C^{2n}\oplus\mathbb C^{2n})$. If the variety Gr$(2; W)$ is $K$-spherical, Gr$(2; \mathbb C^{2n}\otimes\mathbb C^2)$ must be GL$_{2n}\times$GL$_2$-spherical. This is not the case by Lemma~\ref{TensP}.

Finally consider the pair $(\frak{sl}_n,\mathbb C^n\oplus(\mathbb C^n)^*) (n \ge 3)$. If the variety Gr$(2;W)$ is GL$_n$-spherical, the variety $\mathbb C^2\otimes\mathbb C^n$ is GL$_2$-spherical by Theorem~\ref{PTr}. This is not the case for dimension reasons.\end{proof}
\begin{corollary}\label{DecG} The only $K$-spherical partial $W$-flag varieties for a weakly irreducible, but not irreducible, $K$-spherical module $W$ are \begin{center}Gr$(1; W)$ and Gr$(n_W-1; W)$.\end{center}\end{corollary}
\subsection{On the length of certain modules} Let $V$ be a spherical $K$-module and $W$ be a proper nonzero submodule of $V$.
We recall that by Lemma~\ref{Lgr2}, if some partial $V$-flag variety, which is not cotangent-equivalent to $\mathbb P(V)$, is $K$-spherical, then the variety Gr$(2; V)$ is $K$-spherical too.
\begin{lemma}\label{SPSL} Assume that $W$ is a weakly irreducible submodule of $V$ and Gr$(2; V)$ is a $K$-spherical variety. Then the pair $(\frak{sl}(W)\cap\frak k_W, W)$ is isomorphic \begin{center}to $(\frak{sl}(W), W)$ or to $(\frak{sp}(W), W)$.\end{center}\end{lemma}
\begin{proof}Assume $n_W=2$. Then the pair $(\frak k_W\cap\frak{sl}(W), W)$ is isomorphic to $(\frak{sl}_2, \mathbb C^2)$.

Assume $n_W\ge3$. The variety Fl$(1,2; W)$ is $K$-spherical by Corollary~\ref{GG}. Therefore the pair $(\frak k_W\cap\frak{sl}(W), W)$ is isomorphic to $(\frak{sl}(W), W)$, or to $(\frak{sp}(W), W)$ by Corollary~\ref{DecG} and Lemma~\ref{SimG}.\end{proof}
\begin{lemma}\label{<4}Let $W$ be a $K$-module such that $n_W>3$ and Gr$(2; W)$ is $K$-spherical. Then the length of $W$ as a $\frak k$-module equals at most 3.\end{lemma}
\begin{proof} See Theorem~\ref{Tpon}.\end{proof}
\subsection{Concluding case-by-case considerations} The following theorem is a second approximation to Theorem~\ref{KlGr}.
\begin{theorem} Let $W_b, W_s$ be $K$-modules of the dimensions $n_b, n_s$ such that $n_b\ge n_s\ge 2$. Suppose that Hom($W_s, W_b)$ is a
$K$-spherical module. Then one of the following holds.\\1) $[\frak k_{W_s}, \frak k_{W_s}]=0$, $n_s=2$ and $([\frak k_{W_b}, \frak k_{W_b}], W_b)$ is either $(\frak{sl}(W_b), W_b)$, or $(\frak{sp}(W_b), W_b)$;\\2) $([\frak k_{W_s}, \frak k_{W_s}], W_s)=(\frak{sl}_2, \mathbb C^2)$ and $([\frak k_{W_b}, \frak k_{W_b}], W_b)$ appears in the following list: $(\frak{sl}(W_b), W_b), (\frak{sp}(W_b), W_b),\hspace{10pt} (\frak{sl}_n\oplus\frak{sl}_m, \mathbb C^n\oplus\mathbb C^m) \\(m, n\ge1), (\frak{sl}_n\oplus\frak{sp}_{2m}, \mathbb C^n\oplus\mathbb C^{2m}) (n\ge1),\hspace{10pt} (\frak{sp}_{2n}\oplus\frak{sp}_{2m}, \mathbb C^{2n}\oplus\mathbb C^{2m})$;\\3) $([\frak k_{W_s}, \frak k_{W_s}], W_s)=(\frak{sl}_n,\mathbb C^n) (n \ge 3)$ and $([\frak k_{W_b}, \frak k_{W_b}], W_b)$ appears in the following list: $(\frak{sl}(W_b), W_b), (\frak{sl}_m,\mathbb C^m\oplus\mathbb C), (\frak{sp}_4, \mathbb C^4)$;\\3') $([\frak k_{W_s}, \frak k_{W_s}], W_s)=(\frak{sl}_n, \mathbb C^n\oplus\mathbb C) (n \ge 2)$ and $([\frak k_{W_b}, \frak k_{W_b}], W_b)=(\frak{sl}(W_b), W_b);$\\4) $([\frak k_{W_s}, \frak k_{W_s}], W_s)=(\frak{sl}_3, \mathbb C^3)$ and $([\frak k_{W_b}, \frak k_{W_b}], W_b)=(\frak{sp}(W_b), W_b);$\\5) $([\frak k_{W_s}, \frak k_{W_s}], W_s)=(\frak{sp}_4, \mathbb C^4)$ and $([\frak k_{W_b}, \frak k_{W_b}], W_b)=(\frak{sl}(W_b), W_b)$.\end{theorem}
\begin{proof} As the variety Hom$(W_s, W_b)$ is $K$-spherical, the variety Gr$(n_s; W_s\oplus W_b)$ is $K$-spherical. Therefore the
variety $\mathbb P(W_s\oplus W_b)$ is $K$-spherical and the length of the $K$-module $W_s\oplus W_b$ is not more then 3. Using this condition
and the results of~\cite{BR} it is straightforward to derive the required list. We also use Lemma~\ref{SPSL}.\end{proof}
In the remainder of this section we collect the additional information needed to prove Theorem~\ref{KlGr}.

Let $W$ be a nonzero $K$-module. Denote by $\mathbb C^*_{a, b}$ the one-dimensional subgroup of GL$(W)$ which corresponds to $\mathbb Ch_{a,b}$ (see Section~\ref{Stro}). We use similar notation if $W$ has length 1 or 3.

Fix $r\in\{1,..., n_W-1\}$. Recall that the datum $(K, W, r)$ determines the datum $(L, W^r)$ (see the discussion following Theorem~\ref{Pan}).
\begin{lemma}The following data $(K,W,r)$ determine the pairs $(L, W^r)$ as follows:\\1) (SL$(W), W, r)\to$(GL$_r, \mathbb C^r)$;\\2) (SP$(W), W, 2)\to$(SL$_2, \mathbb C^2)$;\\3) (SP$(W), W, 3)\to$(SL$_2\times\mathbb C^*_{0,1}, \mathbb C^2\oplus\mathbb C)$;\\4) (SP$(W), W, 4)\to$(SL$_2\times$SL$_2, \mathbb C^2\oplus\mathbb C^2)$.\end{lemma}
\begin{proof}We omit the proof.\end{proof}
\begin{lemma} For any $q\in\mathbb Z_{\ge2}$ the following statements are equivalent.\\1) The variety Gr$(2; W\oplus\mathbb C^{2q})$ is $K\times$SP$_{2q}$-spherical.\\2) The variety Gr$(2; W\oplus\mathbb C^2)$ is $K\times$SL$_2$-spherical.\\3) The variety Hom$(\mathbb C^2, W)$ is SL$_2\times K$-spherical.\end{lemma}
\begin{proof} As the datum (SP$_{2q}, \mathbb C^{2q}, 2)$ determines the pair (SL$_2, \mathbb C^2)$, part 1) is equivalent to part 2) by Theorem~\ref{PTr}. Part 3) is a reformulation of part 2).\end{proof}
\begin{lemma}\label{Lsp3} For any $q\in\mathbb Z_{\ge2}$ the following statements are equivalent.\\1) The variety Gr$(3; W\oplus\mathbb C^{2q})$ is $K\times$SP$_{2q}$-spherical.\\2) The variety Gr$(3;W\oplus\mathbb C^2\oplus\mathbb C)$ is $K\times$SL$_2\times\mathbb C^*_{0,0,1}$-spherical.\\3) The variety Hom$(\mathbb C^2\oplus\mathbb C, W)$ is SL$_2\times\mathbb C^*_{0,1}\times K$-spherical.\end{lemma}
\begin{proof} As the datum (SP$_{2q}, \mathbb C^{2q}, 3)$ determines the pair (SL$_2\times\mathbb C^*_{0,1}, \mathbb C^2\oplus\mathbb C)$, part 1) is equivalent to part 2) by Theorem~\ref{PTr}. Part 3) is a reformulation of part 2).\end{proof}
\begin{corollary}\label{SP=1} Suppose that Gr$(3; W\oplus\mathbb C^{2q})$ is a $K\times$SP$_{2q}$-spherical variety for some $q\in\mathbb Z_{\ge2}$. Then $\frak{sl}(W)\subset\frak k_W$.\end{corollary}
\begin{corollary}\label{Cflnesp} The variety Fl$(1, 3; W\oplus\mathbb C^{2q})$ is not a $K\times$SP$_{2q}\times\mathbb C^*_{0,1}$-spherical variety for any $q\in\mathbb Z_{\ge2}$.\end{corollary}
\begin{proof}The variety Fl$(1, 3; W\oplus\mathbb C^{2q})$ is $K\times$SP$_{2q}\times\mathbb C^*_{0,1}$-spherical if and only if the variety Fl$(1, 3; W\oplus\mathbb C^2\oplus\mathbb C)$ is $K\times$GL$_2\times\mathbb C^*_{0,0,1}$-spherical (Lemma~\ref{Lsp3}). The latter is false by Theorem~\ref{Tpon}.\end{proof}
\begin{lemma} For any $q\in\mathbb Z_{\ge3}$ the following statements are equivalent.\\1) The variety Gr$(4; W\oplus\mathbb C^{2q})$ is $K\times$SP$_{2q}\times\mathbb C^*_{0,1}$-spherical.\\2) The variety Gr$(4; W\oplus\mathbb C^2\oplus\mathbb C^2)$ is $K\times$SL$_2\times$SL$_2\times\mathbb C^*_{0,1,1}$-spherical.\\3) The variety Hom$(\mathbb C^2\oplus\mathbb C^2, W)$ is SL$_2\times$SL$_2\times \mathbb C^*_{1,1}\times K$-spherical.\end{lemma}
\begin{proof}As the datum (SP$_{2q}, \mathbb C^{2q}, 4)$ determines the pair (SL$_2\times$SL$_2\times\mathbb C^*_{1,1}, \mathbb C^2\oplus\mathbb C^2)$, part 1) is equivalent to part 2) by Theorem~\ref{PTr}. Part 3) is a reformulation of part 2).\end{proof}
\begin{corollary}\label{SPG4} Suppose that Gr$(4; W\oplus\mathbb C^{2q})$ is a $K\times$SP$_{2q}$-spherical variety for some $q\in\mathbb Z_{\ge3}$. Then either $n_W=2$ and $\frak{sl}(W)\subset\frak k_W$, or $n_W=1$.\end{corollary}
\begin{proposition}\label{Pklgr} Let $V$ be a $K$-module.\\a) Suppose that Gr$(r; V)$ is a $K$-spherical variety for some $r\in\{ 2,..., [\frac{n_V}{2}]\}$. Then the datum $$(r; [\frak k_V, \frak k_V], V)\eqno (3)$$ appears in the following list:\\
1) $(r; \frak{sl}_n, \mathbb C^n), (r; \frak{so}_n, \mathbb C^n) (n\ge 3), (r; \frak{sp}_n, \mathbb C^n)$;\\
2-1-1) $(2; \frak{sp}_n\oplus\frak{sl}_m, \mathbb C^n\oplus\mathbb C^m) (m\ge 1);$\\
2-1-2) $(2; \frak{sp}_n\oplus\frak{sp}_m, \mathbb C^n\oplus\mathbb C^m);$\\
2-2) $(3; \frak{sl}_n\oplus\frak{sp}_m, \mathbb C^n\oplus\mathbb C^m) ( n\ge 1);$\\
2-3) $(r; \frak{sp}_n, \mathbb C^n\oplus\mathbb C);$\\
2-4) $(r; \frak{sl}_n\oplus\frak{sp}_4, \mathbb C^n\oplus\mathbb C^4) (n\ge 1)$;\\
2-5) $(r; \frak{sl}_n\oplus\frak{sl}_m, \mathbb C^n\oplus\mathbb C^m) (n,m\ge 1)$;\\
3-1-1) $(2; \frak{sl}_n\oplus\frak{sl}_m\oplus\frak{sl}_q, \mathbb C^n\oplus\mathbb C^m\oplus\mathbb C^q) (m,n,q\ge 1)$;\\
3-1-2) $(2; \frak{sl}_n\oplus\frak{sl}_m\oplus\frak{sp}_q, \mathbb C^n\oplus\mathbb C^m\oplus\mathbb C^q) ( m\ge 1, n\ge 1)$;\\
3-1-3) $(2; \frak{sl}_n\oplus\frak{sp}_m\oplus\frak{sp}_q,\mathbb C^n\oplus\mathbb C^m\oplus\mathbb C^q) ( n\ge 1)$;\\
3-1-4) $(2; \frak{sp}_n\oplus\frak{sp}_m\oplus\frak{sp}_q, \mathbb C^n\oplus\mathbb C^m\oplus\mathbb C^q) $;\\
3-2) $(r; \frak{sl}_n\oplus\frak{sl}_m,\mathbb C^n\oplus\mathbb C^m\oplus\mathbb C) (n,m\ge 1)$.\\
b) For all data $(r; \frak k', V)$ from the list there exists a reductive subgroup $K\subset$GL$(V)$ with Lie algebra $\frak k$ such that Gr$(r; V)$ is $K$-spherical and $\frak k'=[\frak k, \frak k]$.\end{proposition}
\begin{proof}Let $V$ be a direct sum of nonzero simple $\frak k$-modules \begin{center}$V_1\oplus V_2\oplus... \oplus V_l$.\end{center}By Lemma~\ref{Lgr2} the variety Gr$(2; V)$ is $K$-spherical and therefore $l\le 3$ by Lemma~\ref{<4}.

If $l=1$, the $K$-module $V$ is simple. This case is considered in Subsection~\ref{SSssph} (Theorem~\ref{Tsgr2} and Remark~\ref{Rsphgr}).

In the rest of the proof we assume that $l\ge 2$, i.e. $l=2, 3$. By Lemma~\ref{SPSL}, for all $i$, \begin{center}$[\frak k_{V_i}, \frak k_{V_i}]=\frak{sl}(V_i)$, or $[\frak k_{V_i}, \frak k_{V_i}]=\frak{sp}(V_i)$ and $n_{V_i}\ge 4$.\end{center}
If the variety Gr$(r; V)$ is $K$-spherical then the pair of sets \begin{center}$(\{r, n_V-r\}, \{n_{V_1},..., n_{V_l}\})$\end{center} appears in the list of Theorem~\ref{Tpon}. All entries of the required list with $r=2$ are recovered by the list of Theorem~\ref{Tpon}.

In the rest of the proof we assume that $r\ge 3$. It remains to show that the datum (3) is as in cases 2-2, 2-3, 2-4, 2-5 or 3-2. If $\frak{sl}(V_i)\subset\frak k_{V_i}$ for all $i$, then Theorem~\ref{Tpon} implies that we are in cases 2-5 or 3-2.

In the rest of the proof we assume that \begin{center}$[\frak k_{V_1}, \frak k_{V_1}]=\frak{sp}(V_1)$ and $n_{V_1}\ge 4$.\end{center}

As $r\ge 3$, Gr$(3; V)$ is a $K$-spherical variety by Proposition~\ref{SphO}. Then \begin{center}$\frak{sl}(V_2\oplus...\oplus V_s)\subset\frak k_{V_2\oplus...\oplus V_s}$,\end{center} in particular, $l=2$. For $r=3$, this is case 2-2 from the list.

In the rest of the proof we assume that $r\ge 4$. Case 2-4 corresponds to $n_{V_1}=4$. Assume  now that $n_{V_1}\ge 6$. As Gr$(r; V)$ is a $K$-spherical variety, Gr$(4; V)$ is a $K$-spherical variety. Therefore $n_{V_2}\le 2$ by Corollary~\ref{SPG4}. A straightforward computation shows that Gr$(4; V_1\oplus\mathbb C^2)$ is not an SP$(V_1)\times$GL$_2$-spherical variety for all $n_{V_1}$ such that $n_{V_1}\ge 6$. Therefore $n_{V_2}=1$ and we are in case 2-3.

This proves a). Part b) is straightforward.\end{proof}
We are now ready to prove Theorem~\ref{KlGr}.
\begin{proof}[Proof of theorem~\ref{KlGr}]Let $V$ be a direct sum of simple nonzero $\frak k$-modules \begin{center}$V_1\oplus V_2\oplus... \oplus V_l$.\end{center} If $l=1$, the $\frak k$-module $V$ is simple. This case is considered in Subsection~\ref{SSssph} (Remark~\ref{Rpflsl}, Corollary~\ref{Cpflso}, Corollary~\ref{Cpflsp}).

In the rest of the proof we assume that $l\ge 2$. The claim of Theorem~\ref{KlGr} for $s=1$ is established in Proposition~\ref{Pklgr}.

In the rest of the proof we also assume that $s\ge 2$. By Lemma~\ref{Lgr2} the variety Gr$(2; V)$ is $K$-spherical and $l\le 3$ by Lemma~\ref{<4}. Lemma~\ref{SPSL} implies, for all $i$, that\begin{center}$[\frak k_{V_i}, \frak k_{V_i}]=\frak{sl}(V_i)$, or $[\frak k_{V_i}, \frak k_{V_i}]=\frak{sp}(V_i)$ and $n_{V_i}\ge 4$.\end{center}

If the variety Fl$(n_1,...,n_s; V)$ is $K$-spherical then the pair of sets \begin{center}$(\{n_1,...,n_V-n_s\}, \{n_{V_1},..., n_{V_l}\})$\end{center} appears in the list of Theorem~\ref{Tpon}. If $\frak{sl}(V_i)\subset\frak k_{V_i}$ for all $i$, then Theorem~\ref{Tpon} implies that we are in cases \setcounter{AP}{2}\Roman{AP}-1-1, \Roman{AP}-2-1, \Roman{AP}-2-2, \Roman{AP}-2-3.

In the rest of the proof we assume that $[\frak k_{V_1}, \frak k_{V_1}]=\frak{sp}(V_1)$ and $n_{V_1}\ge 4$. It remains to show that the datum (1) is as in cases \Roman{AP}-1-2, \Roman{AP}-1-3, \Roman{AP}-2-4, \Roman{AP}-2-5.

We have now $l\le 2$ by Theorem~\ref{Tpon} and therefore $l=2$. The variety Fl$(n_1,..., n_s; V)$ is cotangent-equivalent to \begin{center}Fl$(1, 2; V)$, or to Fl$(1, 3; V)$,\end{center} or is higher than Fl$(1, 3; V)$ (see Subsection~\ref{SGr}). As the variety Fl$(1,3; V)$ is not $K$-spherical by Corollary~\ref{Cflnesp}, Fl$(n_1,..., n_s; V)$ is cotangent-equivalent to Fl($1,2; V$) and we are in cases \Roman{AP}-2-4, \Roman{AP}-2-5.

This proves a). Part b) is straightforward.\end{proof}
\newpage\section{Joseph ideals}\label{SJId}
In this section $\frak g$ is semisimple. Let $\lambda\in\frak h_\frak g^*$ be a dominant weight. A a two-sided ideal $I$ of U$(\frak g)$ determines the submodule $I$M$_\lambda$ of M$_\lambda$. This defines a map $\alpha$ from the latice of ideals of U$(\frak g)$ to the latice of submodules of M$_\lambda$.
\begin{theorem}[A.~Joseph{~\cite[Prop.~4.3]{Jo2}}] The map $\alpha$ is an embedding of the latice of two-sided ideals $I\subset$U$(\frak g)$ such
that Ker$\chi_\lambda\subset I$ to the latice of submodules of M$_\lambda$. If $\lambda$ is regular, $\alpha$ is an isomorphism.\end{theorem}
\begin{remark}The two-sided ideals of U$(\frak g)$ have been studied extensively in the last 30 years. This section adapts some known results to the setup of the dissertation. We thank Anthony Joseph and Anna Melnikov for useful comments and references.\end{remark}
\subsection{The case $\frak g=\frak{sl}(W)$} Let $\EuScript Q\subset\frak{sl}(W)^*$ be the non-trivial nilpotent orbit of minimal dimension. This nilpotent orbit consists of matrices of rank 1. We represent such matrices as $l\otimes v$, for some $v\in W$ and $l\in W^*$ with $l(v)=0$.
\begin{definition} Let $I$ be a primitive ideal of U$(\frak{sl}(W))$. We call $I$ a {\it Joseph ideal} whenever V$(I)=\bar{\EuScript Q}$.\end{definition}
\begin{lemma} Let $I$ be a Joseph ideal and let $\tilde L$ be a finitely generated $(\frak{sl}(W), \frak b^\frak{sl}_W)$-module such that $I\subset$Ann$\tilde L$. Then $\tilde L$ is an $(\frak{sl}(W), \frak h^{\frak{sl}}_W)$-bounded module.\end{lemma}
\begin{proof} As $\tilde L$ is an $(\frak{sl}(W), \frak b^{\frak{sl}}_W)$-module, V($\tilde L)\subset \bar{\EuScript Q}\cap(\frak b^{\frak{sl}}_W)^\bot$. We have\begin{center}
$\bar{\EuScript Q}\cap(\frak b^{\frak{sl}}_W)^\bot=\{x\in\frak{sl}(W)^*\mid x=l\otimes v$ for some $v\in W$ and $l\in W^*$ such that $l(bv)=0$ for all $b\in\frak b^{\frak{sl}}_W\}$.\end{center}
The variety $\bar{\EuScript Q}\cap(\frak b^{\frak{sl}}_W)^\bot$ has irreducible components $\{\EuScript Q_i\}_{i\le n_W-1}$, where\begin{center}$\EuScript Q_i:=\{x\in\frak{sl}(W)^*\mid x=l\otimes v$ for some $v\in$span$\langle$e$_1,..., $e$_i\rangle$ and $l\in$span$\langle$e$_{i+1}^*, ..., $e$_{n_W}^*\rangle$\}.\end{center}
The dimension of $\EuScript Q_i$ equals $n_W-1$ for any $i\le n_W-1$. Let H$^\frak{sl}_W$ be the connected subgroup of SL$(W)$ with Lie algebra $\frak h^\frak{sl}_W$. The varieties $\EuScript Q_i$ are H$^\frak{sl}_W$-stable and are H$^\frak{sl}_W$-spherical for all $i\le n_W-1$. Therefore $\tilde L$ is a bounded $(\frak{sl}(W), \frak h^\frak{sl}_W$)-module by Proposition~\ref{Pbm}.\end{proof}
\begin{corollary}\label{Cqtoso}a) Let $I$ be a Joseph ideal and let $\lambda\in{\frak h^\frak{sl}_W}^*$ be a weight such that $I=$Ann~L$_\lambda$. Then L$_\lambda$ is a bounded $\frak h^\frak{sl}_W$-module. In particular, $\bar\lambda$ is a semi-decreasing tuple.\\b) If $\bar\lambda$ is a semi-decreasing $n_W$-tuple then I$(\lambda)$ is a Joseph ideal.\end{corollary}
\begin{proof}Part a) is trivial. We have \begin{center}dim~GV$(M)\le$2dim~V$(M)$\end{center} by Theorem~\ref{Tber}. As $\bar\lambda$ is semi-decreasing, L$_\lambda$ is bounded and therefore~V(L$_\lambda)$ is H$_W^\frak{sl}$-spherical. Hence dim~GV$(M)\le 2(n_W-1)$. The only nilpotent orbit which satisfies this inequality is $\EuScript Q$.\end{proof}
\begin{corollary}[{~\cite[Table~3]{Jo3}}]\label{CJid} Let $I\subset$U$(\frak{sl}(W))$ be a two-sided ideal. Then $I$ is a Joseph ideal if and only if one of the following conditions holds:\\a) there exists a semi-integral semi-decreasing tuple $\bar\lambda$ such that $I=$I$(\lambda)$;\\b) there exists a singular integral semi-decreasing tuple $\bar\lambda$ such that $I$=I$(\lambda)$;\\c) there exists a regular integral semi-decreasing tuple $\bar\lambda$ and a number $k\in\{1,..., n_W-1\}$ such that $I$=I$($s$_k\lambda$).\end{corollary}
\begin{proof}Follows from Corollary~\ref{Cqtoso} and the discussion about coherent families in Subsection~\ref{SM}.\end{proof}
\begin{lemma} Let $k\in\{1, 2,..., n_W-1\}$ and $\lambda$ be a regular integral dominant weight. Then Ann~L$_{\mathrm s_k\lambda}$ is a submaximal ideal of U$(\frak{sl}(W))$, i.e. for any two ideals $I_1$, $I_2$ such that Ann~L$_{\mathrm s_k\lambda}\subsetneq I_1, I_2$ we have $I_1=I_2$.\end{lemma}
\begin{proof} The Verma module M$_\lambda$ has a unique maximal submodule rad$^1$M$_\lambda$ such that M$_\lambda/$rad$^1$M$_\lambda$ is a simple module L$_\lambda$. Let rad$^2$M$_\lambda$ be the minimal submodule of rad$^1$M$_\lambda$ such that rad$^1$M$_\lambda/$rad$^2$M$_\lambda$ is a semisimple $\frak{sl}(W)$-module. The module M$_\lambda/$rad$^1$M$_\lambda$ is finite-dimensional and rad$^1$M$_\lambda/$rad$^2$M$_\lambda$ has a unique submodule isomorphic to L$_{\mathrm s_k\lambda}$. Let rad$^{1+\varepsilon_k}$M$_\lambda$ be a the submodule of rad$^1$M$_\lambda$ such that rad$^1$M$_\lambda/$rad$^{1+\varepsilon_k}$M$_\lambda=$L$_{\mathrm s_k\lambda}$. Let $I$ be the corresponding to rad$^{1+\varepsilon_k}$M$_\lambda$ two-sided ideal. Then $I$ is submaximal and $I\subset$Ann~L$_{\mathrm s_k\lambda}$. As the unique ideal which is larger then $I$ is of finite codimension, we have Ann~L$_{\mathrm s_k\lambda}=I$. Therefore the ideal Ann~L$_{\mathrm s_k\lambda}$ is submaximal.\end{proof}
Let $\bar\lambda$ be an $n_W$-tuple. We can consider the algebra U$(\frak{sl}(W))/$I$(\lambda)$ as an $(\frak{sl}(W)_l\oplus\frak{sl}(W)_r)$-module, where both $\frak{sl}(W)_l$ and $\frak{sl}(W)_r$ are isomorphic to $\frak{sl}(W)$ and the action is given by \begin{center} $(g_1, g_2)m=g_1m-mg_2$\end{center}for $g_1, g_2\in\frak{sl}(W)$ and $m\in$U$(\frak{sl}(W))/$I$(\lambda)$.
Fix $i\in\{1,..., n_W-1\}$. The action of \begin{center}$\frak{sl}(W)_d:=\{(g, g)\}_{g\in\frak{sl}(W)}\subset\frak{sl}(W)_l\oplus\frak{sl}(W)_r$\end{center} on U$(\frak{sl}(W))/$I(s$_i\rho)$ is locally finite. Therefore U$(\frak{sl}(W))/$I(s$_i\rho)$ is an $(\frak{sl}(W)_l\oplus\frak{sl}(W)_r, \frak{sl}(W)_d)$-module. This module affords the central character ($\chi_\rho, \chi_\rho)$. The category of $(\frak{sl}(W)_l\oplus\frak{sl}(W)_r, \frak{sl}(W)_d)$-modules which afford the central character $(\chi_\rho, \chi_\rho)$ is equivalent to the category of $(\frak{sl}(W), \frak b^\frak{sl}_W)$-modules which afford the central character $\chi_\rho$. The equivalence is given by the functor\begin{center}
H$_{\mathrm M_\rho}: \frak{sl}(W)_l\oplus\frak{sl}(W)_r-$mod$\to \frak{sl}(W)-$mod,\\$ M\mapsto M\otimes_{\mathrm U(\frak{sl}(W)_r)}$M$_\rho$,\end{center}
see~\cite[Prop. 5.9]{BeG}. The action of the $\frak{sl}(W)_l$-projective functors on the categories of $(\frak{sl}(W)_l\oplus\frak{sl}(W)_r, \frak{sl}(W)_d)$-modules and $(\frak{sl}(W), \frak b^\frak{sl}_W)$-modules commutes with the functor H$_{\mathrm M_\rho}$.

Fix $i, j\in\{1,..., n_W-1\}$. Let $P_j$ be the maximal parabolic subgroup of SL$(V)$ for which the space span$\langle$e$_r\rangle_{r\le j}$ is stable and $\frak p_j$ be the Lie algebra of $P_j$. We have \begin{center}F$_W\EuScript H_\rho^\rho=\Sigma_{i\le n_W}\EuScript H_\rho^{\rho+\varepsilon_i}$\end{center}and\begin{center}$\EuScript H_\rho^{\mathrm s_i\rho}=\EuScript H^\rho_{\rho+\varepsilon_{i+1}}\EuScript H^{\rho+\varepsilon_{i+1}}_\rho=\EuScript H_{\rho+\varepsilon_{i+1}}^{\mathrm s_i\rho}\EuScript H_\rho^{\rho+\varepsilon_{i+1}}$\end{center} for all $i\le n_W-1$.
\begin{lemma} Assume $i\ne j$. Then \begin{center}$\EuScript H_\rho^{\rho+\varepsilon_{i+1}}($M$_\rho/$rad$^{1+\varepsilon_j}$M$_\rho)=0$ and $\EuScript H_\rho^{\mathrm s_i\rho}($M$_\rho/$rad$^{1+\varepsilon_j}$M$_\rho)=0$.\end{center}\end{lemma}
\begin{proof} The module M$_\rho/$rad$^{1+\varepsilon_j}$M$_\rho$ is an $(\frak{sl}(W), \frak p_j)$-module. Therefore \begin{center}$\EuScript H_\rho^{\rho+\varepsilon_{i+1}}($M$_\rho/$rad$^{1+\varepsilon_j}$M$_\rho)$\end{center} is an $(\frak{sl}(W), \frak p_j)$-module. If $\EuScript H_\rho^{\rho+\varepsilon_{i+1}}($M$_\rho/$rad$^{1+\varepsilon_j}$M$_\rho)\ne 0$, then some simple quotient of $\EuScript H_\rho^{\rho+\varepsilon_{i+1}}$M$_\rho$ is an $(\frak{sl}(W), \frak p_j)$-module. If $i\ne j$, a unique simple quotient of \begin{center}$\EuScript H_\rho^{\rho+\varepsilon_{i+1}}$M$_\rho=$M$_{\rho+\varepsilon_{i+1}}$\end{center} is not an $(\frak{sl}(W), \frak p_j)$-module, i.e. the action of $\frak p_j$ on L$_{\rho+\varepsilon_{i+1}}$ is not locally finite. Therefore \begin{center}$\EuScript H_\rho^{\rho+\varepsilon_{i+1}}($M$_\rho/$rad$^{1+\varepsilon_j}$M$_\rho)=0$ and $\EuScript H_\rho^{\mathrm s_i\rho}($M$_\rho/$rad$^{1+\varepsilon_j}$M$_\rho)=0$.\end{center}\end{proof}
As [P$_{\mathrm s_i\rho}] =[$M$_\rho]+[$M$_{\mathrm s_i\rho}]$, the projective functor $\EuScript H_\rho^{\mathrm s_i\rho}$ is identified with s$_i$+1 under the identification of PF$\overline{\mathrm{unc}}(\chi_\rho)$ with $\mathbb Z[$S$_{n_W}]$. As the set $\{$s$_i+1\}_{i\le n_W-1}$ generates $\mathbb Z[$S$_{n_W}]$, the set $\{\EuScript H_\rho^{\mathrm s_i\rho}\}_{i\le n_W-1}$ generates the algebra PF$\overline{\mathrm{unc}}(\chi_\rho)$.
\begin{corollary}\label{Casn}Let $M$ be an $\frak{sl}(W)$-module. Suppose that $M$ is annihilated by a submaximal ideal I$($s$_j\rho)$ for
some $j\le n_W-1$. Then \begin{center}$\EuScript H_\rho^{\rho+\varepsilon_{i+1}}M=0$ for all $i\ne j$.\end{center}\end{corollary}
\begin{remark}The notion of left cell~\cite{Jo4} (see also~\cite{Vo2}) relates nilpotent orbits with Weyl group modules. Moreover, the ordering of primitive ideals is related to left cells via tensoring by finite-dimensional representations. As the projective functors are building blocks of the functors of tensoring with a finite-dimensional representation, the action of the projective functors on a primitive ideal $I$ is closely related~\cite{Jo5} with the nilpotent orbit V$(I)$ (see also~\cite{Jo6}). Corollary~\ref{Casn} is an explicit example of this relationship.\end{remark}
\begin{proposition} Let $M$ be an infinite-dimensional finitely generated $\frak{sl}(W)$-module which affords the generalized central character $\chi_\rho$. Then there exists $i\le n_W-1$ such that $\EuScript H_\rho^{\rho+\varepsilon_{i+1}}M\ne 0$.\end{proposition}
\begin{proof} Without loss of generality we assume that $M=\frak{sl}(W)M$. In~\cite{BeBe2} the authors mention that there exists a Borel subalgebra $\frak b\subset\frak{sl}(W)$ with nilpotent radical $\frak n$ such that $(M/\frak nM)^\frak b\ne 0$. We have$$(M/\frak nM)^\frak b=(\mathrm M_\rho\otimes_\mathbb CM)/(\frak{sl}(W)(\mathrm M_\rho\otimes_\mathbb CM))\ne 0\eqno(4).$$ As the sequence\begin{center}$\oplus_{i\le n_W-1}$P$_{\mathrm s_i\rho}\to$M$_\rho\to $L$_\rho\to 0$\end{center}is exact, the sequence\begin{center}$\frac{\oplus_{i\le n_W-1}\mathrm P_{\mathrm s_i\rho}\otimes_\mathbb CM}{\frak{sl}(W)\oplus_{i\le n_W-1}\mathrm P_{\mathrm s_i\rho}\otimes_\mathbb CM}\to\frac{\mathrm M_\rho\otimes_\mathbb CM}{\frak{sl}(W)(\mathrm M_\rho\otimes_\mathbb CM)}\to\frac{\mathrm L_\rho\otimes_\mathbb CM}{\frak{sl}(W)(\mathrm L_\rho\otimes_\mathbb CM)}\to 0$\end{center}is exact. As \begin{center}(L$_\rho\otimes_\mathbb CM)/\frak{sl}(W)($L$_\rho\otimes_\mathbb CM)= M/\frak{sl}(W)M$=0,\end{center}formula (4) implies\begin{center} (P$_{\mathrm s_i\rho}\otimes_\mathbb CM)/\frak{sl}(W)($P$_{\mathrm s_i\rho}\otimes_\mathbb CM)\ne 0$\end{center} for some $i\le n_W-1$. A straightforward computation shows that\begin{center}$\frac{\mathrm P_{\mathrm s_i\rho}\otimes_\mathbb CM}{\frak{sl}(W)(\mathrm P_{\mathrm s_i\rho}\otimes_\mathbb CM)}=$\\$\frac{\EuScript H_\rho^{\rho+\varepsilon_{i+1}}\mathrm M_\rho\otimes_\mathbb C\EuScript H_\rho^{\rho+\varepsilon_{n_w+1-i}}M}{\frak{sl}(W)(\EuScript H_\rho^{\rho+\varepsilon_{i+1}}\mathrm M_\rho\otimes_\mathbb C\EuScript H_\rho^{\rho+\varepsilon_{n_w+1-i}}M)}=\frac{\mathrm M_\rho\otimes_\mathbb C\EuScript H_\rho^{\mathrm s_{n_W-i}\rho}M}{\frak{sl}(W)(\mathrm M_\rho\otimes_\mathbb C\EuScript H_\rho^{\mathrm s_{n_W-i}\rho}M)}$\hspace{10pt}.\end{center}Therefore $\EuScript H_\rho^{\mathrm s_{n_W-i}\rho}M\ne0$.\end{proof}
\subsection{The case $\frak g=\frak{sp}(W\oplus W^*)$}\label{SSJisp} Let $\EuScript Q_\frak{sp}\subset\frak{sp}(W\oplus W^*)^*$ be the non-trivial nilpotent orbit of minimal dimension. This nilpotent orbit consists of matrices of rank 1. We represent such matrices as $v\otimes v$ with $v\in W\oplus W^*$.
\begin{definition} Let $I$ be a primitive ideal of U$(\frak{sp}(W\oplus W^*))$. We call $I$ a {\it Joseph ideal} whenever V$(I)=\bar{\EuScript Q}_\frak{sp}$.\end{definition}
\begin{lemma} Let $I$ be a Joseph ideal and let $\tilde L$ be a simple $(\frak{sp}(W\oplus W^*), \frak b^\frak{sp}_W)$-module. Then $\tilde L$ is an $(\frak{sp}(W\oplus W^*), \frak h_W)$-bounded module.\end{lemma}
\begin{proof} As $\tilde L$ is an $(\frak{sp}(W\oplus W^*), \frak b^\frak{sp}_W)$-module, we have V$(\tilde L)\subset\bar{\EuScript Q}_\frak{sp}\cap(\frak b^\frak{sp}_W)^\bot$. Furthermore,\begin{center}
$(\frak b^\frak{sp}_W)^\bot\cap\bar{\EuScript Q}_\frak{sp}=\{x\in\frak{sp}(W\oplus W^*)^*\mid x=v\otimes v$ for some $v\in W\oplus W^*$ such that $(v, bv)=0$ for all $b\in\frak b^\frak{sp}_W\}$.\end{center}
The variety $(\frak b^\frak{sp}_W)^\bot\cap\bar{\EuScript Q}_\frak{sp}$ has the unique irreducible component $\EuScript Q_1$, where
\begin{center}$\EuScript Q_1:=\{x\in\frak{sp}(W\oplus W^*)^*\mid x=v\otimes v$ for some $v\in$span$\langle$e$_1, ..., $e$_{n_W} \rangle\}$.\end{center}
The dimension of $\EuScript Q_1$ equals to $n_W$. Let H$_W$ be an irreducible subgroup of SP$(W\oplus W^*)$ with Lie algebra $\frak h_W$. The variety $\EuScript Q_1$ is H$_W$-stable and is H$_W$-spherical. Therefore $\tilde L$ is a bounded $(\frak{sp}(W\oplus W^*), \frak h_W)$-module  by Proposition~\ref{Pbm}.\end{proof}
\begin{corollary}\label{Cqtosw}a) Let $I$ be a Joseph ideal and let $\bar\mu$ be an $n_W$-tuple such that $I=$Ann~L$_\mu$. Then L$_\mu$ is a bounded $\frak h_W$-module. In particular, $\bar\mu$ is a Shale-Weil tuple.\\b)Let $\bar\mu$ be a Shale-Weil $n_W$-tuple then I$_{\frak{sp}}(\mu)$ is a Joseph ideal.\end{corollary}
\begin{proof}Part a) is trivial. We have dim~GV$(M)\le$2dim~V$(M)$ by Theorem~\ref{Tber}. As $\bar\mu$ is Shale-Weil, L$_\mu$ is bounded and therefore~V(L$_\mu)$ is H$_W$-spherical. Hence dim~GV$(M)\le 2n_W$. The only nilpotent orbit which satisfies this inequality is $\EuScript Q_\frak{sp}$.\end{proof}
We recall that $\sigma$ stands to the reflection with respect to the simple long root of $\frak{sp}(W\oplus W^*)$ for our fixed choice of a Borel subalgebra, see Subsection~\ref{SM}.
\begin{lemma} Let $\bar\mu$ be a Shale-Weil tuple. Then Ann~L$_\mu$ is a maximal ideal of U$(\frak{sp}(W\oplus W^*))$.\end{lemma}
\begin{proof} Without loss of generality we assume that $\bar\mu$ is positive, in particular $\mu$ is dominant. Let $I$ be an ideal of U$(\frak{sp}(W\oplus W^*))$ such that I$(\mu)\subsetneq I$. Then I$(\mu)$M$_\mu\subsetneq I$M$_\mu$. As $I$ annihilates M$_\mu/I$M$_\mu$, we have $I\subset$Ann~L$_{\mu'}$ for some simple $\frak h_W$-bounded module L$_{\mu'}$. Therefore either $\mu'=\mu$ or $\mu'=\sigma\mu$. Hence $I\subset$I$(\mu)$ and $I$=I$(\mu$).\end{proof}
Let $\bar\mu$ be a Shale-Weil $n_W$-tuple. Let $\bar\mu'$ be an $n_W$-tuple and let $\mu'\in\frak h_W^*$ be the corresponding weight.
\begin{lemma} Suppose $\EuScript H_\mu^{\mu'}$L$_\mu\ne0$. Then $\mu'$ is a Shale-Weil tuple.\end{lemma}
\begin{proof} The $\frak{sp}(W\oplus W^*)$-module L$_\mu$ is $\frak h_W$-bounded and therefore $\EuScript H_\mu^{\mu'}$L$_\mu$ is $\frak h_W$-bounded. As L$_\mu$ is a quotient of M$_\mu$, $\EuScript H_\mu^{\mu'}$L$_\mu$ is a quotient of $\EuScript H_\mu^{\mu'}$M$_\mu$. As L$_{\mu'}$ is the unique simple quotient of $\EuScript H_\mu^{\mu'}$M$_\mu$, L$_{\mu'}$ is an $\frak h_W$-bounded $\frak{sp}(W\oplus W^*)$-module. Therefore $\mu'$ is a Shale-Weil tuple by Theorem~\ref{Tbwm}.\end{proof}
\begin{corollary}\label{Cgsp} Let $M$ be an $\frak{sp}(W\oplus W^*)$-module such that I$_\frak{sp}(\mu)=$Ann$M$. Suppose $\EuScript H_\mu^{\mu'}M\ne 0$. Then $\mu'$ is a Shale-Weil $n_W$-tuple.\end{corollary}
\newpage\section{Modules of small growth}\label{Ssmg}
\begin{definition}Let $M$ be a $\frak g$-module. We say that $M$ is Z$(\frak g)$-{\it finite} if dim~Z$(\frak g)m<\infty$ for all $m\in M$, where \begin{center}Z$(\frak g)m:=\{m'\in M\mid m'=zm$ for some $m\in M$ and $z\in$Z$(\frak g)$\}.\end{center}\end{definition}
\subsection{The case $\frak g=\frak{sl}(W)$}
\begin{proposition}If $M$ is an infinite-dimensional finitely generated $\frak{sl}(W)$-module then dim~V$(M)\ge n_W-1$.\end{proposition}
\begin{proof}The variety GV$(M)$ is a union of several nilpotent SL$(W)$-orbits in $\frak{sl}(W)^*$. There is a unique nonzero nilpotent orbit of minimal dimension equal to 2$(n_W-1)$. As $M$ is infinite-dimensional, dim V$(M)\ge\frac{1}{2}$dim~GV$(M)\ge n_W-1$ by Theorem~\ref{Tber}.\end{proof}
\begin{definition}Let $M$ be a finitely generated $\frak{sl}(W)$-module which is Z$(\frak{sl}(W))$-finite. We say that $M$ is {\it of small growth} if dim~V$(M)\le n_W-1$, i.e. if dim~V$(M)$ equals either $n_W-1$ or 0.\end{definition}
The $\frak{sl}(W)$-modules of small growth form a full subcategory of the category of $\frak{sl}(W)$-modules. This subcategory is stable under the projective functors.

We recall that a finitely generated $\frak g$-module $M$ has a graded version gr$M$ which is a finitely generated S$(\frak g)$-module. The sum of ranks\footnote{The rank of finitely generated module over a commutative ring is the rank of a maximal free submodule.} of gr$M$ over the function rings of all irreducible components of V$(M)$ of maximal dimension is called {\it the Bernstein number of} $M$ (see also~\cite[p.~78]{KL}), and we denote this number b$(M)$.
\begin{proposition}Any $\frak{sl}(W)$-module $M$ of small growth is of finite length.\end{proposition}
\begin{proof}As $M$ is finitely generated, it is enough to check the descending chain condition for $M$. Let \begin{center}$...\subset M_{-i}\subset...\subset M_{-0}=M$\end{center} be a strictly descending chain of $\frak{sl}(W)$-submodules. These $\frak{sl}(W)$-submodules $M_{-i}$ are finitely generated. We have dim~V$(M_{-i})=n_W-1$ for all $i\in\mathbb Z_{\ge0}$ or dim~V$(M_{-i})=0$ for some $i\in\mathbb Z_{\ge0}$. In the second case dim~$M_{-i}<\infty$ and therefore the chain $...\subset M_{-i}$ stabilizes.

Assume that dim~V$(M_{-i})=n_W-1$ for all $i\in\mathbb Z_{\ge0}$. If $i\ge j$ then b($M_{-i})\le$b$(M_{-j})$, i.e. the sequence $\{$b$(M_{-i})\}_{i\in\mathbb Z_{\ge0}}$ is decreasing. Therefore there exists $i\in\mathbb Z_{\ge0}$ such that b$(M_{-i})=$b$(M_{-j})$ for all $j\ge i$. Therefore dim~V$(M_{-i}/M_{-j})<n_W-1$ and $M_{-i}/M_{-j}$ is a finite-dimensional $\frak{sl}(W)$-module for all $j\ge i$. We have an inclusion\begin{center}$M_{-i}/\cap_{j\ge i}M_{-j}\hookrightarrow \oplus_{j\ge i}M_{-i}/M_{-j}$.\end{center}The right-hand side is a direct sum of a finite-dimensional $\frak{sl}(W)$-modules. As $M_{-i}$ is Z$(\frak{sl}(W))$-finite and finitely generated, only a finite number of simple $\frak{sl}(W)$-modules appear in this direct sum. Therefore $M_{-i}/\cap_{j\ge i}M_{-j}$ is a finite-dimensional $\frak{sl}(W)$-module.\end{proof}
\begin{proposition}\label{Pasm} Let $I$ be the annihilator of a simple infinite-dimensional $\frak{sl}(W)$-module $M$ of small growth. Then $I$ is a Joseph ideal.\end{proposition}
\begin{proof} As dim~GV$(M)\le2$dim~V$(M)$ (Theorem~\ref{Tber}),\begin{center} dim~GV$(M)\le$2$(n_W-1)$.\end{center} As $M$ is infinite-dimensional, $I$ is a Joseph ideal.\end{proof}
Let $\EuScript J$ be the set of simple $\frak{sl}(W)$-modules of small growth, and $\langle\EuScript J\rangle$ be the vector space generated by the set $\EuScript J$. The action of the projective functors on the category of modules of small growth defines an action of the projective functors on $\langle\EuScript J\rangle$ by linear operators. As any $\frak{sl}(W)$-module $M$ of small growth has finite length, $M$ defines the vector $[M]\in\langle\EuScript J\rangle$ which is a sum of the simple $M$-subquotients with their multiplicities in $M$.

Let $\bar\lambda$ be a decreasing $n_W$-tuple and $\EuScript J_\lambda$ be the set of simple $\frak{sl}(W)$-modules of small growth annihilated by Ker$\chi_\lambda$, and $\langle\EuScript J_\lambda\rangle$ be the free vector space generated by $\EuScript J_\lambda$. The action of PFunc$(\chi_\lambda)$ defines an action of S$_{n_W}$ on $\langle\EuScript J_\lambda\rangle$, see Proposition~\ref{Ppfsn}. We recall that $\mathbb C^{sgn}$ is the sign representation of S$_{n_W}$. Set $\langle\EuScript J_\lambda\rangle^{sgn}:=\langle\EuScript J_\lambda\rangle\otimes_\mathbb C\mathbb C^{sgn}$. As vector spaces $\langle\EuScript J_\lambda\rangle$ and $\langle\EuScript J_\lambda\rangle^{sgn}$ are identical.
\begin{lemma}An S$_{n_W}$-module $\langle\EuScript J_\lambda\rangle^{sgn}$ is a direct sum of copies of $\mathbb C^{n_W-1}$ and $\mathbb C$.\end{lemma}
\begin{proof} Without loss of generality we assume that $\lambda=\rho$ ($\rho$ is defined in Subsection~\ref{SSpro}). The action of $\EuScript H_\rho^{\mathrm s_k\rho}$ on $\EuScript J_\lambda$ corresponds to multiplication by (s$_k-1$) on $\langle\EuScript J_\lambda\rangle^{sgn}$(cf. Corollary~\ref{Casn}). Let $M$ be a simple $(\frak{sl}(W), \frak{sl}(V))$-module of small growth and $k$ be a number such that Ann$M$=I(s$_k\rho$). Then $\EuScript H_\rho^{\mathrm s_l\rho}M=0$ for all $k\ne l$ (Corollary~\ref{Casn}). As the simple modules generate $\langle\EuScript J_\lambda\rangle^{sgn}$, \begin{center}$\langle\EuScript J_\lambda\rangle^{sgn}=+_{i\le n_W-1}\langle\EuScript J_\lambda\rangle^{sgn}_{\bar i}$.\end{center} Therefore as an S$_{n_W}$-module $\langle\EuScript J_\lambda\rangle^{sgn}$ is a direct sum of copies of $\mathbb C^{n_W-1}$ and $\mathbb C$ (Lemma~\ref{Lsn}).\end{proof}
\begin{lemma}Let $M$ be a simple $\frak{sl}(W)$-module of small growth annihilated by Ker$\chi_\rho$ and $k\in\{1,..., n_W-1\}$ be a number. Then, for some $j\in\{1, 2\}$, \begin{center}$[\EuScript H_\rho^{\rho+\varepsilon_k}M]=j[M']$\end{center}for a simple $\frak{sl}(W)$-module $M'$ of small growth.\end{lemma}
\begin{proof} Assume that $\EuScript H_\rho^{\rho+\varepsilon_k}M$ is nonzero. The identity homomorphism \begin{center}Id: $ \EuScript H_\rho^{\rho+\varepsilon_k}M\to\EuScript H_\rho^{\rho+\varepsilon_k}M$\end{center} gives rise to a nonzero homomorphism\begin{center}$M\to \EuScript H^\rho_{\rho+\varepsilon_k}\EuScript H_\rho^{\rho+\varepsilon_k}M=\EuScript H_\rho^{\mathrm s_k\rho}M$.\end{center}Let $M'$ be a simple subquotient of $\EuScript H_\rho^{\rho+\varepsilon_k}M$ such that $M$ is a subquotient of $\EuScript H^\rho_{\rho+\varepsilon_k}M'$. As $\EuScript H_\rho^{\rho+\varepsilon_k}\EuScript H^\rho_{\rho+\varepsilon_k}=2$Id, we have $\EuScript H_\rho^{\rho+\varepsilon_k}\EuScript H^\rho_{\rho+\varepsilon_k}M'=M'\oplus M'$. Therefore $\EuScript H_\rho^{\rho+\varepsilon_k}M$ is semisimple of length 2 or less and $M'$ is, up to isomorphism, the only simple constituent of $M$.\end{proof}
\begin{lemma} Let $M$ be a simple $\frak{sl}(W)$-module of small growth annihilated by I$(\rho+\varepsilon_k)$. Then, for some $j\in\{1, 2\}$,\begin{center}$[\EuScript H^\rho_{\rho+\varepsilon_k}M]=j[M']+[T]$,\end{center} where $M'$ is a simple $\frak{sl}(W)$-module of small growth and $T$ is some $\frak{sl}(W$)-module of small growth such that \begin{center}$\EuScript H_\rho^{\rho+\varepsilon_k}T=0$, $\EuScript H_\rho^{\rho+\varepsilon_k}M'=\frac{2}{j}M$.\end{center}\end{lemma}
\begin{proof} Let $0\subset M^0\subset...\subset M^i=M$ be a composition series of $\EuScript H^\rho_{\rho+\varepsilon_k}M$.
As $\EuScript H_\rho^{\rho+\varepsilon_k}\EuScript H^\rho_{\rho+\varepsilon_k}M=2$Id, the number of simple subquotients $M^{l+1}/M^l$ such that $\EuScript H_\rho^{\rho+\varepsilon_k}(M^{l+1}/M^l)$ is nonzero does not exceed 2. Assume that there is a unique such subquotient $M'$. Then there exists a semisimple $\frak{sl}(W)$-module $T$ such that \begin{center}$\EuScript H^\rho_{\rho+\varepsilon_k}[M]=[M']+[T]$,\end{center} and\begin{center} $[\EuScript H_\rho^{\rho+\varepsilon_k}M']=2[M]$, $\EuScript H_\rho^{\rho+\varepsilon_k}[T]=0$.\end{center}

Suppose that such a subquotient is not unique. Then there exist two simple $\frak{sl}(W)$-modules $\hat M_1$ and $\hat M_2$ and a semisimple $\frak{sl}(W)$-module $T$ such that \begin{center}$[\EuScript H^\rho_{\rho+\varepsilon_k}M]=[\hat M_1]+[\hat M_2]+[T]$ and $[\EuScript H_\rho^{\rho+\varepsilon_k}\hat M_1]=[\EuScript H_\rho^{\rho+\varepsilon_k}\hat M_2]=[M]$.\end{center} We have Hom$_{\frak{sl}(W)}(\EuScript H_\rho^{\rho+\varepsilon_k}\hat M_1, \EuScript H_\rho^{\rho+\varepsilon_k}\hat M_1)=$\begin{center}Hom$_{\frak{sl}(W)}(\hat M_1, \EuScript H^\rho_{\rho+\varepsilon_k}\EuScript H_\rho^{\rho+\varepsilon_k}\hat M_1)=$Hom$_{\frak{sl}(W)}(\EuScript H^\rho_{\rho+\varepsilon_k}\EuScript H_\rho^{\rho+\varepsilon_k}\hat M_1, \hat M_1)$.\end{center}
Then there exists a sequence of nonzero $\frak{sl}(W)$-homomorphisms\begin{center}
$\EuScript H_\rho^{\mathrm s_k\rho}\hat M_1\to\hat M_1\to\EuScript H_\rho^{\mathrm s_k\rho}\hat M_1.$\end{center}
Therefore there exists an $\frak{sl}(W)$-homomorphism $\phi_1: \EuScript H^\rho_{\rho+\varepsilon_k}M\to \EuScript H^\rho_{\rho+\varepsilon_k}M$ whose image is $\hat M_1$. In the same way there exists an $\frak{sl}(W)$-homomorphism \begin{center}$\phi_2:\EuScript H^\rho_{\rho+\varepsilon_k}M\to\EuScript H^\rho_{\rho+\varepsilon_k}M$\end{center}with image $\hat M_2$. Assume that $\hat M_1\not\cong\hat M_2$. Then both compositions $\phi_1\circ\phi_2$ and $\phi_2\circ\phi_1$ equal zero. We have \begin{center}Hom$_{\frak{sl}(W)}(\EuScript H^\rho_{\rho+\varepsilon_k}M, \EuScript H^\rho_{\rho+\varepsilon_k}M)=$\\Hom$_{\frak{sl}(W)}(M, \EuScript H_\rho^{\rho+\varepsilon_k}\EuScript H^\rho_{\rho+\varepsilon_k}M)=$Hom$_{\frak{sl}(W)}(M, \EuScript H^{\rho+\varepsilon_k}_{\rho+\varepsilon_k}M)$.\end{center}
As $M$ is simple, Hom$_{\frak{sl}(W)}(M, \EuScript H^{\rho+\varepsilon_k}_{\rho+\varepsilon_k}M)=$Hom$_{\frak{sl}(W)}($Id$, \EuScript H^{\rho+\varepsilon_k}_{\rho+\varepsilon_k})$, where the right-hand side refers to homomorphisms of functors. Therefore\begin{center} Hom$_{\frak{sl}(W)}(\EuScript H^\rho_{\rho+\varepsilon_k}M, \EuScript H^\rho_{\rho+\varepsilon_k}M)=$Hom$_{\frak{sl}(W)}(\EuScript H^\rho_{\rho+\varepsilon_k}, \EuScript H^\rho_{\rho+\varepsilon_k})$.\end{center} However,\begin{center} Hom$_{\frak{sl}(W)}(\EuScript H^\rho_{\rho+\varepsilon_k}, \EuScript H^\rho_{\rho+\varepsilon_k})=$Hom$_{\frak{sl}(W)}($P$_{\mathrm s_k\rho}, $P$_{\mathrm s_k\rho})=\mathbb C[x]/(x^2)$.\end{center} As $\phi_1\circ\phi_2=0$, $\phi_1$ and $\phi_2$ are identified with $\alpha_1x$ and $\alpha_2x$ for some constants $\alpha_1, \alpha_2\in\mathbb C$. Then $\alpha_1$ is proportional to $\alpha_2$ and therefore \begin{center}$\hat M_1\cong\hat M_2$.\end{center}\end{proof}
\begin{corollary}\label{Crisi} There is a natural bijection between the set of simple infinite-dimensional $\frak{sl}(W)$-modules of small growth annihilated by I$(s_k\rho)$ and the set of simple $\frak{sl}(W)$-modules of small growth annihilated by I($\rho+\varepsilon_k)$.\end{corollary}
\begin{lemma} The subspace of S$_{n_W}$-invariants in $\langle\EuScript J_\lambda\rangle$ is one-dimensional and is generated by the class of the simple finite-dimensional module.\end{lemma}
\begin{proof} Without loss of generality we assume that $\lambda=\rho$. Let $\{\alpha_i\}$ be numbers such that $\Sigma \alpha_i[M_i]$ is non-zero and S$_{n_W}$-invariant. Then $\EuScript H_\rho^{\rho+\varepsilon_k}\alpha_i[M_i]=0$, i.e. $\Sigma_i\alpha_i\EuScript H_\rho^{\rho+\varepsilon_k}M_i=0$. Therefore $\alpha_i=0$ for all infinite-dimensional simple modules annihilated by I(s$_k\rho$). Hence $\alpha_i\ne 0$ implies that $M_i$ is finite-dimensional.\end{proof}
\begin{corollary}For any $i, j\in\{1,..., n_W-1\}$ there is a natural bijection between the set of simple infinite-dimensional $\frak{sl}(W)$-modules of small growth annihilated by I$(s_i\rho)$ and the set of simple infinite-dimensional $\frak{sl}(W)$-modules of small growth annihilated by I($s_j\rho)$.\end{corollary}
The following statement can be considered as a weak analogue of Theorem~\ref{Tweq}.
\begin{corollary}\label{Cii} Let $\bar\lambda_1, \bar\lambda_2$ be semi-decreasing $n_W$-tuples such that\begin{center}m$(\lambda_1)$=m$(\lambda_2)$.\end{center} There is a natural bijection between the set of simple infinite-dimensional $\frak{sl}(W)$-modules of small growth annihilated by I$(\lambda_1)$ and the set of simple infinite-dimensional $\frak{sl}(W)$-modules of small growth annihilated by \begin{center}I$(\lambda_1)$ and I($\lambda_2)$.\end{center}\end{corollary}
\begin{proof}This follows from Lemma~\ref{Lsoch}, Corollary~\ref{Crisi} and Theorem~\ref{Tweq}.\end{proof}
\subsection{The case $\frak g=\frak{sp}(W\oplus W^*)$} All proofs for this subsection are similar to their counterparts in the previous subsection. We leave them to the reader.
\begin{proposition}If $M$ is an infinite-dimensional finitely generated\\ $\frak{sp}(W\oplus W^*)$-module, then dim~V$(M)\ge n_W$.\end{proposition}
\begin{definition}Let $M$ be a finitely generated $\frak{sp}(W\oplus W^*)$-module which is Z$(\frak{sp}(W\oplus W^*))$-finite. We say that $M$ is {\it of small growth} if dim~V$(M)\le n_W$, i.e. if dim~V$(M)$ equals either $n_W$ or 0.\end{definition}
The $\frak{sl}(W)$-modules of small growth form a full subcategory of the category of $\frak{sp}(W\oplus W^*)$-modules. This subcategory is stable under the projective functors.
\begin{proposition}Any $\frak{sp}(W\oplus W^*)$-module $M$ of small growth is of finite length.\end{proposition}
\begin{proposition} Let $I$ be an annihilator of a simple infinite-dimensional $\frak{sl}(W)$-module $M$ of small growth. Then $I$ is a Joseph ideal.\end{proposition}
Let $\mu$ be a Shale-Weil tuple. The set PFunc$(\chi_\mu)$ is generated by one involutive element $\EuScript H_\mu^{\sigma\mu}$ which preserve the set of simple modules. We recall that $\bar\mu_0$ is an $n_W$-tuple $(n_W-\frac{1}{2}, n_W-\frac{3}{2}..., \frac{1}{2})$.
\begin{proposition}The categories of $\frak{sp}(W\oplus W^*)$-modules of small growth annihilated by I$(\mu_0$) and I$(\mu)$ are equivalent.\end{proposition}
\newpage\section{Categories of $(\frak{sl}($S$^2V), \frak{sl}(V))$ and $(\frak{sl}(\Lambda^2V), \frak{sl}(V))$-modules}\label{Sbsl}
\subsection{Construction of (D($W), \frak k$)-modules} We recall that $K$ is a connected reductive Lie group with Lie algebra $\frak k$ and a Borel subgroup $B$. Let $\EuScript G$ be an associative algebra and $\psi$: U$(\frak k)\to\EuScript G$ be a homomorphism of associative algebras injective on $\frak k$. We identify $\frak k$ with its image. Any $\EuScript G$-module can be considered as a $\frak k$-module.
\begin{definition}(cf. Def.~\ref{Dgk}) A {\it $(\EuScript G, \frak k)$-module} is a $\EuScript G$-module which is locally finite as a $\frak k$-module.\end{definition}
\begin{definition}(cf. Def.~\ref{Dbm}) A {\it bounded $(\EuScript G, \frak k)$-module} is a $(\EuScript G, \frak k)$-module bounded as a $\frak k$-module.\end{definition}
Let $W$ be a spherical $K$-module. There is an obvious homomorphism\\ U$(\frak k)\to$D$(W)$ (see the discussion at the end of Subsection~\ref{SSbmsl}).

The algebra of differential operators D$(W)$ has a natural filtration  by degree \begin{center}$0\subset\mathbb C[W]\subset D_1\subset...\subset$D$(W)=\cup_{i\in\mathbb Z_{\ge0}}D_i$.\end{center} The associated graded algebra \begin{center}gr~D$(W) =\oplus_{i\in\mathbb Z_{\ge0}}(D_{i+1}/D_i)$\end{center} is isomorphic to $\mathbb C[$T$^*W]$. Let $M$ be D$(W)$-module with a finite-dimensional space of generators $M_{gen}$. This defines a filtration \begin{center}$\{D_iM_{gen}\}_{i\in\mathbb Z_{\ge0}}$\end{center} of $M$. The associated graded space \begin{center}gr$M:=\oplus_{i\in\mathbb Z_{\ge0}}(D_{i+1}M_{gen}/D_iM_{gen})$\end{center} is a finitely generated $\mathbb C[$T$^*W]$-module. We denote the support of this $\mathbb C[$T$^*W]$-module as $\EuScript V($Loc$M)$. As finitely generated D$(W)$-modules are in a natural one-to-one correspondence with coherent $\EuScript D(W)$-modules, $M$ corresponds to the $\EuScript D(W)$-module Loc$M$. The variety $\EuScript V($Loc$M)$ coincides with the singular support of Loc$M$ (see Section~\ref{Snot}) and is conical and coisotropic in T$^*W$. In what follows we assume that $M$ is a finitely generated (D$(W), \frak k)$-module.
\begin{proposition}\label{Pdbm}a) The module $M$ is bounded if and only if all irreducible components of $\EuScript V($Loc$M)$ are $K$-spherical.\\ b) If the equivalent conditions of a) are satisfied, any irreducible component $\tilde V$ of $\EuScript V($Loc$M)$ is a conical Lagrangian subvariety of T$^*W$.\end{proposition}
\begin{proof}The proof repeats the proof of Proposition~\ref{Pbm} verbatim.\end{proof}
As explained in Section~\ref{Stro}, to a holonomic $\EuScript D(W)$-module $\EuScript M$ one assigns a pair $(S, Y)$, where $S\subset W$ is an irreducible subvariety and $Y$ is an $\EuScript O(S)$-coherent $\EuScript D(S)$-module. Let $K_{ss}$ be the connected simply-connected semisimple group with Lie algebra $[\frak k, \frak k]$, $A(K)$ be the center of $K$ and $\frak a\subset\frak k$ be the Lie algebra of $A(K)$.
\begin{lemma}\label{Lkrs}The module $M$ is $\frak k$-bounded. If $M$ is simple then $M$ is $\frak k$-multiplicity free. The $\EuScript D(W)$-module Loc$M$ is holonomic and has regular singularities with respect to the stratification by $K$-orbits on $W$.\end{lemma}
\begin{proof}Let $\phi^{-1}_\frak k(0)\subset$T$^*W$ be the union of conormal bundles to all $K$-orbits in $W$. As $\frak k$ acts locally finitely on $M$, $\EuScript V(M)\subset\phi^{-1}_\frak k(0)$. Since the dimension of any irreducible component of $\phi^{-1}_\frak k(0)$ is $n_W$ and dim$\EuScript V(M)\ge n_W$ (Theorem~\ref{GabD}), $\EuScript V(M)$ is the union of irreducible components of $\phi^{-1}_\frak k(0)$. Therefore Loc$M$ is holonomic.

Let $\tilde V$ be an irreducible component of $\EuScript V(M)$. It is the total space of the conormal bundle to a $K$-orbit $S\subset W$. As $W$ is a $K$-spherical variety, $\tilde V$ is a spherical variety~\cite{Pan}. Therefore $M$ is a $\frak k$-bounded module by Proposition~\ref{Pdbm}. Assume $M$ is simple. Since the algebra D$(W)^\frak k$ is commutative~\cite{Kn}, $M$ is multiplicity free by ~\cite[Corollary~5.8]{PS2}.

We have to prove that Loc$M$ has regular singularities. Without loss of generality we assume that $M$ is simple. Let $(S, Y)$ be the pair corresponding to Loc$M$. Then $S$ is $K$-stable. As $W$ is $K$-spherical, $S$ is a $K$-orbit.  Since Loc$M$ is $K_{ss}$-equivariant, $Y$ is also $K_{ss}$-equivariant. As $\frak a$ acts locally finitely on $M$, $\frak a$ acts  locally finitely on $\Gamma($A$(K)s, M|_{\mathrm A(K)s})$ for any point $s\in S$. Since $K$ is a quotient of $K_{ss}\times$A$(K)$, $Y$ has regular singularities. Therefore Loc$M$ has regular singularities (cf.~\cite[12.11]{Bo}).\end{proof}
\begin{lemma}If a holonomic D$(W)$-module Loc$M$ has regular singularities with respect to the stratification of $W$ by $K$-orbits, then $M$ is $\frak k$-locally finite.\end{lemma}
\begin{proof}Without loss of generality assume that $M$ is simple. Let $(S, Y)$ be the pair corresponding to Loc$M$. Then $S$ is a $K$-orbit. As $S$ is a homogeneous $K_{ss}\times$A$(K)$-space, the sheaf $Y$ is $K_{ss}$-equivariant~\cite[12.11]{Bo}. Hence Loc$M$ is $K_{ss}$-equivariant and therefore $[\frak k, \frak k]$ acts locally finitely on $M$. As $Y|_{\mathrm A(K)s}$ has regular singularities for any point $s\in S$, $\frak a$ acts locally finitely on $\Gamma($A$(K)s, M|_{\mathrm A(K)s})$ for any point $s\in S$. Therefore $\frak a$ acts locally finitely on $M$. As $\frak k=[\frak k, \frak k]\oplus\frak a$, $\frak k$ acts locally finitely on $M$.\end{proof}
By a famous theorem of M.~Kashiwara~\cite{Ka2}, the category of holonomic D$(W)$-modules with regular singularities with respect to a given stratification is equivalent to the category of perverse sheaves on $W$ with respect to the same stratification.
\begin{corollary} The category of bounded (D($W), \frak k)$-modules is equivalent to the category of perverse sheaves on $W$ with respect to the stratification of $W$ by $K$-orbits.\end{corollary}
The categories of such perverse sheaves has been described explicitly for the stratifications of S$^2V$ and $\Lambda^2V$ by the GL($V$)-orbits~\cite{BG}. In particular, in~\cite{BG} the authors have described the simple objects of this categories.

The following theorem provides a checkable condition under which the functions $[M_{1,2}:\cdot]_\frak k$ are pairwise disjoint (see Section~\ref{Stro}) for non-isomorphic simple (D$(W), \frak k)$ modules $M_1$ and $M_2$. The proof, which we present below, follows essentially the proof of~\cite[Corollary 5.8]{PS2}.
\begin{theorem}\label{T1m} Assume D$(W)^\frak k$ is the image of Z$(\frak k)$. Let $E$ be a simple $\frak k$-module. Then there exists at most one simple (D$(W), \frak k)$-module $M$ such that Hom$_\frak k(E, M)\ne0$.\end{theorem}
To prove Theorem~\ref{T1m} we need the following preparatory lemma.
\begin{lemma}\label{L96}Let $E$ be a simple finite-dimensional $\frak k$-module. Then for any $n\in\mathbb Z_{\ge1}$ and $\phi\in$Hom$(E, E\otimes \mathbb C^n)$ there exist $\phi_1,..., \phi_s\in$Hom$_\frak k(E, E\otimes\mathbb C^n)$ such that $\phi_i\in$U$(\frak k)\phi$U$(\frak k)$ and $\phi\in+_i$U$(\frak k)\phi_i$U$(\frak k)$.\end{lemma}
\begin{proof}We omit the proof.\end{proof}
\begin{proof}[Proof of Theorem~\ref{T1m}]Let $R$ be the isotypic component of $E$ in D$(W)\otimes_{\mathrm U(\frak k)}E$ and pr denote the projection D$(W)\otimes_{\mathrm U(\frak k)}E\to R$. For any $g\in$D$(W)$ there exists $g'\in$D$(W)$ such that pr($gm)=g'm$ for any $m\in E$. Therefore $g'$ induces a homomorphism $\phi$ from $E$ to $R$. Let $\phi_i$ be as in Lemma~\ref{L96}. Then for any $i$ there exist $g_i\in$D$(W)$ such that $g_im=\phi_im$ for any $m\in E$. We have $[g_i, k]m=0$ for any $k\in\frak k$ and $m\in E$. As the adjoint action of $\frak k$ on D$(W)$ is semisimple, $g_i=c_i+a_i$ for some $c_i, a_i\in$D$(W)$ such that $c_i\in$D$(W)^\frak k$ and $a_im=0$ for all $m\in E$. Hence $g_iE\subset$U$(\frak k)E$ and therefore $R=E$.

Let $\tilde M$ be the unique maximal submodule of D$(W)\otimes_{\mathrm U(\frak k)}E$ which does not contain $R=E$. Then the quotient D$(W)\otimes_{\mathrm U(\frak k)}E/\tilde M$ is simple and, up to isomorphism, is the unique simple (D$(W), \frak k)$-module for which dim~Hom$_\frak k(E,\cdot)\ne 0$.\end{proof}
Let $\frak k=\frak{gl}(V)$ and let $W$ equal $\Lambda^2V$ or S$^2V$. Then D$(W)^\frak k$ is the image of Z$(\frak k)$~\cite{HU}. Therefore, for non-isomorphic simple (D$(W), \frak k$)-modules $M_{1,2}$, the functions $[M_{1,2}:\cdot]_\frak k$ are pairwise disjoint, i.e. the product of these functions is the zero function (see Section~\ref{Stro}).
\subsection{A useful algebraic trick}\label{SSutr} We recall that $E$ is the Euler operator and that \{D$^i(W)\}_{i\in\mathbb Z}$ are the $E$-eigenspaces of D$(W)$ (see the discussion at the end of Subsection~\ref{SSbmsl}). In this section we relate the category of (D$(W), E$)-modules to the categories of D$^0(W)$ and D$^{\bar 0}(W)$-modules (see also~\cite{PS2}). Let $M$ be a locally finite $E$-module. Fix $t\in\mathbb C$. Set\begin{center}$M_t:=\{m\in M\mid (E-t)^nm=0$ for some $n\in\mathbb Z_{\ge0}$\}.\end{center}
If $M$ is D$(W)$-module which is semisimple an an $E$-module, then $M_t$ is a D$^t\mathbb P(W)$-module.
\begin{definition}We say that a D$(W)$-module $M$ is {\it a $\mathrm  D(W)$-module with monodromy} e$^{2\pi it}$ if $M=\oplus_{j\in\mathbb Z}M_{t+j}$.\end{definition}
Any simple D$(W)$-module with a locally finite action of $E$ is a module with a monodromy. Let $\delta_0$ be the delta function at $0\in W$. The functions 1 and $\delta_0$ generate D$(W)$-modules which we denote D$(W)1$ and D$(W)\delta_0$.
\begin{lemma}\label{Lsup} Let $M$ be a simple D$(W)$-module with monodromy e$^{2\pi it}$. Assume that \begin{center}dim$M_t<\infty$ and $n_W\ge 2$.\end{center} Then $M$ is isomorphic to one of the following modules: \begin{center}0, D$(W)1$, D$(W)\delta_0$.\end{center}\end{lemma}
\begin{proof}Let dim$M_t\ne0$. Then the action of $\frak{sl}(W)$ on $M_t$ is locally finite and therefore $M$ is a locally finite $\frak{sl}(W)$-module. As $W$ is a spherical SL$(W)$-module, $M$ is a holonomic D$(W)$-module with regular singularities by Lemma~\ref{Lkrs}. Let Sh be the corresponding to $M$ simple perverse sheaf. Then Sh is constructible with respect to the stratification \{$\{0\}$, $W-\{0\}$\} of $W$. As \begin{center}$\pi_1(\{0\})=\pi_1(W-\{0\})=0$,\end{center} Sh is the simple perverse sheaf corresponding to $\{0\}$ or $W-\{0\}$, in both cases the local system being trivial. The latter two perverse sheaves correspond to the simple D$(W)$-modules D$(W)\delta_0$ and D$(W)1$, respectively. Therefore $M$ is isomorphic to 0, to D$(W)1$, or to D$(W)\delta_0$.

Assume that $M_t=0$. Then, for all $i\le n_W$, the action of e$_i$ on $M$ is locally finite, or the action of $\partial_{\mathrm e_i}$ is locally finite for all $i$. In the first case $M$ is isomorphic to D$(W)\delta_0$ by Kashiwara's theorem. In the second case D$(W)$ is isomorphic to D$(W)1$ by Kashiwara's theorem (we can interchange the roles of e$_i$ and $\partial_{\mathrm e_i}$).\end{proof}
\begin{lemma}\label{Ldtpw}a) Let $M$ be a nonzero simple D$(W)$-module with monodromy e$^{2\pi it}$ which is not isomorphic to D$(W)1$ and D$(W)\delta_0$. Then $M_t$ is a simple infinite-dimensional D$^t\mathbb P(W)$-module.\\b) Let $M_{res}$ be a simple infinite-dimensional D$^t\mathbb P(W)$-module. Then the D$(W)$-module \begin{center}D$(W)\otimes_{\mathrm D^0(W)}M_{res}$\end{center} has a unique simple quotient $\tilde M$. This quotient is not isomorphic to D$(W)1$ or to D$(W)\delta_0$ and has monodromy e$^{2\pi it}$. Moreover, there is an isomorphism $\tilde M_t=M_{res}$ of D$^t\mathbb P(W)$-modules.\end{lemma}
\begin{proof}a) By Lemma~\ref{Lsup} we have dim$M_t=\infty$. Let $\tilde M_t\subset M_t$ be a D$^t\mathbb P(W)$-submodule of $M_t$. The inclusion $\tilde M_t\to M_t$ induces a homomorphism \begin{center}inc: D$(W)\otimes_{\mathrm D^0(W)}\tilde M_t\to M$\end{center} of $\mathbb Z$-graded D$(W)$-modules. We have \begin{center}(D$(W)\otimes_{\mathrm D^0(W)}\tilde M_t)_t=\tilde M_t$,\end{center}and the map\begin{center}inc$_t$: (D$(W)\otimes_{\mathrm D^0(W)}\tilde M_t)_t\to M_t$\end{center} is the inclusion map. As $M$ is simple,  inc equals 0 or is surjective and therefore $\tilde M_t$ is isomorphic to 0 or to $M_t$. Therefore a) follows.

Any submodule $\bar M$ of D$(W)\otimes_{\mathrm D^0(W)}M_{res}$ either intersects $M_{res}$ trivially, or coincides with D$(W)\otimes_{\mathrm D^0(W)}M_{res}$. Therefore D$(W)\otimes_{\mathrm D^0(W)}M_{res}$ has a unique maximal submodule $\bar M$ and this submodule is isomorphic to a direct sum of copies of D$(W)1$ and D$(W)\delta_0$. The quotient \begin{center}$\tilde M:=$(D$(W)\otimes_{\mathrm D^0(W)}M_{res})/\bar M$\end{center} is simple and $\tilde M_t=M_{res}$. As $M_{res}$ is infinite-dimensional, $\tilde M$ is not isomorphic to D$(W)1$ or to D$(W)\delta_0$.\end{proof}
This reduces the study of (simple) D$^t\mathbb P(W)$-modules to a study of some (simple) modules over the Weyl algebra. Sometimes the latter is much easier.

We recall that D$(W)=$D$^{\bar0}(W)\oplus$D$^{\bar1}(W)$ is a
$\mathbb Z_2$-graded algebra. We say that a D$^{\bar 0}(W)$-module
{\it has half-monodromy} e$^{\pi it}$ if $M$ is a direct sum
$\oplus_{j\in\mathbb Z}M_{t+2j}$. We denote by \begin{center}(D($W),
E)$-mod\end{center} the category of D$(W)$-modules with a locally
finite action of $E$ and by \begin{center}(D$(W), E)$-mod$^{\mathrm
e^{2\pi it}}$\end{center} the subcategory of (D$(W), E)$-mod of
modules with monodromy e$^{2\pi it}$ for any $t\in\mathbb C$.
Similarly we define \begin{center}(D$^{\bar 0}(W), E)$-mod and
(D$^{\bar 0}(W), E)$-mod$^{\mathrm e^{\pi it}}$.\end{center} The
sign '$\cong$' stands for equivalence of categories. For any
$t\in\mathbb C$ we have two functors:
\begin{center}Res$_{\mathrm e^{2\pi it}}^{\mathrm e^{\pi it}}:$ D$(W)$-mod$^{\mathrm e^{2\pi it}}\to$D$^{\bar 0}(W)$-mod$^{\mathrm e^{\pi it}}$,\\
$M\mapsto \oplus_{j\in\mathbb Z}M_{t+2j}$;\\
Ind$_{\mathrm e^{\pi it}}^{\mathrm e^{2\pi it}}:$ D$^{\bar 0}(W)$-mod$^{\mathrm e^{\pi it}}\to$D($W$)-mod$^{\mathrm e^{2\pi it}}$,\\
$M\mapsto$D$(W)\otimes_{\mathrm D^{\bar 0}(W)}M$.\end{center}
\begin{theorem}The functors Ind and Res are mutually inverse equivalences of the categories of (D$(W), E)$-mod$^{\mathrm e^{2\pi it}}$ and (D$^{\bar 0}(W), E)$-mod$^{\mathrm e^{\pi it}}$.\end{theorem}
\begin{proof}For a D$(W)$-module $M$ with monodromy e$^{2\pi it}$, set \begin{center}$M_{\bar t}:=\oplus_{j\in\mathbb Z}M_{t+2j+1}$.\end{center}Then $M=M_{\bar t}\oplus M_{\overline{t+1}}$, and this is a $\mathbb Z_2$-grading on $M$ considered as a D$(W)$-module.

Let $M_{\bar t}$ be a D$^{\bar 0}(W)$-module with half-monodromy e$^{\pi it}$. The D$(W)$-module \begin{center}D$(W)\otimes_{\mathrm D^{\bar0}(W)}M_{\bar t}$\end{center} is $\mathbb Z_2$-graded and has monodromy e$^{2\pi it}$. Moreover,\begin{center}(D$(W)\otimes_{\mathrm D^{\bar0}(W)}M_{\bar t})_{\bar t}=$D$^{\bar 0}(W)\otimes_{\mathrm D^{\bar0}(W)}M_{\bar t}=M_{\bar t}$.\end{center} Therefore Res$^{\mathrm e^{\pi it}}_{\mathrm e^{2\pi it}}\circ$Ind$_{\mathrm e^{\pi it}}^{\mathrm e^{2\pi it}}$ is the identity functor.

Let $M$ be a D$(W)$-module with monodromy e$^{2\pi it}$. If $M_t=0$, then $M$ has finite length and, up to isomorphism, the only simple constituents of it are D$(W)1$ and D$(W)\delta_0$. In both cases $M_{\bar t}\ne 0$. There is a natural homomorphism \begin{center}$\psi$: D$(W)\otimes_{\mathrm D^{\bar0}(W)}M_{\bar t}\to M$.\end{center} We have \begin{center}(D$(W)\otimes_{\mathrm D^{\bar0}(W)}M_{\bar t})_{\bar t}=$D$(W)^{\bar 0}\otimes_{\mathrm D^{\bar0}(W)}M_{\bar t}=M_{\bar t}$.\end{center}Moreover,\begin{center}$\psi_{\bar t}:$ (D$(W)\otimes_{\mathrm D^{\bar0}(W)}M_{\bar t})_{\bar t}\to M_{\bar t}$\end{center} is an isomorphism. Therefore $(M/$Im$\psi)_{\bar t}=0$ and hence $M/$Im$\psi=0$. This shows that Ind$_{\mathrm e^{\pi it}}^{\mathrm e^{2\pi it}}\circ$Res$^{\mathrm e^{\pi it}}_{\mathrm e^{2\pi it}}$ is the identity functor.\end{proof}
\begin{corollary}We have \begin{center}(D$^{\bar 0}(W), E)$-mod$\cong$(D$(W), E)$-mod$\oplus$(D$(W), E)$-mod.\end{center}\end{corollary}
\begin{proof}We have\begin{center}(D$^{\bar 0}(W), E)$-mod$\cong\oplus_{t\in\mathbb C/\mathbb Z}$(D$^{\bar 0}(W), E)$-mod$^{\mathrm e^{2\pi it}}$,\\(D$(W), E)$-mod$\cong\oplus_{t\in\mathbb C/\mathbb Z}$(D$(W), E)$-mod$^{\mathrm e^{2\pi it}}$.\end{center}On the other hand (D$^{\bar 0}(W), E)$-mod$^{\mathrm e^{\pi it}}\cong$(D$(W), E)$-mod$^{\mathrm e^{2\pi it}}$. The two half-monodromies e$^{\pi it}, -$e$^{\pi it}=$e$^{\pi i(t+1)}$ combine into monodromy e$^{2\pi it}$, and\begin{center}(D$^{\bar 0}(W), E)$-mod$\cong$(D$(W), E)$-mod$\oplus$(D$(W), E)$-mod.\end{center}\end{proof}
\subsection{A description of $(\frak{sl}(W), \frak{sl}(V))$-modules}\label{SSdsl}In this section\begin{center}$K:$=SL$(V)$, $\frak g:=\frak{sl}(W)$\end{center} where $W:=\Lambda^2V (n_V=2k, n_V\ge5)$ or $W=$S$^2V (n_V\ge3)$.
\begin{lemma}\label{LslJ} Let $M$ be a simple bounded $(\frak{sl}(W), \frak{sl}(V))$-module. Then Ann$M$ is a Joseph ideal.\end{lemma}
\begin{proof} By Theorem~\ref{Tcob} the variety GV$(M)$ is SL$(V)$-coisotropic. According to the discussion following Example~\ref{Pn}, GV$(M)$ is birationally isomorphic to T$^*$Fl for some partial $W$-flag variety Fl. Hence Fl is an SL$(V)$-spherical variety. As the only partial $W$-flag varieties which are SL$(V)$-spherical are $\mathbb P(W)$ and $\mathbb P(W^*)$, GV$(M)$ coincides with $\EuScript Q$ (cf. Example~\ref{Pn}).\end{proof}
We are now ready to prove Theorems ~\ref{Tsmsl}-\ref{Tasl}. 
\begin{proof}[Proof of Theorem~\ref{Tsmsl}] As Ann$M$ is a Joseph ideal, V$(M)\subset\EuScript Q$. As $M$ has finite type over $\frak k$, we have dim~V$(M)=\frac{1}{2}$dim$\EuScript Q=n_W-1$ and $M$ is of small growth.\end{proof}
\begin{proof}[Proof of Theorem~\ref{Tannsl}]By Lemma~\ref{LslJ} the ideal Ann$M$ is a Joseph ideal. Therefore Ann$M$=I$(\lambda)$ for some semi-decreasing $n_W$-tuple $\bar\lambda$ by Corollary~\ref{CJid}.\end{proof}
\begin{proof}[Proof of Theorem~\ref{Tasl}]As Ann$M$=I$(\lambda)$ for some semi-decreasing tuple $\bar\lambda$, then V(Ann$M$)=GV($M$)=$\EuScript Q$ (Corolary~\ref{CJid}). The variety $\mathbb P(W)$ is SL$(V)$-spherical, and hence the variety T$^*\mathbb P(W)$ is SL$(V)$-coisotropic. As $\EuScript Q$ is birationally isomorphic to T$^*\mathbb P(W)$ (Example~\ref{Pn}), $\EuScript Q$ is SL$(V)$-coisotropic. Therefore $M$ is a bounded $(\frak{sl}(W), \frak{sl}(V))$-module by Theorem~\ref{Tcob}.\end{proof}
Let B$_t(W)$ denote the cardinality of the set of isomorphism classes of simple infinite-dimensional bounded $(\frak{sl}(W), \frak{sl}(V))$-modules annihilated by I($t, n_W-1,..., 1)$. Let P$_{e^{2\pi it}}(W)$ be the cardinality of the set of isomorphism classes of simple perverse sheaves on $W$ (with respect to the stratification by GL$(V)$-orbits) such that they have fixed monodromy $e^{2\pi it}$ and are neither supported at 0 nor are smooth on $W$. The monodromies of the simple perverse sheaves on $W$ with respect to the stratification by GL$(V)$-orbits are described in quiver terms in the Appendix.
\begin{lemma}\label{Ldps}We have B$_t(W)=$P$_{\mathrm e^{2\pi it}}(W)<\infty$ for all $t\in\mathbb C$.\end{lemma}
\begin{proof} Let ${\lambda_t\in\frak h_W^\frak{sl}}^*$ be the weight corresponding to $(t, n_W-1,..., 1)$. The quotient U$(\frak{sl}(W))/$Ann~L$_{\lambda_t}$ is isomorphic to D$^t\mathbb P(W)$ (see the discussion at the end of Subsection~\ref{SSbmsl}). Therefore the isomorphism classes of simple infinite-dimensional $(\frak{sl}(W), \frak{sl}(V))$-modules annihilated by I($\lambda_t)$ are in one-to-one correspondence with the simple (D$(W), \frak{gl}(V))$-modules with monodromy $e^{2\pi it}$, which are not isomorphic to D$(W)1$ or D$(W)\delta_0$ (see Lemma~\ref{Ldtpw}).\end{proof}
For a semi-decreasing $n_W$-tuple $\bar\lambda$ we denote by B$_\lambda(W)$ the cardinality of the set of isomorphism classes of simple infinite-dimensional bounded $(\frak{sl}(W), \frak{sl}(V))$-modules annihilated by I($\lambda)$. Then for some $t\in\mathbb C$ such that e$^{2\pi it}=m(\lambda)$, we have \begin{center}B$_\lambda(W)$=B$_t(W)=$P$_{\mathrm m(\lambda)}(W)<\infty$,\end{center} by Lemma~\ref{Lsoch} and Corollary~\ref{Cii}. This proves Theorem~\ref{Tgkps}.

We are now ready to prove Theorem~\ref{Tsptosl}.
\begin{proof}[Proof of Theorem~\ref{Tsptosl}]a) Let $M$ be a bounded $(\frak{sp}(W\oplus W^*), \frak{gl}(V))$-module. We can consider $M$ as an $(\frak{gl}(W), \frak{gl}(V))$-module. Let $E$ be a generator of the center of $\frak{gl}(W)$. Then $M$ is a direct sum of $E$-eigenspaces $\oplus_{t\in\mathbb C}M_t$. As $M$ is a bounded $(\frak{gl}(W), \frak{gl}(V))$-module, $M_t$ is a bounded $(\frak{sl}(W), \frak{sl}(V))$-module.\\b) Any simple bounded $(\frak{sl}(W), \frak{sl}(V))$-module annihilated by I($s_1\rho)$ is an $(\frak{sl}(W), \frak{sl}(V))$-submodule of an $(\frak{sp}(W\oplus W^*), \frak{gl}(V))$-module. Any other simple bounded module is isomorphic to a subquotient of a tensor product $F\otimes M$ where $F$ is a finite-dimensional SL$(W)$-module and $M$ is a simple $(\frak{sl}(W), \frak{sl}(V))$-module annihilated by I$(\rho+\varepsilon_k)$. Let $F_\frak{sp}$ be an SP$(W\oplus W^*)$-module which contains $F$ as an SL$(W)$-submodule (such a module always exists, see~\cite{Gr}). Then $F_\frak{sp}\otimes M$ is a bounded $\frak{gl}(V)$-module and contains $M$ as an $\frak{gl}(W)$-submodule.\end{proof}
\subsection{Categories of $(\frak{sp}(W\oplus W^*), \frak{gl}(V))$-modules}\label{Sspgl}In this section \begin{center}$K:$=GL$(V)$, $\frak g:=\frak{sp}(W\oplus W^*)$\end{center} where $W:=\Lambda^2V (n_V=2k, n_V\ge5)$ or $W=$S$^2V (n_V\ge3)$.
\begin{lemma}\label{LspJ}Let $M$ be a simple bounded $(\frak{sp}(W\oplus W^*), \frak{gl}(V))$-module. Then Ann$M$ is a Joseph ideal.\end{lemma}
\begin{proof}Let $\tilde{\EuScript Q}$ be the image of the moment map $\phi:$T$^*\mathbb P(W\oplus W^*)\to\frak{sp}(W\oplus W^*)^*$. Assume that GV$(M)$ does not equal to 0 or to $\bar{\EuScript Q}_\frak{sp}$. Then~\cite[6.2]{CM} $\tilde{\EuScript Q}\subset$GV$(M)$ and dim~GV$(M)\ge2(2n_W-1)$. Therefore dim$\frak b_{\frak{sl}(V)}\ge2n_W-1$ by Corollary~\ref{Cps}. This inequality is false.\end{proof}
We are now ready to prove Theorems ~\ref{Tsmsp}---~\ref{Tspeq} and Theorem~\ref{Tspps}.
\begin{proof}[Proof of Theorem~\ref{Tsmsp}] As Ann$M$ is a Joseph ideal, we have V$(M)\subset\bar{\EuScript Q}_{\frak{sp}}\cap\frak k^\bot$. As dim$\bar{\EuScript Q}_{\frak{sp}}\cap\frak k^\bot=n_W$, $M$ is of small growth.\end{proof}
\begin{proof}[Proof of Theorem~\ref{Tannsp}]As $M$ is a bounded module, Ann$M$ is a Joseph ideal. Therefore Ann$M$=I$(\bar\mu)$ for some Shale-Weil tuple $\bar\lambda$ by Corollary~\ref{Cqtosw}.\end{proof}
\begin{proof}[Proof of Theorem~\ref{Tasp}]As $M$ is an $(\frak{sp}(W\oplus W^*), \frak{gl}(V))$-module, we have V$(M)\subset\frak{gl}(V)^\bot\cap\EuScript Q_\frak{sp}$. Moreover, $\EuScript Q_\frak{sp}$ is a quotient of $W\oplus W^*$ by $\mathbb Z_2$ (see Subsection~\ref{SSJisp}), hence $\frak{gl}(V)^\bot\cap\EuScript Q_\frak{sp}$ is a quotient by $\mathbb Z_2$ of the union of conormal bundles to GL$(V)$-orbits. Since $W$ is GL$(V)$-spherical, these conormal bundles are GL$(V)$-spherical. Therefore any irreducible component of V$(M)$ is GL$(V)$-spherical, and $M$ is $\frak{gl}(V)$-bounded by Proposition~\ref{Pbm}.\end{proof}
We recall that $\bar\mu_0$ is an $n_W$-tuple $(n_W-\frac{1}{2}, n_W-\frac{3}{2},..., \frac{1}{2})$ and $\mu_0\subset\frak h_W^*$ is the corresponding weight.
\begin{proof}[Proof of Theorem~\ref{Tspeq}] Let $\bar\mu_1, \bar\mu_2$ be Shale-Weil $n_W$-tuples. The categories of bounded $(\frak{sp}(W\oplus W^*), \frak{gl}(V))$-modules annihilated by I$(\mu_1)$ and I$(\mu_2)$ are equivalent by Theorem~\ref{Tweq}.\end{proof}
\begin{proof}[Proof of Theorem~\ref{Tspps}] The category of $(\frak{sp}(W\oplus W^*), \frak{gl}(V))$-modules annihilated by I($\mu$) is equivalent, via the functors $\EuScript H_{\mu_0}^{\mu'}$ and $\EuScript H_{\mu'}^{\mu_0}$, to the category of $(\frak{sp}(W\oplus W^*), \frak{gl}(V))$-modules annihilated by I($\mu_0)$. The second category is equivalent to the direct sum of two copies of the category of perverse sheaves on $W$ with respect to the stratification by GL$(V)$-orbits (see Subsection~\ref{SSutr}).\end{proof}
Let $\langle\EuScript J_{\mu_0}\rangle$ be the free vector space generated by the isomorphism classes of simple bounded $(\frak{sp}(W\oplus W^*), \frak{gl}(V))$-modules annihilated by Ker$\chi_{\mu_0}$. We recall that PF$\overline{\mathrm{unc}}(\chi_{\mu_0})$ acts on $\langle\EuScript J_{\mu_0}\rangle$ by linear operators ( Section~\ref{SSpro}). Let $\bar\mu'$ be an $n_W$-tuple such that $\EuScript H_{\mu_0}^{\mu'}|_{\langle\EuScript J_{\mu_0}\rangle}\ne 0$. Then $\bar\mu'$ is a Shale-Weil tuple by Corollary~\ref{Cgsp}. Therefore the action of PF$\overline{\mathrm{unc}}(\chi_{\mu_0})$ collapses to an action of $\EuScript H_{\mu_0}^{\mu_0}$ and $\EuScript H_{\mu_0}^{\sigma\mu_0}$. The functor $\EuScript H_{\mu_0}^{\sigma\mu_0}$ is involutive and we call it Inv (see also Section~\ref{Stro}).

Let $M$ be a simple U$(\frak{sp}(W\oplus W^*))$-module of small growth annihilated by I$(\mu_0)$. Let $E$ be a finite-dimensional simple $\frak{sp}(W\oplus W^*$)-module.  We recall that F$_E:=E\otimes\cdot$. Then F$_E\EuScript H_{\mu_0}^{\mu_0}$ is a direct sum $\oplus_i\EuScript H_{\mu_0}^{\mu_i}$ of indecomposable projective functors. Therefore \begin{center}$E\otimes M=\oplus_i\EuScript H_{\mu_0}^{\mu_i}M$.\end{center} If $\mu_i$ is not a Shale-Weil tuple for some $i$, then $\EuScript H_{\mu_0}^{\mu_i}M=0$ (Corollary~\ref{Cgsp}). If $\mu_i$ is a negative Shale Weil tuple for some $i$, then $\EuScript H_{\mu_0}^{\mu_i}=\EuScript H_{\mu_0}^{\sigma\mu_i}$Inv. Therefore \begin{center}F$_EM=\oplus_{i\le s}\EuScript H_{\mu_0}^{\mu_i'}M\oplus(\oplus_{i\le s'}\EuScript H_{\mu_0}^{\mu_i''}$Inv$M$)\end{center} for some positive Shale-Weil tuples $\bar\mu_1',..., \bar\mu_s'$ and $\bar\mu_1'',..., \bar\mu_{s'}''$.

Let  $\EuScript J$ be the free vector space generated by the simple $\frak{sp}(W\oplus W^*)$-modules of small growth and $\langle\EuScript J\rangle$ be the free vector space generated by $\EuScript J$ (cf. Section~\ref{Ssmg}). One can express the projective functors $\EuScript H_{\mu_0}^\mu$ in terms of the functors F$_E$.  More precisely, for any Shale-Weil tuple $\bar\mu$ we have$$\EuScript H_{\mu_0}^{\mu}=(\oplus_{i\le s}c_i\mathrm F_{E_i}\EuScript H_{\mu_0}^{\mu_0})+(\oplus_{i\le s'} c_i'\mathrm F_{E_i'}\mathrm{Inv})\eqno(5)$$ for some finite-dimensional $\frak{sp}(W\oplus W^*)$-modules $E_1,..., E_s, E_1',..., E_s'$ and some integers $c_1,..., c_s, c_1',..., c_{s'}'$; this is an equality for linear operators on $\langle\EuScript J\rangle$. Therefore\begin{center}$\EuScript H_{\mu_0}^\mu([M]+[$Inv$M])=([M]+$Inv$[M])\otimes_\mathbb C(\oplus_ic_i [E_i]+\oplus_ic_i'[E_i'])$\end{center}and\begin{center}$\EuScript H_{\mu_0}^\mu([M]-[$Inv$M])=([M]-$Inv$[M])\otimes_\mathbb C(\oplus_ic_i[E_i]-\oplus_ic_i'[E_i'])$.\end{center}
This shows that if we are given the functions $[M:\cdot]_\frak k$ and $[$Inv$M:\cdot]_\frak k$, the coefficients $c_i, c_i'$, and the functions $[E_i:\cdot]_\frak k, [E_i':\cdot]_\frak k$, we can find the functions $[\EuScript H_{\mu_0}^\mu M:\cdot]_\frak k$ and $[\EuScript H_{\mu_0}^\mu$Inv$M:\cdot]_\frak k$ as formal power series.

Moreover, it is enough to determine the coefficients $c_i, c_i'$ and the modules $E_i, E_i'$ for one simple module $M$ of small growth annihilated by I$(\mu_0)$. This work has been done by O.~Mathieu~\cite{M} for a simple module L$_{\mu_0}$, and we now explain the result.

Let $\bar\mu$ be a positive Shale-Weil tuple and be the coefficients $c_i, c_i'$, the modules $E_i, E_i'$ be as in formula~(5). Let $L$ and $L^\sigma$, $L^{\bar\mu}$ and $L^{\sigma\bar\mu}$ be finite-dimensional Spin$_{2n_W}$-modules as in Section~\ref{Stro}. Then\begin{center}$([L^{\bar\mu}:\cdot]_{\frak h_W}+[L^{\sigma\bar\mu}:\cdot]_{\frak h_W})=$\\$(\oplus_ic_i[E_i:\cdot]_{\frak h_W}+\oplus_ic_i'[E_i':\cdot]_{\frak h_W})\otimes([L:\cdot]_{\frak h_W}+[L^\sigma:\cdot]_{\frak h_W})$,\\$([L^{\bar\mu}:\cdot]_{\frak h_W}-[L^{\sigma\bar\mu}:\cdot]_{\frak h_W})=$\\$(\oplus_ic_i[E_i:\cdot]_{\frak h_W}-\oplus_i[E_i':\cdot]_{\frak h_W})\otimes([L:\cdot]_{\frak h_W}-[L^\sigma:\cdot]_{\frak h_W})$.\end{center}
As the functions \begin{center}$[L^{\bar\mu}:\cdot]_{\frak h_W}, [L^{\sigma\bar\mu}:\cdot]_{\frak h_W}, [L:\cdot]_{\frak h_W}, [L^\sigma:\cdot]_{\frak h_W}$\end{center} are computed by the Weyl character formula, we consider the above mentioned formula as an explicit description of the coefficients $c_i, c_i'$ and the $\frak{sp}(W\oplus W^*)$-modules $E_i, E_i'$.

The mystery of this formula is that the function $[L:\cdot]_{\frak h_W}$ comes from the universe of Spin$_{2n_W}$-modules and the functions $[E_i:\cdot]_{\frak h_W}$ come from the universe of SP$_{2n_W}$-modules.
\newpage\section{Appendix}\label{SAP}
\subsection{Results of C. Benson, G. Ratcliff, V. Kac} The classification of spherical modules has been worked out in several steps. V.~Kac has classified simple spherical modules in~\cite{Kc}, C.~Benson and G.~Ratcliff have classified all spherical modules in~\cite{BR}. The classification is contained also in paper~\cite{Le} of A.~Leahy. Below we reproduce their list.

Let $W$ be a $K$-module. Then $W$ is a spherical $\frak k$-variety if and only if the pair $([\frak k, \frak k], W)$ is a direct sum of pairs $(\frak k_i, W_i)$ listed below and in addition $\frak k+\oplus_i c_i=$N$_{\frak{gl}(W)}(\frak k+\oplus_i c_i)$ for certain abelian Lie algebras $c_i$ attached to $(\frak k_i, W_i)$. Here N$_{\frak{gl}(W)}\frak k$ stands for the normalizer of $\frak k$ in $\frak{gl}(W)$ and $c_i$ is a 0-, 1- or 2-dimensional Lie algebra listed in square brackets after the pair $(\frak k_i, W_i)$. This subalgebra is generated by linear operators $h_1$ and $h_{m, n}$. By definition, $h_1=$Id. The notation $h_{m, n}$ is used only when $W=W_1\oplus W_2$: in this case $h_{m, n}|_{W_1}=m\cdot$Id and $h_{m, n}|_{W_2}=n\cdot$Id. The notation '$(\frak k_i, \{W_i, W_i'\})$' is shorthand for '$(\frak k_i, W_i)$' and '$(\frak k_i, W_i')$'. Finally, $\omega_i$ stands for the $i$-th fundamental weight and the corresponding fundamental representation. We follow the enumeration convention for fundamental weights of~\cite{OV}.
\begin{center}{\bfseries Table~\ref{SAP}.1}: Weakly irreducible spherical pairs.\end{center}
0) $(0, \mathbb C) [0]$.\\
\setcounter{AP}{1}\roman{AP}) Irreducible representations of simple Lie algebras:\\
1) $(\frak{sl}_n, \{\omega_1, \omega_{n-1}\}) [\mathbb Ch_1] (n\ge 2)$;\\2) $(\frak{so}_n, \omega_1) [0] (n\ge 3)$;\\3) $(\frak{sp}_{2n}, \omega_1) [\mathbb Ch_1] (n\ge 2)$;\\4) $(\frak{sl}_n, \{2\omega_1, 2\omega_{n-1}\}) [0] (n\ge 3)$;\\5) $(\frak{sl}_{2n+1}, \{\omega_2, \omega_{n-2}\}) [\mathbb Ch_1] (n\ge 2)$;\\6) $(\frak{sl}_{2n}, \{\omega_2, \omega_{n-2}\}) [0] (n\ge 3)$;\\7) $(\frak{so}_7, \omega_3) [0]$;\\8) $(\frak{so}_8, \{\omega_3, \omega_4\}) [0]$;\\9) $(\frak{so}_9, \omega_4) [0]$;\\10) $(\frak{so}_{10}, \{\omega_4, \omega_5\}) [\mathbb Ch_1]$;\\11) $($E$_6, \omega_1) [0]$;\\12) (G$_2, \omega_1 ) [0]$.\\
\setcounter{AP}{2}\roman{AP}) Irreducible representations of nonsimple Lie algebras:\\
1) $(\frak {sl}_n\oplus\frak{sl}_m, \{\omega_1, \omega_{n-1}\}\otimes\{\omega_1, \omega_{m-1}\}) [\mathbb Ch_1](m>n\ge 2)$;\\2) $(\frak {sl}_n\oplus\frak{sl}_n, \{\omega_1, \omega_{n-1}\}\otimes\{\omega_1, \omega_{n-1}\}) [0](n\ge 2)$;\\3) $(\frak{sl}_2\oplus\frak{sp}_{2n}, \omega_1\otimes\omega_1) [0](n\ge 2)$;\\4)
$(\frak{sl}_3\oplus\frak{sp}_{2n}, \{\omega_1,\omega_2\}\otimes\omega_1) [0](n\ge 2)$;\\5) $(\frak{sl}_n\oplus\frak{sp}_4, \{\omega_1, \omega_{n-1}\}\otimes\omega_{1}) [\mathbb Ch_1](n\ge 5)$;\\6) $(\frak{sl}_4\oplus\frak{sp}_4, \{\omega_1, \omega_3\}\otimes\omega_{1}) [0]$.\\
\setcounter{AP}{3}\roman{AP}) Reducible representations of Lie algebras:\\
1) $(\frak{sl}_n\oplus\frak{sl}_m\oplus\frak{sl}_2; (\{\omega_1, \omega_{n-1}\}\oplus\{\omega_1, \omega_{n-1}\})\otimes\omega_1) [\mathbb Ch_{1,0}\oplus\mathbb Ch_{0,1}](n, m\ge 3)$;\\
2) $(\frak{sl}_n; \{\omega_1\oplus\omega_1, \omega_{n-1}\oplus\omega_{n-1}\}) [\mathbb Ch_{1,1}](n\ge 3)$;\\
3) $(\frak{sl}_n; \omega_1\oplus\omega_{n-1}) [\mathbb Ch_{1,-1}](n\ge 3)$;\\
4) $(\frak{sl}_{2n}; \{\omega_1, \omega_{n-1}\}\oplus\{\omega_2, \omega_{n-2}\}) [\mathbb Ch_{0,1}](n\ge 2)$;\\
5) $(\frak{sl}_{2n+1}; \omega_1\oplus\omega_2) [\mathbb Ch_{1,-m}](n\ge 2)$;\\
6) $(\frak{sl}_{2n+1}; \omega_{n-1}\oplus\omega_2) [\mathbb Ch_{1, m}](n\ge 2)$;\\
7) $(\frak{sl}_n\oplus\frak{sl}_m; \{\omega_1, \omega_{n-1}\}\otimes(\mathbb C\oplus\{\omega_1, \omega_{m-1}\})) [\mathbb Ch_{1,0}](2\le n<m)$;\\
8) $(\frak{sl}_n\oplus\frak{sl}_m; \{\omega_1, \omega_{n-1}\}\otimes(\mathbb C\oplus\{\omega_1, \omega_{m-1}\})) [\mathbb Ch_{1, 1}](m\ge 2, n\ge m+2)$;\\
9) $(\frak{sl}_n\oplus\frak{sl}_m; \{\omega_1, \omega_{n-1}\}\oplus\{\omega_1^* (=\omega_{n-1}), \omega_{n-1}^* (=\omega_1)\}\otimes\{\omega_1, \omega_{m-1}\}) [\mathbb Ch_{1,0}](2\le n<m)$;\\
10) $(\frak{sl}_n\oplus\frak{sl}_m; \{\omega_1, \omega_{n-1}\}\oplus\{\omega_1^* (=\omega_{n-1}), \omega_{n-1}^* (=\omega_1)\}\otimes\{\omega_1, \omega_{m-1}\}) [\mathbb Ch_{1,-1}](m\ge 2, n\ge m+2)$;\\
11) $(\frak{sl}_n\oplus\frak{sp}_{2m}\oplus\frak{sl}_2; (\{\omega_1, \omega_{n-1}\}\oplus\omega_1)\otimes\omega_1) [\mathbb Ch_{0,1}](n\ge 3, m\ge 1)$;\\
12) $(\frak{sl}_2; \{\omega_1\oplus\omega_1\}) [0]$;\\
13) $(\frak{sl}_n\oplus\frak{sl}_n; \{\omega_1, \omega_{n-1}\}\oplus\{\omega_1^{(*)}, \omega_{n-1}^{(*)}\}\otimes\{\omega_1, \omega_{n-1}\}) [0](n\ge 2)$;\\
14) $(\frak{sl}_{n+1}\oplus\frak{sl}_n; \{\omega_1, \omega_n\}\oplus\{\omega_1^{(*)}, \omega_n^{(*)}\}\otimes\{\omega_1, \omega_{n-1}\}) [0](n\ge 2)$;\\
15) $(\frak{sl}_2\oplus\frak{sp}_{2n}; \omega_1\otimes(\mathbb C\oplus\omega_1)) [0](n\ge 2)$;\\
16) $(\frak{sp}_{2n}\oplus\frak{sp}_{2m}\oplus\frak{sl}_2; (\omega_1\oplus \omega_1)\otimes\omega_1) [0] (m, n\ge2)$;\\
17) $(\frak{sl}_2\oplus\frak{sl}_2\oplus\frak{sl}_2, (\omega_1\oplus\omega_1)\otimes\omega_1) [0]$;\\
18) $(\frak{so}_8, \{\omega_1\oplus\omega_3, \omega_1\oplus\omega_4, \omega_3\oplus\omega_4\}) [0]$.
\subsection{Results of T.~Braden and M.~Grinberg}Let $n$ be a positive integer and $V$ be a $\mathbb C$-vector space of dimension $2n$. Then the category of perverse sheaves on $\Lambda^2V$ with respect to the stratification by GL$(V)$-orbits is equivalent to the category of representations of the following quiver A with relations:\begin{center}{\Large A$_0$}$_{\xrightarrow{q_0}}^{\xleftarrow{p_0}}${\Large A$_1$}$_{\xrightarrow{q_1}}^{\xleftarrow{p_1}}$...$_{\xrightarrow{q_{n-2}}}^{\xleftarrow{p_{n-2}}}${\Large A$_{n-1}$}$_{\xrightarrow{q_{n-1}}}^{\xleftarrow{p_{n-1}}}${\Large A$_n$},\\ $\xi_i$ and $\nu_i$ are invertible for all $i$, $\xi_i=\nu_i$ for $i\in\{1,..., n-1\}$,\end{center}where $\xi_i:=1+q_{i-1}p_{i-1}$ for $i\in\{1,..., n\}$, and $\nu_i:=1+p_iq_i$ for $i\in\{0,..., n-1\}$. Let $R$ be a representation of the quiver A.  If $R$ is simple, the invertible operators $\{\nu_0^n, \nu_1^{n-1}\xi^1,..., \xi_n^n\}$ are proportional to the identity map with a fixed constant $c\in\mathbb C^*$. We call $c$ {\it the monodromy of $R$}.

The set of eigenvalues of $1+p_iq_i$ is independent from $i$, and we call them {\it eigenvalues of $R$}. If $R$ is simple, this set consists of one element $\lambda$. For a given eigenvalue $\lambda\ne 1$, there exists precisely one simple representation of A with eigenvalue $\lambda$. The simple representations of A with eigenvalue 1 are enumerated by the vertices of the quiver.

By definition, the {\it support of $R$} is the set of vertices corresponding to non-zero vector spaces. The corresponding to $R$ perverse sheaf is supported at 0 if and only if $R$ is supported at A$_0$. The corresponding to $R$ perverse sheaf is smooth along $W$ if and only if $R$ is supported \\at A$_n$.

Let now $V$ be a $\mathbb C$-vector space of dimension $n$. Then the
category of perverse sheaves on S$^2V$ with respect to the
stratification by GL$(V)$-orbits is equivalent to the category of
representations of the following quiver B with
relations:\begin{center}{\Large
B$_0$}$_{\xrightarrow{q_0}}^{\xleftarrow{p_0}}${\Large
B$_1$}$_{\xrightarrow{q_1}}^{\xleftarrow{p_1}}$...$_{\xrightarrow{q_{n-2}}}^{\xleftarrow{p_{n-2}}}${\Large
B$_{n-1}$}$_{\xrightarrow{q_{n-1}}}^{\xleftarrow{p_{n-1}}}${\Large
B$_n$},\\ $\xi_i$ and $\nu_i$ are invertible for all $i$,
$\xi_i^2=\nu_i^2$ for $i\in\{1,..., n-1\}$, $p_j\nu_{j+1}=-\nu_j
p_j,\hspace{10pt} q_j\nu_j=-\nu_{j+1} q_j,\hspace{10pt}
p_j\xi_{j+1}=-\xi_j p_j,\hspace{10pt} q_j\xi_j=-\xi_{j+1} q_j$
whenever both sides of equalities are well-defined. Here
$\xi_i:=1+q_{i-1}p_{i-1}$ for $i\in\{1,..., n\}$, $\nu_i:=1+p_iq_i$
for $i\in\{0,..., n-1\}$.\end{center} Let $R$ be a representation of
the quiver B.  Assume $R$ is simple. Then the invertible operators
$\{\nu_0^n, \nu_1^{n-1}\xi^1,..., \xi_n^n\}$ are proportional to the
identity map with a fixed constant $c\in\mathbb C^*$. We call $c$
{\it the monodromy of $R$}. The set of eigenvalues of $(-1)^i\xi_i$
consists of one element $\bar\xi$; the set of eigenvalues of
$(-1)^i\nu_i$ consists of one element $\bar\nu$.  We call the pair
$(\bar\xi, \bar\nu)$ the {\it spectrum of $R$}. We have
$\bar\xi=\bar\nu=1$ or $\bar\xi=-\bar\nu$. The simple
representations of B with spectrum $(1, 1)$ are enumerated by the
inner vertices of the B-quiver. For $\lambda\ne\pm1$ there exists
precisely one simple representation of B with the spectrum
$(\lambda, -\lambda)$. The simple representations of B with spectra
$(1, -1)$ and $(-1, 1)$ are enumerated by the vertices of the
quiver.

By definition, the {\it support of $R$} is the set of vertices corresponding to non-zero vector spaces. The corresponding to $R$ perverse sheaf is supported at 0 if and only if $R$ is supported at B$_0$. The corresponding to $R$ perverse sheaf is smooth along $W$ if and only if $R$ is supported \\at B$_n$.
\newpage\section{Acknowledgements}
I thank my scientific advisor Ivan Penkov for his attention to my work and the great help with the text-editing. Overall, he spend more then 60 hours of pure time on just text of my dissertation. I thank all participants of '\'Ernest Vinberg's seminar for fruitful discussions around invariant theory, and especially I thank Vladimir Zhgoon. I thank Dmitri Panyushev for interest to my work and useful comments to preliminary versions of texts.

\end{document}